\documentclass{amsart}

\usepackage[foot]{amsaddr}

\usepackage{bbm}	
\usepackage{mathrsfs}	

\usepackage{amsmath, amsthm, amssymb}	
\usepackage{enumitem}	

\usepackage{tikz}	
\usetikzlibrary{arrows,shapes,trees,calc}

\usepackage[linktocpage,hidelinks]{hyperref}
\usepackage{mathtools}
\usepackage[extra]{tipa}

\setlist[enumerate]{label=(\roman*)}

\newtheorem{theorem}{Theorem}[section]
\newtheorem{lemma}[theorem]{Lemma}
\newtheorem{corollary}[theorem]{Corollary}

\theoremstyle{definition}
\newtheorem{definition}[theorem]{Definition}
\newtheorem{example}[theorem]{Example}

\theoremstyle{remark}
\newtheorem{remark}[theorem]{Remark}

\numberwithin{equation}{section}

\newcommand{\R}{\mathbb{R}}	
\newcommand{\N}{\mathbb{N}}	
\newcommand{\Q}{\mathbb{Q}}	

\newcommand{\sB}{\mathscr{B}}	
\newcommand{\sE}{\mathscr{E}}	

\newcommand{\cC}{\mathcal{C}}	

\newcommand{\BB}{B}		

\renewcommand{\O}{\Omega}	
\renewcommand{\o}{\omega}	
\newcommand{\PV}{\mathbb{P}}	
\newcommand{\EV}{\mathbb{E}}	

\newcommand{\sF}{\mathscr{F}}	
\newcommand{\sG}{\mathscr{G}}	
\newcommand{\T}{\Theta}		
\newcommand{\D}{\Delta}		
\newcommand{\z}{\zeta}		

\renewcommand{\a}{\alpha}	

\renewcommand{\t}{\tau}		
\renewcommand{\r}{\rho}		
\newcommand{\h}{H}		

\newcommand{\sS}{\mathscr{S}}	
\newcommand{\sD}{\mathscr{D}}	

\newcommand{\cG}{\mathcal{G}}	
\newcommand{\cV}{\mathcal{V}}	
\newcommand{\cI}{\mathcal{I}}	
\newcommand{\cE}{\mathcal{E}}	
\newcommand{\cL}{\mathcal{L}}	
\newcommand{\cLV}{\partial}	
\newcommand{\cR}{\rho}		

\newcommand{\dint}{d^{\operatorname{int}}} 
\newcommand{\dPV}{d^{\cP,\cR}}

\newcommand{\tcG}{\widetilde{\cG}}
\newcommand{\tcV}{\widetilde{\cV}}
\newcommand{\tcI}{\widetilde{\cI}}
\newcommand{\tcE}{\widetilde{\cE}}

\newcommand{\tcLV}{\widetilde{\cLV}}
\newcommand{\tcR}{\widetilde{\cR}}

\newcommand{\cP}{\mathcal{P}}   

\newcommand{\bs}{\backslash}	
\newcommand{\comp}{\complement}	

\renewcommand{\b}{\beta}

\newcommand{\e}{\varepsilon}	
\renewcommand{\d}{\delta}	        

\newcommand{\tT}{\widetilde{T}}

\newcommand{\tX}{\widetilde{X}{}}

\newcommand{\1}{\mathbbm{1}}		

\newcommand{\abs}[1]{\left\lvert #1 \right\rvert}  
\newcommand{\bigabs}[1]{\big\lvert #1 \big\rvert}  

\newcommand{\norm}[1]{\left\lVert #1 \right\rVert} 

\newcommand\restr[2]{{		  	
  \left.\kern-\nulldelimiterspace 
  #1 				  
  \vphantom{\big|}		  
  \right|_{#2}    		  
  }}

\newcommand{\overbar}[1]{\mkern 1.5mu\overline{\mkern-1.5mu#1\mkern-1.5mu}\mkern 1.5mu}

\newcommand{\processX}{$X = \big( \O, \sG, (\sG_t)_{t \geq 0}, (X_t)_{t \geq 0}, (\T_t)_{t \geq 0}, (\PV_x)_{x \in E} \big)$}

\newcommand{\KPS}{Kostrykin, Potthoff and Schrader}
\newcommand{\IM}{It\^{o}\textendash{}McKean}

\newcommand{\FW}{Feller\textendash{}Wentzell}

\newcommand{\bigslant}[2]{{\raisebox{.2em}{$#1$}\big/\raisebox{-.2em}{$#2$}}}

\newcommand{\textdef}{\textit}

\begin{document}

\title[Brownian Motions on Metric Graphs I: Characterization]{Brownian Motions on Metric Graphs with Non-Local Boundary Conditions I: Characterization}

\author[Florian Werner]{Florian Werner}
\address{Institut f\"ur Mathematik, Universit\"at Mannheim, 68131 Mannheim, Germany}
\email{fwerner@math.uni-mannheim.de}



\date{\today}



\newcommand{\sqed}{\relax}

\hyphenation{Pott-hoff}

\begin{abstract}
  A classification for Brownian motions on metric graphs, that is,
right continuous strong Markov processes which behave like a one-dimensional Brownian motion on the edges and feature effects like 
Walsh skewness, stickiness and jumps at the vertices, is obtained.
The Feller property of these processes is proved, and the boundary conditions of their generators are identified as non-local \FW\ boundary conditions.
By using a technique of successive revivals, a complete description of the generator is achieved for Brownian motions on star graphs.
\end{abstract}

\maketitle

\section{Introduction and Main Results}
This article is the second part in a series of works in which we achieve a classification and pathwise construction of Brownian motions on metric graphs.
The interested reader may find a short survey covering the developments and applications of this class of stochastic processes 
at the beginning of the first part~\cite{WernerStar}.
In that article, we obtained a complete pathwise construction of every possible Brownian motion on special graphs having only one vertex and a finite set of edges without loops, so called star graphs.
In the present article and its continuation, we extend our findings to general metric graphs with finite sets of vertices and edges
(a concise overview of metric graphs can be found in~Appendix~\ref{app:metrix graphs}).

Following \IM~\cite{ItoMcKean63}, we understand a Brownian motion on a metric graph~$\cG$ as a right continuous, strong Markov process on $\cG$ 
which behaves on every edge like the standard one-dimensional Brownian motion (see section~\ref{sec:G_BM:def} for a rigorous definition).
Thus, we will extend the results of \KPS~\cite{KPS12} to Brownian motions admitting discontinuities at the vertex points.
This will generalize the classical local boundary condition at any vertex, as given in their work, to non-local \FW\ boundary conditions,
as given in equation~\eqref{eq:G_BM:generator boundary conditions}.

The admittance of non-local effects introduces several difficulties: 
As the process may jump from any vertex to any point of the graph, an analysis of the complete form of the semigroup or resolvent becomes unfeasible.
We will solve this by utilizing localization techniques like Dynkin's formulas (cf.~Appendix~\ref{app:Feller processes}) and local constructions (in part II),
which however, due to the possibility of jumps to distant points, are only practicable to certain extends.
Furthermore, we will only achieve an incomplete description of the generator in the case of a general metric graph.
 Thus, we need to carefully trace characteristic components of the processes (cf.~Theorem~\ref{theo:G_BM:Feller data})
 and examine the effects of path transformations on these components (see., e.g., Lemma~\ref{lem:G_IM:Feller data of revived BB on star graph}, and part II).

In the present article, we will give a characterization of Brownian motions on metric graphs by identifying the boundary conditions of their generators. 
Our main result (as shown in sections~\ref{sec:G_BM:def} and~\ref{sec:G_BM:Fellers theorem}), which we will call ``Feller's theorem'', is as follows:

\begin{theorem} \label{theo:BM characterization}
 Let $X$ be a Brownian motion on $\cG$.
 Then $X$ is a Feller process, uniquely determined by its $\cC_0$-generator $A = \frac{1}{2} \D$, with $\sD(A) \subseteq \cC^2_0(\cG)$.
 For every vertex $v \in \cV$ there exist constants
 $p^v_1 \geq 0$, $p^{v,l}_2 \geq 0$ for each $l \in \cL(v)$, $p^v_3 \geq 0$ and a measure $p^v_4$ on $\cG \bs \{v\}$ satisfying
  \begin{align*} 
    p^v_1 + \sum_{l \in \cL(v)} p^{v,l}_2 + p^v_3 + \int \big( 1 - e^{-d(v,g)} \big) \, p^v_4 (dg) = 1,
  \end{align*}
 and
  \begin{align} \label{eq:G_BM:infinite jumpmeasure}
    p^v_4 \big( \cG \bs \{v\} \big) = +\infty, \quad \text{if} \quad \sum_{l \in \cL(v)} p^{v,l}_2 + p^v_3 = 0,
  \end{align}
 such that the domain of $A$ reads
  \begin{align}
   \sD(A) & \subseteq \notag
       \Big\{ f \in \cC^2_0(\cG) : \forall v \in \cV: \\ 
   & \qquad           p^v_1 f(v) - \sum_{l \in \cL(v)} p^{v,l}_2 f_l'(v) + \frac{p^v_3}{2} f''(v) - \int \big( f(g) - f(v) \big) \, p^v_4(dg) = 0 \Big\}.  \label{eq:G_BM:generator boundary conditions}
  \end{align}
\end{theorem}

In part~II, we will give a pathwise construction for every possible set of boundary conditions on a given metric graph:
Having already achieved the pathwise construction of Brownian motions on metric graphs with only one vertex in~\cite{WernerStar},
we will use these local solutions and piece them together to a global solution on the complete metric graph.
The construction will involve several non-trivial process transformations, such as killing, revival and state space mappings.
In order to be able to verify the correctness of the resulting boundary conditions, we need to have access to the boundary values of the (partial) processes
via their local behavior. We will mainly work with the following representation, which we obtain in section~\ref{sec:G_BM:Fellers theorem}:

\begin{theorem} \label{theo:G_BM:Feller data}
 Let $X$ be a Brownian motion on $\cG$ with generator $A$.
 Then, for every $v \in \cV$, there exist constants
 $c^v_1 \geq 0$, $c^{v,l}_2 \geq 0$ for each $l \in \cL(v)$, $c^v_3 \geq 0$ and a measure $c^v_4$ on $\cG \bs \{v\}$, satisfying
  \begin{align*}
    c^v_1 + \sum_{l \in \cL(v)} c^{v,l}_2 + c^v_3 + \int_{\cG \bs \{v\}} \big( 1 - e^{-d(v,g)} \big) \, c^v_4 (dg) = 1,
  \end{align*}
 such that for every $f \in \sD(A)$, the following relation holds:
  \begin{align*}
    c^v_1 \, f(v) - \sum_{l \in \cL(v)} c^{v,l}_2 \, f_l'(v) + c^v_3 \, A f(v) - \int_{\cG \bs \{v\}} \big( f(g) - f(v) \big) \, c^v_4(dg) = 0.
  \end{align*}
 The constants and the measure only depend on the process' exit behavior from any arbitrarily small neighborhood of $v$.
 They are given by
  \begin{align*}
   c^v_1 & = c^{v, \D}_1 + c^{v, \infty}_1, \\
   \text{with} & \quad
   c^{v, \D}_1 = \lim_{n \rightarrow \infty} \frac{\PV_v ( X_{\t_{\e_n}} = \D )}{\EV_v(\t_{\e_n}) \, K^v_{\e_n}}, \quad
   c^{v, \infty}_1 = \sum_{e \in \cE} \overbar{\mu}^v \big( \{ (e, +\infty) \} \big), \\
   c^{v,l}_2 & = 
    \begin{cases}
     \overbar{\mu}^v \big( \{ (l, 0+) \} \big), & l \in \cE(v), \\
     \overbar{\mu}^v \big( \{ (l, 0+) \} \big), & l \in \cI(v), v = \cLV_-(l), \\
     \overbar{\mu}^v \big( \{ (l, \r_l-) \} \big), & l \in \cI(v), v = \cLV_+(l),
    \end{cases} \\
   c^v_3 & = \lim_{n \rightarrow \infty} \frac{1}{K^v_{\e_n}}, \\
   c^v_4 (dg) & = \frac{1}{1 - e^{-d(v,g)}} \, \overbar{\mu}^v(dg),
  \end{align*}
 where for every $\e > 0$, $\t_\e = \inf \big\{ t \geq 0: X_t \in \comp \overline{\BB_\e(v)} \big\}$,
  \begin{align*}
    K^v_{\e} & = 1 + \frac{\PV_v ( X_{\t_{\e}} = \D )}{\EV_v(\t_\e)} + \int_{\cG \bs \{v\}} \big( 1 - e^{-d(v,g)} \big) \, \nu^v_{\e} (dg),
  \end{align*}
  $\nu^v_\e$ and $\mu^v_\e$ are measures on $\cG \bs \{v\}$ defined by 
  \begin{align*}
    \nu^v_\e (dg) & = \frac{\PV_v ( X_{\t_\e} \in dg )}{\EV_v(\t_\e)}, \\
    \mu^v_\e (dg) & = \big( 1 - e^{-d(v,g)} \big) \, \frac{\nu^v_\e (dg)}{K^v_\e},
  \end{align*}
  as well as $\overbar{\mu}^v_\e$, $\overbar{\mu}^v$ are measures on $\overbar{\cG \bs \{v\}}$ with
  \begin{align*}
    \overbar{\mu}^v_\e (dg) & = \mu^v_\e \big( dg \cap \big(\cG \bs \{v\}\big) \big), \\
    \overbar{\mu}^v & = \lim_{n \rightarrow \infty} \overbar{\mu}^v_{\e_n} \quad \text{(as weak limit)},
  \end{align*} 
 and $(\e_n, n \in \N)$ is a sequence of positive numbers converging to zero such that all of the above limits exist.
\end{theorem}

Theorem~\ref{theo:G_BM:Feller data} gives explicit (albeit rather unwieldy) expressions for the boundary conditions of a Brownian motion.
As we will utilize them quite frequently, we assign the following, supposably appropriate name:

\begin{definition} \label{def:G_BM:Feller data}
 For any Brownian motion $X$ on a metric graph $\cG$, the collection
  \begin{align*}
    \big( c^{v,\D}_1, c^{v, \infty}_1, (c^{v,l}_2)_{l \in \cL(v)}, c^v_3, c^v_4 \big)_{v \in \cV}
  \end{align*}
 as defined in Theorem~\ref{theo:G_BM:Feller data} is called \textdef{\FW\ data} of $X$.  
 If no distinction is necessary, $c^{v,\D}_1$ and $c^{v, \infty}_1$ are combined, denoted by $c^{v}_1 = c^{v,\D}_1 + c^{v, \infty}_1$.
\end{definition}

The effects of the parameters $( c^{v,\D}_1, c^{v, \infty}_1, (c^{v,l}_2)_{l \in \cL(v)}, c^v_3, c^v_4 \big)_{v \in \cV}$
on the paths of the Brownian motions are well-understood and have 
already been explained in the introduction of~\cite{WernerStar}. Namely, they govern the ratio (and direction) of
killing, reflection, stickiness and jumps at each vertex $v \in \cV$.

It will turn out that the killing ratio $c_1^{v, \infty}$ (which is caused by infinitely large jumps from $v$ in an arbitrarily small time interval)  
is rather ``artificial'' in the context of Brownian motions, 
and will become a source of problems in our upcoming constructions of part~II.
Therefore, we will show that $c_1^{v, \infty}$ indeed vanishes for Brownian motions on star graphs (and thus, in all of our constructions of part~II):

\begin{theorem} \label{theo:G_IM:Feller data c_infty vanishes}
 Let $X$ be a Brownian motion on a star graph $\cG$ with star vertex $v$. Then $c^{v,\infty}_1 = 0$ holds true
 in Theorem~\ref{theo:G_BM:Feller data}.
\end{theorem}

In general, we are not able to determine a complete description of the generator domain, that is,
it does not appear to be easy to show that the boundary conditions, as given in Theorem~\ref{theo:BM characterization},
are also sufficient for the generator domain.
If the Brownian motion is assumed to be continuous up to its lifetime, \KPS\ showed in~\cite[Section~3]{KPS12}
that equality in~\eqref{eq:G_BM:generator boundary conditions} is attained. 
In~\cite[Lemma~2.6]{WernerStar}, we were able to prove sufficiency in the discontinuous setting for metric graphs with only one vertex:

\begin{theorem}  \label{theo:G_BM:Fellers theorem on star graph}
 Let $X$ be a Brownian motion on star graph $\cG$ with star point $v$.  Then $X$ is a Feller process with generator $A = \frac{1}{2} \D$,
 and there exist constants
 $p_1 \geq 0$, $p^{e}_2 \geq 0$ for each $e \in \cE$, $p_3 \geq 0$ and a measure $p_4$ on $\cG \bs \{v\}$ with
  \begin{align*}
    p_1 + \sum_{e \in \cE} p^{e}_2 + p_3 + \int_{\cG \bs \{v\}} \big( 1 - e^{-d(v,g)} \big) \, p^v_4 \big( dg \big) = 1,
  \end{align*}
 and
  \begin{align*}
    p_4 \big( \cG \bs \{v\} \big) = +\infty, \quad \text{if} \quad  \sum_{e \in \cE} p^{e}_2 + p_3 = 0,
  \end{align*}
 such that the domain of $A$ reads
  \begin{align*}
   \sD(A)  = 
       \Big\{ & f \in \cC^2_0(\cG): \\ 
              &       p_1 f(v) - \sum_{e \in \cE} p^{e}_2 f_e'(v) + \frac{p_3}{2} f''(v) - \int_{\cG \bs \{v\}} \big( f(g) - f(v) \big) \, p_4(dg) = 0 \Big\}. 
  \end{align*}
 Furthermore, $X$ is uniquely characterized by this set of normalized constants.
\end{theorem}

For the interval case, the corresponding result is shown in~\cite[Section~17]{Werner16},
demonstrating the technical difficulties that already arise for metric graphs with only two vertex points.
We cite it here for completeness:

\begin{theorem} \label{theo:B_IN:Feller data}
 Let $X$ be a Brownian motion on $[a,b]$.
 Then $X$ is a Feller process with generator $A = \frac{1}{2} \D$.
 There exist constants
 $p^a_1 \geq 0$, $p^a_2 \geq 0$, $p^a_3 \geq 0$ and a measure $p^a_4$ on $(a, b]$
 as well as
 $p^b_1 \geq 0$, $p^b_2 \geq 0$, $p^b_3 \geq 0$ and a measure $p^b_4$ on $[a, b)$,
 satisfying
  \begin{align*}
    p^a_1 + p^a_2 + p^a_3 + \int_{(a,b]} \big( 1 \wedge x \big) \, p^a_4 (dx) & = 1, \\
    p^b_1 + p^b_2 + p^b_3 + \int_{[a,b)} \big( 1 \wedge x \big) \, p^b_4 (dx) & = 1,
  \end{align*}
 and
   \begin{align*}
    p^a_4 \big( (a,b] \big) = +\infty, & \quad \text{if} \quad p^a_2 = p^a_3 = 0, \\
    p^b_4 \big( [a,b) \big) = +\infty, & \quad \text{if} \quad p^b_2 = p^b_3 = 0,
  \end{align*}
 such that the domain of the generator of $X$ reads
  \begin{align*}
    \sD(A) = \Big\{ & f \in \cC_0^2(\R_+):  \\ 
                    & p^a_1 \, f(a) - p^a_2 \, f'(a+) + \frac{p^a_3}{2} \, f''(a+) - \int_{(a,b]} \big( f(x) - f(a) \big) \, p^a_4(dx) = 0, \\
                    & p^b_1 \, f(b) + p^b_2 \, f'(b-) + \frac{p^b_3}{2} \, f''(b-) - \int_{[a,b)} \big( f(x) - f(b) \big) \, p^b_4(dx) = 0 ~ \Big\}.
  \end{align*}
\end{theorem}

For Brownian motions on general metric graphs, the complete description of the generator domain remains unsolved.
\section{Definition and Fundamental Properties} \label{sec:G_BM:def}
As already explained in the introduction, it is suitable to characterize Brownian motions on metric graphs 
by their generators, which is the objective of this section. 

After giving the rigorous definition of a ``Brownian motion on a metric graph'', we collect some basic properties of such a process
by utilizing its locally ``one-dimensional Brownian behavior'' on the edges and applying some classical results for the half-line and interval cases.  
We are then able to analyze the resolvents---yielding their Feller property---and the generators of Brownian motions on metric graphs,
giving explicit formulas for the computation of their ``\FW'' boundary conditions. 
For a short summary of results on Markov and Feller processes, which are used throughout this section, the reader may refer to Appendix~\ref{app:Feller processes}.

As announced in the introduction, a Brownian motion on a metric graph~$\cG$ is a right continuous, strong Markov process on $\cG$ 
with a local one-dimensional Brownian behavior. More precisely, the local coordinate of such a process, 
if stopped once the process leaves its starting edge,
needs to be equivalent to the Brownian motion on $\R$, stopped when leaving the corresponding interval of the process' initial edge.
Extending the definition of~\cite{KPS12} and~\cite[Chapter~6]{Knight81} to the discontinuous setting of~\cite{ItoMcKean63}, we set:

\begin{definition} \label{def:G_BM:BM}
 Let \processX\ be a right continuous, strong Markov process on a metric graph $\cG$. 
 $X$ is a \textdef{Brownian motion on}~$\cG$, if
 for all $g = (l,x) \in \cG$, the random time 
  \begin{align*}
   H_X := \inf \big\{ t \geq 0: X_t \notin l^0 \big\}, \quad 
   \text{with } l^0 = \{l\} \times ( 0, \cR_l ),
  \end{align*}
 is a stopping time over $(\sG_t, t \geq 0)$,
 and 
  \begin{align*}
   \EV_{(l,x)} \big( f_1(X_{t_1 \wedge H_X}) \cdots f_n(X_{t_n \wedge H_X}) \big)
   = \EV^B_{x} \big( f_1(l, B_{t_1 \wedge H_B}) \cdots f_n(l, B_{t_n \wedge H_B}) \big)
  \end{align*}
 holds for all $n \in \N$, $f_1, \ldots, f_n \in b\sB(\cG)$, $t_1, \ldots, t_n \in \R_+$,
 with $B$ being the Brownian motion on $\R$ and $H_B := \inf \big\{ t \geq 0: B_t \notin ( 0, \cR_l ) \big\}$. 
\end{definition}

As we are dealing with potentially discontinuous processes, we needed to pay special attention to the measurability of the 
debut~$H_X$ of the closed set~$\comp l^0$ in the above definition. 
This technical requirement on~$H_X$ is always fulfilled in the following two common cases: 
If the Brownian motion on~$\cG$ is known to be constructed
with the help of continuous excursions of a ``standard'' one-dimensional Brownian motion 
and thus features continuity while running inside any edge (cf.~\cite{WernerStar} and~part~II),
that is, continuity until~$H_X$, \cite[Theorem~49.5]{Bauer96} ensures the stopping time property
of~$H_X$.
Otherwise, the measurability of~$H_X$ can always achieved by working in the context of usual hypotheses (cf.~\cite[Sections~10, A.5]{Sharpe88}).

We first need to collect some basic properties of Brownian motions on metric graphs.
Most of them are implicitly used without proof in earlier works, such as in \cite{ItoMcKean63}, \cite{Knight81}, or \cite{KPS12}, 
and may be attained quite easily in the continuous setting.
However, a little bit more care is needed for discontinuous Brownian motions.

For all that follows, let $X$ be a Brownian motion on a metric graph $\cG$, 
$H_X$ be the first exit time from $l^0 = \{l\} \times (0, \cR_l)$ for a given initial point $g = (l,x) \in \cG$, 
as well as $B$ be the one-dimensional Brownian motion with the first exit time~$H_B$ from the corresponding edge interval $(0, \cR_l)$, 
as specified in Definition~\ref{def:G_BM:BM}. 
As usual, we identify any edge $l \in \cL$ with its geometric representation $\{l\} \times [0, \cR_l]$,
where we set $[0, \cR_l] := [0, +\infty)$ if $\cR_l = +\infty$.

We start with some basic results on $H_X$. The first property follows directly from the right-continuity of~$X$ (and of~$B$):

\begin{lemma} \label{lem:G_BM:exit time dist, basic}
 For all $t \geq 0$, 
  \begin{align*}
    \{ H_X \leq t \} = \{ X_{t \wedge H_X} \in \comp l^0 \}
    \quad \text{and} \quad
    \{ H_B \leq t \} = \{ B_{t \wedge H_B} \in \comp (0, \cR_l) \}.
  \end{align*}
\end{lemma}
%

\begin{corollary} \label{cor:G_BM:exit time dist, basic}
 For all $g = (l,x) \in \cG$,
  \begin{align*}
   \PV_{(l,x)} \circ H_X^{-1} = \PV^B_{x} \circ H_B^{-1},
  \end{align*}
 in particular, we have
  \begin{align*}
   \PV_{(l,x)} (H_X < +\infty) = \PV^B_{x}(H_B < +\infty) = 1.
  \end{align*}
\end{corollary}

These results will be considerately improved in Theorem~\ref{theo:G_BM:exit distributions} below.
For the time being, they are sufficient to deduce a slightly more general property of the distributions of the stopped Brownian motion:

\begin{lemma} \label{lem:G_BM:distributions stopped process}
 For $g = (l,x) \in \cG$, $n \in \N$, $f_1, \ldots, f_n, h \in b\sB(\cG)$, $0 \leq t_1 \leq \cdots \leq t_n$,
  \begin{align*}
   & \EV_{(l,x)} \big( f_1(X_{t_1}) \cdots f_n(X_{t_n}) \, h(X_{H_X}) ; t_n < H_X \big) \\
   & = \EV^B_{x} \big( f_1(l, B_{t_1}) \cdots f_n(l, B_{t_n}) \, h(l, B_{H_B}) ; t_n < H_B \big).
  \end{align*}
\end{lemma}
\begin{proof}
 Observe that, because $H_X < +\infty$ a.s.\ (by Corollary~\ref{cor:G_BM:exit time dist, basic}) and $X_{s \wedge H_X} = X_{H_X}$ holds for all $s \geq H_X$,
 we have
  \begin{align*}
    \lim_{s \rightarrow \infty} h(X_{s \wedge H_X}) = h(X_{H_X}) \quad \text{a.s.},
  \end{align*}
 and analogously,
  \begin{align*}
    \lim_{s \rightarrow \infty} h(l, B_{s \wedge H_B}) = h(l, B_{H_B}) \quad \text{a.s.}\,.
  \end{align*}
 Thus, by using Lebesgue's dominated convergence theorem and the defining properties of a Brownian motion on a metric graph, we conclude that
  \begin{align*}
   & \EV_{(l,x)} \big( f_1(X_{t_1}) \cdots f_n(X_{t_n}) \, h(X_{H_X}) ; t_n < H_X \big) \\
   & = \lim_{s \rightarrow \infty} \EV_{(l,x)} \big( f_1(X_{t_1 \wedge H_X}) \cdots f_n(X_{t_n \wedge H_X}) \, h(X_{s \wedge H_X}) \, \1_{l^0}(X_{t \wedge H_X}) \big) \\
   & = \lim_{s \rightarrow \infty} \EV^B_{x} \big( f_1(l, B_{t_1 \wedge H_B}) \cdots f_n(B_{t_n \wedge H_B}) \, h(l, B_{s \wedge H_B}) \, \1_{l^0}(l, B_{t \wedge H_B}) \big) \\
   & = \EV^B_{x} \big( f_1(l, B_{t_1}) \cdots f_n(l, B_{t_n}) \, h(l, B_{H_B}) ; t_n < H_B \big). \qedhere
  \end{align*}
\end{proof}

The above lemma allows us to achieve an equivalent set of defining properties for Brownian motions on metric graphs. They will turn out to be more suitable
for our work, as they are based on the (local) resolvent and the exit behavior of the process rather than on its stopped distributions:

\begin{theorem} \label{theo:G_BM:resolvent characterization of BM on MG}
 Let $X$ be a right continuous, strong Markov process on $\cG$. 
 $X$ is a Brownian motion on $\cG$, if and only if
 for all $g = (l,x) \in \cG$, 
 the following assertions hold:
  \begin{enumerate}
   \item for all $\a > 0$, $f \in b\sB(\cG)$,  \label{itm:G_BM:resolvent characterization of BM on MG, i}
     \begin{align*}
      \EV_{(l,x)} \Big( \int_0^{H_X} e^{-\a t} \, f(X_t) \, dt \Big)
      = \EV^B_{x} \Big( \int_0^{H_B} e^{-\a t} \, f(l, B_t) \, dt \Big),
     \end{align*}
   \item $\displaystyle \PV_{(l,x)} \circ \big( H_X, X_{H_X} \big)^{-1} = \PV^B_{x} \circ \big( H_B, (l, B_{H_B}) \big)^{-1}$. \label{itm:G_BM:resolvent characterization of BM on MG, ii}
  \end{enumerate}
\end{theorem}
\begin{proof}
 Necessity follows directly from Lemma~\ref{lem:G_BM:distributions stopped process}. 
 
 Now let \ref{itm:G_BM:resolvent characterization of BM on MG, i} and \ref{itm:G_BM:resolvent characterization of BM on MG, ii} hold true.
 As $X$ and $B$ are right continuous, strong Markov processes and $H_X$, $H_B$ are debuts of closed sets,
 the stopped processes $X_{\,\cdot\, \wedge H_X}$, $B_{\,\cdot\, \wedge H_B}$ are indeed right continuous, strong Markov processes (cf.~\cite[Theorem~10.2]{Dynkin65}).
 Let $(\tT_t, t \geq 0)$ and $(\tT^B_t, t \geq 0)$ be their respective semigroups, that is, consider for $f \in b\sB(\cG)$ and $f_l := f(l, \cdot) \in b\sB([0, \cR_l])$:
  \begin{align*}
    \tT_t f(l,x) & = \EV_{(l,x)} \big( f(X_{t \wedge H_X}) \big), \quad
    \tT^B_t f_l(x) = \EV^B_{x} \big( f_l(B_{t \wedge H_X}) \big).
  \end{align*}
 As the stopped process $X_{\,\cdot\, \wedge H_X}$ is strongly Markovian, Dynkin's formula~\eqref{eq:Dynkins formula (resolvent)} 
 for the decomposition of the resolvent at $H_X$ gives for all $\a > 0$:
  \begin{align*}
   \int_0^\infty e^{-\a t} \, \tT_t f(l,x) \, dt
   & = \EV_{(l,x)} \Big( \int_0^{H_X} e^{-\a t} \, f(X_t) \, dt \Big) \\
   & \quad + \EV_{(l,x)} \Big( e^{-\a H_X} \, \EV_{X_{H_X}} \Big( \int_0^{\infty} e^{-\a t} \, f(X_{t \wedge H_X}) \, dt \Big) \Big).
  \end{align*}
 With $X_{H_X} \in \comp l^0$, we have $H_X = 0$ $\PV_{X_{H_X}}$-a.s., thus the above decomposition becomes
  \begin{align*}
   \int_0^\infty e^{-\a t} \, \tT_t f(l,x) \, dt
   & = \EV_{(l,x)} \Big( \int_0^{H_X} e^{-\a t} \, f(X_t) \, dt \Big)
     + \frac{1}{\a} \, \EV_{(l,x)} \Big( e^{-\a H_X} \, f(X_{H_X}) \Big).
  \end{align*}
 Analogously, we get by decomposing the resolvent of $B_{\,\cdot\, \wedge H_B}$ at $H_B$:
  \begin{align*}
   \int_0^\infty e^{-\a t} \, \tT^B_t f_l(x) \, dt
   & = \EV^B_{x} \Big( \int_0^{H_B} e^{-\a t} \, f(l,B_t) \, dt \Big)
     + \frac{1}{\a} \, \EV^B_{x} \Big( e^{-\a H_B} \, f(l, B_{H_B}) \Big).
  \end{align*}
 Using \ref{itm:G_BM:resolvent characterization of BM on MG, i} and \ref{itm:G_BM:resolvent characterization of BM on MG, ii} immediately yields
   \begin{align*}
   \int_0^\infty e^{-\a t} \, \tT_t f(l,x) \, dt
   & = \int_0^\infty e^{-\a t} \, \tT^B_t f_l(x) \, dt,
  \end{align*}
 holding true for all $\a > 0$ and all $f \in b\cC(\cG)$, $(l,x) \in \cG$. The maps $t \mapsto \tT_t f(l,x)$ and $t \mapsto \tT^B_t f_l(x)$ are right continuous, 
 so the uniqueness theorem for Laplace transforms (cf.~\cite[Lemma 1.1]{Dynkin65}) asserts that
  \begin{align*}
    \forall t \geq 0: \quad \tT_t f(l,x) = \tT^B_t f_l(x).
  \end{align*}
 As $X_{\,\cdot\, \wedge H_X}$, $B_{\,\cdot\, \wedge H_B}$ are Markov processes with the ``same'' semigroup, we are able to show inductively that
 for all $(l,x) \in \cG$, $f_1, \ldots, f_n \in b\cC(\cG)$, $0 \leq t_1 \leq \cdots \leq t_n$,
  \begin{align*}
   & \EV_{(l,x)} \big( f_1(X_{t_1 \wedge H_X}) \, \cdots f_n(X_{t_n \wedge H_X}) \big) \\
   & = \EV_{(l,x)} \big( f_1(X_{t_1 \wedge H_X}) \, \cdots f_{n-1}(X_{t_{n-1} \wedge H_X}) \, \EV_{X_{t_{n-1} \wedge H_X}} \big( f_n(X_{(t_n - t_{n-1}) \wedge H_X}) \big) \big) \\
   & = \EV_{(l,x)} \big( f_1(X_{t_1 \wedge H_X}) \, \cdots f_{n-1}(X_{t_{n-1} \wedge H_X}) \, \tT_{t_n - t_{n-1}} f_n(X_{t_{n-1} \wedge H_X}) \big) \\
   & = \EV^B_{x} \big( f_1(l, B_{t_1 \wedge H_B}) \, \cdots f_{n-1}(l, B_{t_{n-1} \wedge H_B}) \, \tT^B_{t_n - t_{n-1}} \big(f_n(l, \,\cdot\,)\big) (B_{t_{n-1} \wedge H_B}) \big) \\
   & = \EV^B_{x} \big( f_1(l, B_{t_1 \wedge H_B}) \, \cdots f_n(l, B_{t_n \wedge H_B}) \big),
  \end{align*}
 which is easily extended to $f_1, \ldots, f_n \in b\sB(\cG)$ by using the monotone class theorem.
\end{proof}

With the help of the above theorem, we can further refine the properties of the first exit time $H_X$.
Indeed, despite of its potential discontinuities, the Brownian motion can only exit its initial edge by hitting vertices incident with it:

\begin{corollary} \label{cor:G_BM:exit time is hitting time of vertex}
 For all $g = (l,x) \in \cG$,
  \begin{align*}
   H_X = H_{\cLV(l)} \quad \text{$\PV_{(l,x)}$-a.s.}\,.
  \end{align*}
\end{corollary}
\begin{proof}
 As $\cLV(l) \subseteq \comp l^0 $, we always have $H_X = H_{\comp l^0} \leq H_{\cLV(l)}$. 
 Theorem~\ref{theo:G_BM:resolvent characterization of BM on MG}~\ref{itm:G_BM:resolvent characterization of BM on MG, ii} gives
  \begin{align*}
   \PV_{(l,x)} \big( X_{H_X} \in l \big) 
   = \PV^B_x \big( B_{H_B} \in [0, \cR_l] \big)
   = 1.
  \end{align*}
 On the other hand, $X_{H_X} \in \comp l^0$ holds, as $\comp l^0$ is closed and $X$ is right continuous.
 So we conclude that $X_{H_X} \in l \cap \comp l^0 = \cLV(l)$ a.s., which results in $H_{\cLV(l)} \leq H_X$ a.s.\,.
\end{proof}

It immediately follows that 
 \begin{align*}
  \forall t \geq 0: \quad X_{t \wedge H_X} \in l \quad \text{$\PV_{(l,x)}$-a.s.},
 \end{align*}
because if we assume the contrary, that is $X_{t \wedge H_X} \in \comp l \subseteq \comp l^0$, 
then $H_X \leq t \wedge H_X$ and so $X_{H_X} = X_{t \wedge H_X} \notin l$,
contradicting to $X_{H_X} = X_{H_{\cLV(l)}} \in \cLV(l) \subseteq l$.

This seemingly small result implies that any Brownian motion, stopped on leaving the open interior of its starting edge,
remains on this edge (especially at the exit time):

\begin{theorem} \label{theo:G_BM:stopped BM is on starting edge}
 For all $g = (l,x) \in \cG$,
  \begin{align*}
   \PV_{(l,x)} \big( \forall t \geq 0: X_{t \wedge H_X} \in l \big) = 1.
  \end{align*}
\end{theorem}
\begin{proof}
 Assume the contrary, that is,
  \begin{align*}
   \PV_{(l,x)} \big( \exists t \geq 0: X_{t \wedge H_X} \notin l  \big) > 0.
  \end{align*}
 Consider the optional set
  \begin{align*}
    A := \big\{ (t,x) \in \R_+ \times \O: X_{t \wedge H_X}(\o) \notin l \big\},
  \end{align*}
 and the projection $\pi \colon \R_+ \times \O \rightarrow \O$ onto the second coordinate. Then, by the assumption, there exists $\e > 0$ such that
  \begin{align*}
   \PV \big( \pi(A) \big) > \e.
  \end{align*}
 The section theorem (cf.\ \cite[IV-83, p.~137f]{DellacherieMeyerA}) asserts that there exists a stopping time~$R$ with
  \begin{enumerate}
   \item for all $\o \in \O$ with $R(\o) < +\infty$: $\big( R(\o),\o \big) \in A$, that is, $X_{R \wedge H_X}(\o) \notin l$,
   \item $\PV_{(l,x)}(R < +\infty) \geq \PV \big( \pi(A) \big) - \e > 0$.
  \end{enumerate}
 In particular, we have $\PV_{(l,x)}( X_{R \wedge H_X} \notin l ) \geq \PV_{(l,x)}(R < +\infty) > 0$.
 
 However, we are going to show that for every stopping time $R$, we have
  \begin{align*}
    \PV_{(l,x)} \big( X_{R \wedge H_X} \notin l \big) = 0,
  \end{align*}
 which yields a contradiction to the above: We start by observing that 
  \begin{align*}
   \PV_{(l,x)} \big( X_{R \wedge H_X} \notin l ; R \geq H_X \big) 
   = \PV_{(l,x)} \big( X_{H_X} \notin l ; R \geq H_X \big) = 0,
  \end{align*}
 because $X_{H_X} = X_{H_{\cLV(l)}} \in \cLV(l) \subseteq l$. Thus, we have
  \begin{align*}
    \PV_{(l,x)} \big( X_{R \wedge H_X} \notin l \big)
    & = \PV_{(l,x)} \big( X_R \notin l , R < H_X \big) \\
    & \leq \PV_{(l,x)} \big( X_R \in \comp l^0 , R < H_X \big) \\
    & = \EV_{(l,x)} \big( \PV_{X_R} (H_X = 0) ; R < H_X \big),    
  \end{align*}
 where in the last step we used the fact that for all $g \in \cG$,
  \begin{align*}
   \PV_g (H_X = 0) =
    \left\{
    \begin{aligned}
     1, & \quad g \in \comp l^0 \\
     0, & \quad g \in l^0
    \end{aligned}   
    \right\}
    = \1_{\comp l^0}(g),
  \end{align*}
 which is an immediate consequence of $H_X$ being the debut of the closed set $\comp l^0$
 for the right continuous, normal process $X$.
 Next, the strong Markov property of $X$ implies $\PV_{X_R} (H_X = 0) = \PV_{(l,x)} (H_X \circ R = 0 \,|\, \sF_{R+})$,
 so by using this
 together with the terminal time property of $H_X$ and $\{R < H_X\} \in \sF_R$ (see, e.g., \cite[Proposition~I.6.8]{BlumenthalGetoor69}), we get
    \begin{align*}
    \PV_{(l,x)} \big( X_{R \wedge H_X} \notin l \big)
    & = \PV_{(l,x)} \big( H_X = R , R < H_X \big) = 0.    \qedhere
  \end{align*}
\end{proof}

We are now able to restrict our attention to the initial edge (and, thus, to its local coordinate)
of the Brownian motion when considering the process stopped on leaving this edge.
This allows us to gain full insight into its exit distributions.

In the following results, we set $[0, \cR_l] := [0, +\infty)$ if $\cR_l = +\infty$.

\begin{theorem} \label{theo:G_BM:exit distributions}
 Let $X$ be a Brownian motion on $\cG$, $B$ be the one-dimensional Brownian motion,
 as well as $\pi^2 \colon \cG \rightarrow \R_+$ be the projection onto the local coordinate.
 Then for every $g = (l,x) \in \cG$, and for $A \in \sB(\cG)$ with
 $A' := \pi^2(A \cap l) \subseteq [0, \cR_l]$ being open (in the topology of $[0, \cR_l]$), 
 the following holds true:
  \begin{align*}
    \PV_{(l,x)} \circ \big( H^{X'}_A, X'_{H^{X'}_A} \big)^{-1}
    = \PV^B_{x} \circ \big( H^{B'}_{A'}, (l, B'_{H^{B'}_{A'}}) \big)^{-1},
  \end{align*}
 where $X' := X_{\,\cdot\, \wedge H_X}$, $B' := B_{\,\cdot\, \wedge H_B}$,
 and $H^{X'}_A$,  $H^{B'}_{A'}$ are the debuts of $A$, $A'$ for~$X'$, $B'$ respectively.
\end{theorem}
\begin{proof}
 It follows from Theorem~\ref{theo:G_BM:stopped BM is on starting edge} 
 that $\tX' := \pi^2(X_{\,\cdot\, \wedge H_X})$ is a right continuous process with values in $[0, \cR_l]$,
 having the same finite dimensional distributions as the stopped Brownian motion $B' = B_{\,\cdot\, \wedge H_B}$. 

 Let $Y$ be the canonical right continuous coordinate process on the space $\O$ of all right continuous maps $\R_+ \rightarrow [0, \cR_l]$.
 Define the path mappings $\Phi^{\tX'}$ and $\Phi^{B'}$ from~$\tX'$ and $B'$ to $\O$ by 
  \begin{align*}
    \Phi^{\tX'} \colon \O^X \rightarrow \O, & \quad \o^X \mapsto \tX'_{\cdot\,}(\o^X), \\
    \Phi^{B'} \colon \O^B \rightarrow \O, & \quad \o^B \mapsto B'_{\cdot\,}(\o^B),
  \end{align*}
 that is, we have
  \begin{align*}
   \forall t \geq 0: \quad Y_t \circ \Phi^{\tX'} = \tX_t' \quad \text{and} \quad Y_t \circ \Phi^{B'} = B_t'.
  \end{align*}

 Consider the debut of $A' \in \sB \big( [0, \cR_l] \big)$ for $Y$:
   \begin{align} \label{eq:G_BM:exit distributions, measurability of hitting time}
    H^Y_{A'} := \inf \{ t \geq 0: Y_t \in A' \}.
   \end{align}
 $H^Y_{A'}$ and $Y_{H^Y_{A'}}$ are $\sF^Y_\infty$-measurable (as the hitting time of any open set is a stopping time over 
 $(\sF^Y_{t+}, t \geq 0)$).
 If $A \in \sB(\cG)$ with $\pi^2(A \cap l) = A'$, then we have 
  \begin{align*}
   H^Y_{A'} \circ \Phi^{\tX'} 
   & = \inf \{ t \geq 0: Y_t \circ \Phi^{\tX'} \in A' \} \\
   & = \inf \{ t \geq 0: \pi^2(X_{t \wedge H_X}) \in \pi^2(A \cap l) \} \\
   & = H^{X'}_A,
  \end{align*}
 where we used Theorem~\ref{theo:G_BM:stopped BM is on starting edge} for the last identity. This gives for any $\o^X \in \O^X$:
  \begin{align*}
   Y_{H^Y_{A'}} \circ \Phi^{\tX'} (\o^X)
   & = Y_{H^Y_{A'} (\Phi^{\tX'} (\o^X) )} \big( \Phi^{\tX'} (\o^X) \big) \\
   & = \tX'_{H^{X'}_A} (\o^X) \\ 
   & = \pi^2(X'_{H^{X'}_A}) (\o^X).  
  \end{align*}  
 Analogously, we get
  \begin{align*}
   H^Y_{A'} \circ \Phi^{B'} 
    = H^{B'}_{A'}   
   \quad \text{and} \quad
   Y_{H^Y_{A'}} \circ \Phi^{B'} 
    = B'_{H^{B'}_{A'}}.
  \end{align*}
 Thus, for any $f \in \sB \big( [0, +\infty] \big) \otimes \sB \big( [0, \cR_l] \big)$,
 setting $G := f( H^Y_{A'}, Y_{H^Y_{A'}} ) \in \sF^Y_\infty$ gives
  \begin{align*}
   \EV_{(l,x)} \big( f \big( H^{X'}_{A}, \pi^2(X'_{H^{X'}_{A}}) \big) \big)
   & = \EV_{(l,x)} \big( G \circ \Phi^{X'} \big) \\
   & = \EV^B_{x} \big( G \circ \Phi^{B'} \big) \\
   & = \EV^B_{x} \big( f \big( H^{B'}_{A'}, B'_{H^{B'}_{A'}} \big) \big),
  \end{align*}
 which together with Theorem~\ref{theo:G_BM:stopped BM is on starting edge} concludes the proof.
\end{proof}

\begin{remark} \label{rem:G_BM:exit distributions}
 As easily observed in the above proof, Theorem~\ref{theo:G_BM:exit distributions} can also be stated for any $A \in \sB(\cG)$ 
  with $A' := \pi^2(A \cap l) \in \sB \big( [0, \cR_l] \big)$,
  as long as the debut~$H^Y_{A'}$ of~$A'$, as defined in \eqref{eq:G_BM:exit distributions, measurability of hitting time},
  attains $\sF^Y_\infty$-measurability, with $\sF^Y_\infty = \sigma(Y_t, t \geq 0)$ being the $\sigma$-algebra generated by
  a suitable coordinate process $Y$ on $[0, \cR_l]$.
 
 For instance, this is the case if $A$ is a closed set and the Brownian motion $X$ is known to be continuous up to the hit of $A$,
 cf.~\cite[Theorem~49.5]{Bauer96}, as we can then consider the continuous canonical coordinate process $Y$ 
 in the proof instead.
\end{remark}

We are usually interested in the exit distributions of the ``original'' Brownian motion~$X$ on a metric graph
instead of the stopped process $X'$, so we lift the results of Theorem~\ref{theo:G_BM:exit distributions} from $X'$ to $X$
(the same remark on the limitation to open subsets $A'$ also applies here):

\begin{corollary} \label{cor:G_BM:exit distributions}
 Let $g = (l,x) \in \cG$, $A \in \sB(\cG)$.
 \begin{enumerate}
  \item \label{itm:G_BM:exit distributions i}
        If $A \subseteq l$ and  
          $A' := \pi^2(A) \subseteq [0, \cR_l]$ is open,
          then Theorem~\ref{theo:G_BM:exit distributions} holds true.
  \item \label{itm:G_BM:exit distributions ii}
         If $A \subseteq l^0$ and
         $A' := \pi^2(A^\comp \cap l) \subseteq [0, \cR_l]$ is open, then
	  \begin{align*}
	    \PV_{(l,x)} \circ \big( H^X_{\comp A}, X_{H^X_{\comp A}} \big)^{-1}
	    = \PV^B_{x} \circ \big( H^B_{A'}, (l, B_{H^B_{A'}}) \big)^{-1}.
	  \end{align*}
 \end{enumerate}
\end{corollary}
\begin{proof}
  In the context of~\ref{itm:G_BM:exit distributions i}, the requirements of Theorem~\ref{theo:G_BM:exit distributions} are fulfilled, as 
   \begin{align*}
    A' = \pi^2(A) = \pi^2(A \cap l).
   \end{align*}
  
  Now, let $A$, $A'$ satisfy the assumptions of~\ref{itm:G_BM:exit distributions ii}.
  Then Theorem~\ref{theo:G_BM:exit distributions} gives  
	  \begin{align*}
	    \PV_{(l,x)} \circ \big( H^{X'}_{\comp A}, X'_{H^{X'}_{\comp A}} \big)^{-1}
	    = \PV^B_{x} \circ \big( H^{B'}_{A'}, (l, B'_{H^{B'}_{A'}}) \big)^{-1}.
	  \end{align*}
	We will consider both distributions separately:  
	
        As $\comp A \supseteq \comp l^0$, it is $H^X_{\comp A} \leq H_X$ and therefore
          \begin{align*}
           X'_{H^X_{\comp A}} = X_{H^X_{\comp A} \wedge H_X} = X_{H^X_{\comp A}}.
          \end{align*}
        Furthermore, we observe that 
          \begin{align*}
           H^{X'}_{\comp A}
           & = \inf \{ t \geq 0: X_{t \wedge H_X} \in \comp A \} \\
           & = \inf \{ t \in [0, H_X]:  X_t \in \comp A \} \\
           & = H^{X}_{\comp A},
          \end{align*}
       where the last identity follows again from $H^X_{\comp A} \leq H_X$:
         If $H^X_{\comp A} < H_X$, the identity is clear.
         If $H^X_{\comp A} = H_X$, then as $\comp l^0$ is closed, we have $X_{H_X} \in \comp l^0 \subseteq \comp A$,
         so $H_X$ lies in both sets, thus concluding that both infima are equal.
       In summary, this gives 
	  \begin{align*}
	    \PV_{(l,x)} \circ \big( H^{X'}_{\comp A}, X'_{H^{X'}_{\comp A}} \big)^{-1}
	    = \PV_{(l,x)} \circ \big( H^{}_{\comp A}, X_{H^{X}_{\comp A}} \big)^{-1}.
	  \end{align*}
	  
       Turning to the part for the Brownian motion $B$, observe that $A \subseteq \{l\} \times (0, \cR_l)$. This means that
         $A' = \pi^2(\comp A \cap l)$ contains the points $0$ and (if $l$ is an internal edge)~$\cR_l$.
       Thus, we have
         \begin{align*}
          H^{B'}_{A'} = H^{B}_{A'} \leq H_B,
         \end{align*}
       which shows
         \begin{align*}
          B'_{H^{B'}_{A'}} = B_{H^{B}_{A'} \wedge H_B} = B_{H^{B}_{A'}},
         \end{align*}
       resulting in
          \begin{align*}
	    \PV^B_{x} \circ \big( H^{B'}_{A'}, (l, B'_{H^{B'}_{A'}}) \big)^{-1}
	    & = \PV^B_{x} \circ \big( H^{B}_{A'}, (l, B_{H^{B}_{A'}}) \big)^{-1}. \qedhere
	  \end{align*}
\end{proof}

We are ready to turn to the fundamental properties of Brownian motions on metric graphs:

\begin{theorem} \label{theo:G_BM:resolvent of BM}
 Let $X$ be a Brownian motion on $\cG$. Then, for every~$f \in b\sB(\cG)$, $\a > 0$ and $g = (l,x) \in \cG$,
 the resolvent of $X$ reads, if $l = e \in \cE$,
  \begin{align*} 
    U_\a f (g) =  U^{D,e}_\a f(g) + e^{-\sqrt{2 \a} \, d( \cLV(e), g)} \, U_\a f \big( \cLV (e) \big), 
  \end{align*}
 and if $l = i \in \cI$,
  \begin{align*}
    U_\a f (g) = U^{D,i}_\a f(g) & + \frac{ \sinh \big( \sqrt{2 \a} \, d( \cLV_{+}(i), g) \big) }{ \sinh(\sqrt{2 \a} \, \cR_i) } \, U_\a f \big( \cLV_{-}(i) \big) \\
                             & + \frac{ \sinh \big( \sqrt{2 \a} \, d( \cLV_{-}(i), g) \big) }{ \sinh(\sqrt{2 \a} \, \cR_i) } \, U_\a f \big( \cLV_{+}(i) \big),
  \end{align*}   
 with 
  \begin{equation} \label{eq:G_BM:dirichlet resolvents}
  \begin{aligned}
   U^{D,e}_\a f(g) 
   & := U^{[0,\infty)}_\a f_l \big( d( \cLV(e), g) \big), \quad \hspace{0.3em} g \in e,\\
   U^{D,i}_\a f(g) 
   & := U^{[0, \r_i]}_\a f_l \big( d( \cLV_{-}(i), g) \big), \quad g \in i,
  \end{aligned}
  \end{equation}
 where $\big( U^{[0,\infty)}_\a, \a > 0 \big)$ and $\big( U^{[0, \cR_i]}_\a, \a > 0 \big)$ are the resolvents of
 the one-dimensional Brownian motion killed on leaving $[0, \infty)$, $[0, \cR_i]$ respectively, which are given 
 in~Examples~\ref{ex:B_HL:dirichlet BB on half line} and~\ref{ex:B_IN:dirichlet BB on interval}.
\end{theorem}
\begin{proof}
 The decomposition of the resolvent at the stopping time $H_X$ with the help of Dynkin's formula~\eqref{eq:Dynkins formula (resolvent)} yields
 for $g = (l,x) \in \cG$, $f \in b\sB(\cG)$:
  \begin{align*}
   U_\a f(g)
   & = \EV_g \Big( \int_0^{H_X} e^{-\a t} \, f(X_t) \, dt \Big) + \EV_g \big( e^{-\a H_X} \, U_\a f(X_{H_X}) \big).
  \end{align*}
 Thus, by Theorem~\ref{theo:G_BM:resolvent characterization of BM on MG}, we have
  \begin{align*}
   U_\a f(g)
   & = \EV^B_x \Big( \int_0^{H_B} e^{-\a t} \, f(l, B_t) \, dt \Big) + \EV^B_x \big( e^{-\a H_B} \, U_\a f(l, B_{H_B}) \big).
  \end{align*}
 With $H_B = \inf \{ t \geq 0: B_t = 0 \}$ or $H_B = \inf \big\{ t \geq 0: B_t \in \{ 0, \cR_l \} \big\}$ depending on whether
 $l \in \cE$ or $l \in \cI$, the passage time formulas of the one-dimensional Brownian motion (cf.~\cite[Section~1.7]{ItoMcKean74}) conclude
 the proof: We only need to note that for any $g = (l,x) \in \cG$, we have  
 $\cLV_{-}(l) = (l,0)$, $x = d \big( \cLV_{-}(l), g \big)$ and $\cLV_{+}(l) = (l, \cR_l)$, $\cR_l - x = d \big( \cLV_{+}(l), g \big)$
 in case ~$l \in \cI$,
 whereas $\cLV(l) = (l,0)$, $x = d \big( \cLV(l), g \big)$ in case~$l \in \cE$.
\end{proof}

As seen in the examinations for the resolvents $\big( U^{[0,\infty)}_\a, \a > 0 \big)$ and ${\big( U^{[a,b]}_\a, \a > 0 \big)}$
of the ``Dirichlet'' Brownian motions on $[0,\infty)$ and $[a,b]$
(cf.~Examples~\ref{ex:B_HL:dirichlet BB on half line} and \ref{ex:B_IN:dirichlet BB on interval}),
 \begin{itemize}
  \item $( U^{[0,\infty)}_\a, \a > 0 )$ maps $b\sB([0, \infty))$ on $b\cC([0, \infty))$ and $\cC_0([0, \infty))$ on $\cC^2_0([0,\infty))$,
        and assumes the boundary values $U^{[0,\infty)} f(0) = 0$, $U^{[0,\infty)} f''(0) = -2 f(0)$,
  \item $( U^{[0, \cR_i]}_\a, \a > 0 )$ maps $b\sB([0, \cR_i])$ on $b\cC([0, \cR_i])$ and $\cC([0, \cR_i])$ on $\cC^2([0,\cR_i])$,
        and assumes the boundary values $U^{[0, \cR_i]} f(x) = 0$, $U^{[0, \cR_i]} f''(x) = -2 f(x)$, for $x \in \{ 0, \cR_i \}$.
 \end{itemize}
Thus, the resolvents defined in equation~\eqref{eq:G_BM:dirichlet resolvents} are continuous functions,
twice continuously differentiable inside their respective edge for any $f \in \cC_0(\cG)$, and assume for $e \in \cE$, $i \in \cI$ the~values
 \begin{align*}
  U^{D,e}_\a f \big( \cLV (e) \big)  = 0, \quad
  &  U^{D,e}_\a f'' \big( \cLV (e) \big)  = -2 f \big( \cLV (e) \big), \\
  U^{D,i}_\a f \big( \cLV_{-}(i) \big) = 0, \quad
  & U^{D,i}_\a f'' \big( \cLV_{-}(i) \big) = -2 f \big( \cLV_{-}(i) \big), \\
  U^{D,i}_\a f \big( \cLV_{+}(i) \big) = 0, \quad
   & U^{D,i}_\a f'' \big( \cLV_{+}(i) \big) = -2 f \big( \cLV_{+}(i) \big).
 \end{align*}
Therefore, these boundary values for resolvents $U^{D,e}$, $U^{D,i}$ of various edges $e$, $i$, incident with the same vertex,
coincide on their common vertex.  
Then, by the decompositions given in Theorem~\ref{theo:G_BM:resolvent of BM} 
for the resolvent $(U_\a, \a > 0)$ of a Brownian motion on a metric graph,
$U_\a f$ extends to a twice continuously differentiable function on~$\cG$, yielding:

\begin{corollary} \label{cor:G_BM:resolvent preserves C}
 The resolvent $(U_\a, \a > 0)$ of a Brownian motion on a metric graph maps
 $b\sB(\cG)$ on $b\cC(\cG)$, and $\cC_0(\cG)$ on $\cC^2_0(\cG)$. 
\end{corollary}

We are now able to prove the first part of our main characterization result:

\begin{proof}[Proof of~Theorem~\ref{theo:BM characterization}, first part]
 The right continuity of $X$ together with Corollary~\ref{cor:G_BM:resolvent preserves C} show the Feller property of $X$
 (cf.~Appendix~\ref{app:Feller processes}, especially equation~\eqref{eq:Feller (resolvent)}), thus $X$ is uniquely determined by its $\cC_0$-generator.

 Let $f \in \sD(A)$. Then there exist $h \in \cC_0(\cG)$ and $\a > 0$ with $f = U_\a h$, 
 and $U_\a h \in \cC^2_0(\cG)$ holds by Corollary~\ref{cor:G_BM:resolvent preserves C}.
 By differentiating twice the decomposition given in Theorem~\ref{theo:G_BM:resolvent of BM}, we get for $g = (l,x) \in \cG$, 
 in case $l = i \in \cI$:
   \begin{align*}
   \frac{1}{2} f''(g) 
   & = \frac{1}{2} \, U^{D,i}_\a h''(g)   + \a \, \frac{ \sinh \big( \sqrt{2 \a} \, d( \cLV_{-}(i), g) \big) }{ \sinh(\sqrt{2 \a}) } \, U_\a h \big( \cLV_{-}(i) \big)  \\
   & \hspace*{6.6em}                       + \a \, \frac{ \sinh \big( \sqrt{2 \a} \, d( \cLV_{+}(i), g) \big) }{ \sinh(\sqrt{2 \a}) } \, U_\a h \big( \cLV_{+}(i) \big), \\
   & = \a \, U^{D,i}_\a h(g) - h(g)       + \a \, \frac{ \sinh \big( \sqrt{2 \a} \, d( \cLV_{-}(i), g) \big) }{ \sinh(\sqrt{2 \a}) } \, U_\a h \big( \cLV_{-}(i) \big)  \\
   & \hspace*{9.12em}                       + \a \, \frac{ \sinh \big( \sqrt{2 \a} \, d( \cLV_{+}(i), g) \big) }{ \sinh(\sqrt{2 \a}) } \, U_\a h \big( \cLV_{+}(i) \big), \\
   & = \a \, U_\a h (g) - h(g),
  \end{align*}
 and in case $l = e \in \cE$:
  \begin{align*}
   \frac{1}{2} f''(g) 
   & = \frac{1}{2} \, U^{D,e}_\a h''(g) + \a \, e^{-\sqrt{2 \a} \, d( \cLV_{-}(e), g)} \, U_\a h \big( \cLV_{-}(e) \big) \\
   & = \a \, U^{D,e}_\a h(g) - h(g) + \a \, e^{-\sqrt{2 \a} \, d( \cLV_{-}(e), g)} \, U_\a h \big( \cLV_{-}(e) \big) \\
   & = \a \, U_\a h (g) - h(g).
  \end{align*}
 Thus, for any $f \in \sD(A)$, we have $f \in \cC^2_0(\cG)$ and $A f = \frac{1}{2} \D f$ on $\cG$.
\end{proof}

\section{Computing the Generator: Feller's Theorem} \label{sec:G_BM:Fellers theorem}

As seen above, every Brownian motion on a metric graph is a Feller process with generator $A = \frac{1}{2} \D$.
Therefore (cf.~Appendix~\ref{app:Feller processes}), it is uniquely characterized 
by its generator domain, more accurately: by the generator's boundary conditions.
We are going to extend the classical results of the half-line and interval cases
by generalizing the approach of \cite[Lemma~6.2]{Knight81} and \cite[Section~8]{ItoMcKean63},
and will prove our main results Theorem~\ref{theo:G_BM:Feller data} and Theorem~\ref{theo:BM characterization}:

\begin{proof}[Proof of Theorem~\ref{theo:G_BM:Feller data}]
 Let $v \in \cV$. For all $\e > 0$, consider
  \begin{align*}
    \t_\e := \inf \big\{ t \geq 0: X_t \in \comp \overline{\BB_\e(v)} \big\}.
  \end{align*}
 
 In case $v$ is a trap, 
 we can compute the generator directly: Then
  \begin{align*}
   Af(v) = \lim_{t \downarrow 0} \frac{\EV_v \big( f(X_t) \big) - f(v)}{t} = 0
  \end{align*}
 holds true, thus choosing $c^v_3 = 1$ and $c^v_1 = c^{v,l}_2 = c^v_4 = 0$ for all $l \in \cL(v)$ gives
  \begin{align*}
    c^v_1 f(v) - \sum_{l \in \cL(v)} c^{v,l}_2 f_l'(v) + c^v_3 A f(v) - \int \big( f(g) - f(v) \big) \, c^v_4(dg) = 0.
  \end{align*} 
 This choice coincides with the definition of the parameters in the theorem,
 because in the case of a trap $v$, we have $\EV_v(\t_\e) = +\infty$ for all $\e > 0$,
 all (scaled) exit distributions read $\PV_v ( X_{\t_\e} = \D ) = \nu^v_\e = \mu^v_\e = 0$, 
 and thus $K^v_\e = 1$ holds for all~$\e > 0$ as well as $\overbar{\mu}^v = 0$.
 
 If $v$ is not a trap, then due to the Feller property of~$X$,
 Lemma~\ref{lem:A_FP:trap exit expectation} ensures that
 $\EV_v(\t_\e) < +\infty$ holds true for all $\e > 0$ sufficiently small.
 Thus, Dynkin's formula~\eqref{eq:Dynkins formula (generator)} is applicable. It yields 
  \begin{equation} \label{eq:G_BM:Feller data, proof I}
  \begin{aligned}
   Af(v)
   & = \lim_{\e \downarrow 0} \frac{\EV_v \big( f(X_{\t_\e}) \big) - f(v)}{\EV_v(\t_\e)} \\
   & = \lim_{\e \downarrow 0} \Big( - f(v) \, \frac{\PV_v ( X_{\t_\e} = \D )}{\EV_v(\t_\e)} + \int_{\cG \bs \{v\}} \big( f(g) - f(v) \big) \, \nu^{v}_\e (dg) \Big),
  \end{aligned}
  \end{equation} 
 with the scaled exit measures $\nu^v_\e$ being defined by 
  \begin{align*}
    \nu^v_\e (dg) & := \frac{\PV_v ( X_{\t_\e} \in dg )}{\EV_v(\t_\e)}, \quad \e > 0.
  \end{align*}
 They are measures on $\cG \bs \{v\}$, as the support of $X_{\t_\e}$ is the completion of~$\comp \overline{\BB_\e(v)}$ in~$\cG$ and therefore is a subset of~$\cG \bs \{v\}$.
 Introducing the normalizing constants
  \begin{align*}
    K^v_{\e} & := 1 + \frac{\PV_v ( X_{\t_{\e}} = \D )}{\EV_v(\t_\e)} + \int_{\cG \bs \{v\}} \big( 1 - e^{-d(v,g)} \big) \, \nu^v_{\e} (dg), \quad \e > 0,
  \end{align*}
 equation \eqref{eq:G_BM:Feller data, proof I} implies (as $\frac{1}{K^v_{\e}} \in [0,1]$ for all $\e > 0$) that
  \begin{align} \label{eq:G_BM:Feller data, proof II}
   0 = \lim_{\e \downarrow 0} \Big(  f(v) \, \frac{\PV_v ( X_{\t_\e} = \D )}{\EV_v(\t_\e) \, K^v_{\e}} 
                                 + A f(v) \, \frac{1}{K^v_\e} 
                                 -  \int_{\cG \bs \{v\}} \big( f(g) - f(v) \big) \, \frac{\nu^v_\e (dg)}{K^v_\e} \Big).
  \end{align} 
 We rescale the measures $\nu^v_\e$ by introducing the measures 
  \begin{align*}
    \mu^v_\e (dg) & := \big( 1 - e^{-d(v,g)} \big) \, \frac{\nu^v_\e (dg)}{K^v_\e}, \quad \e > 0,
  \end{align*}  
 on $\cG \bs \{v\}$. It is immediate that equation \eqref{eq:G_BM:Feller data, proof II} then is equivalent to
  \begin{align} \label{eq:G_BM:Feller data, proof III}
   0 = \lim_{\e \downarrow 0} \Big(  f(v) \, \frac{\PV_v ( X_{\t_\e} = \D )}{\EV_v(\t_\e) \, K^v_{\e}} 
                                 + A f(v) \, \frac{1}{K^v_\e} 
                                 - \int_{\cG \bs \{v\}} \frac{f(g) - f(v)}{ 1 - e^{-d(v,g)} } \, \mu^v_\e (dg) \Big).
  \end{align}  
  
 Let $\overbar{\mu}^v_\e$ be the extensions of the measures $\mu^v_\e$ to the compactification $\overbar{\cG \bs \{v\}}$ of~$\cG \bs \{v\}$ 
 (see subsection \ref{subsec:MG:compactification} for details on the compactification of a subspace of a metric graph), that is, we define the measures
 $\overbar{\mu}^v_\e$ on $\overbar{\cG \bs \{v\}}$ by
  \begin{align*}
    \overbar{\mu}^v_\e (dg) & := \mu^v_\e \big( dg \cap \big(\cG \bs \{v\}\big) \big), \quad \e > 0.
  \end{align*}
 Then the above identity \eqref{eq:G_BM:Feller data, proof III} remains valid for $\overbar{\mu}^v_\e$ instead of $\mu^v_\e$, where 
  \begin{align*}
   g \mapsto \frac{f(g) - f(v)}{ 1 - e^{-d(v,g)} }
  \end{align*}
 is continuously extended from $\cG \bs \{v\}$ to $\overbar{\cG \bs \{v\}}$ by
  \begin{align*}
   \forall l \in \cL(v): &  \quad \lim_{g \rightarrow v, g \in l^0} \frac{f(g) - f(v)}{ 1 - e^{-d(v,g)} } = f_l'(v), \\
   \forall e \in \cE: & \quad     \lim_{g \rightarrow \infty, g \in e^0} \frac{f(g) - f(v)}{ 1 - e^{-d(v,g)} } = - f(v),
  \end{align*}
 because $f \in \sD(A) \subseteq \cC^2_0(\cG) \subseteq \cC^{0,2}_0(\cG)$ and $f \in \sD(A) \subseteq \cC^2_0(\cG) \subseteq \cC_0(\cG)$.
 
 As $\cC( \overbar{\cG \bs \{v\}} )$ is separable (see Theorem~\ref{theo:G_MG:cont functions on compactification separable})
 and all measures $\overbar{\mu}^v_\e$, $\e > 0$, are bounded by $1$,
 there exists a sequence $(\e_n, n \in \N)$ of strictly positive numbers, converging to zero, such that
 $(\overbar{\mu}^v_{\e_n}, n \in \N)$ converges weakly to a measure~$\overbar{\mu}^v$ on~$\overbar{\cG \bs \{v\}}$.\footnote{This can be shown
 by employing the standard argument used in Helly's selection theorem: Let $\sS := \{ h_m, m \in \N \}$ be a countable, dense subset of $\cC( \overbar{\cG \bs \{v\}} )$,
 and $(\tilde{\e}_n, n \in \N)$ a sequence of strictly positive numbers, converging to zero.
 As all measures are bounded by $1$, the ``array'' $\big( \int \frac{h_m}{\norm{h_m}} \, d \overbar{\mu}^v_{\tilde{\e}_n}, m,n \in \N \big)$
 is bounded by $1$. By the diagonal method (see, e.g., \cite[Theorem~25.13]{Billingsley79}), it is possible to choose a subsequence $(\e_n, n \in \N)$ 
 of $(\tilde{\e}_n, n \in \N)$ such that $\lim_n \int h_m \, d \overbar{\mu}^v_{\e_n}$ exists for all $m \in \N$, that is for 
 all functions in a dense subset of $\cC( \overbar{\cG \bs \{v\}} )$. Thus, $\lim_n \int f \, d \overbar{\mu}^v_{\e_n}$ exists
 for all $f \in \cC( \overbar{\cG \bs \{v\}} )$ and defines a positive linear functional on~$\cC( \overbar{\cG \bs \{v\}} )$. 
 Therefore, by the Riesz\textendash{}Markov\textendash{}Kakutani representation theorem, there exists a measure~$\overbar{\mu}^v$ on $\overbar{\cG \bs \{v\}}$
 which satisfies $\lim_n \int f \, d \overbar{\mu}^v_{\e_n} = \int f \, d \overbar{\mu}^v$.}
 
 The sequences $\big( \frac{1}{K^v_{\e}}, \e > 0 \big)$ and $\big( \frac{\PV_v ( X_{\t_\e} = \D )}{\EV_v(\t_\e) \, K^v_{\e}}, \e > 0 \big)$ are bounded by $1$ as well,
 thus by choosing appropriate subsequences of $(\e_n, n \in \N)$ and naming them $(\e_n, n \in \N)$ again if necessary, we also obtain the existence of
  \begin{align*}
   c^{v, \D}_1 & := \lim_{n \rightarrow \infty} \frac{\PV_v ( X_{\t_{\e_n}} = \D )}{\EV_v(\t_{\e_n}) \, K^v_{\e_n}}, \\
   c^v_3 & := \lim_{n \rightarrow \infty} \frac{1}{K^v_{\e_n}}. 
  \end{align*}
 Inserting everything in equation \eqref{eq:G_BM:Feller data, proof III} shows that
  \begin{align*}
   0 = c^{v, \D}_1 \, f(v) + c^v_3 \, A f(v) 
   - \int_{\overbar{\cG \bs \{v\}}} \frac{f(g) - f(v)}{ 1 - e^{-d(v,g)} } \, \overbar{\mu}^v (dg),
  \end{align*} 
 or equivalently that
  \begin{align*}
   0 & = \Big( c^{v, \D}_1  + \sum_{e \in \cE} \overbar{\mu}^v \big( \{ (e, +\infty) \} \big) \Big) \, f(v) 
       - \sum_{\substack{l \in \cL(v), \\ v = \cLV_-(l)}} \overbar{\mu}^v \big( \{ (l, 0+) \} \big) \, f_l'(v) \\
     & \quad 
       - \sum_{\substack{l \in \cI(v), \\ v = \cLV_+(l)}} \overbar{\mu}^v \big( \{ (l, \cR_l-) \} \big) \, f_l'(v) 
       + \ c^v_3 \, A f(v) 
       - \int_{\cG \bs \{v\}} \frac{f(g) - f(v)}{ 1 - e^{-d(v,g)} } \, \overbar{\mu}^v (dg).
  \end{align*} 
  Setting $c^v_1 := c^{v, \D}_1  + \sum_{e \in \cE} \overbar{\mu}^v \big( \{ (e, +\infty) \} \big)$, and for each $l \in \cL(v)$ either $c^{v,l}_2 := \overbar{\mu}^v \big( \{ (l, 0+) \} \big)$ 
  or $c^{v,l}_2 := \overbar{\mu}^v \big( \{ (l, \cR_l-) \} \big)$ depending on whether $v \in \cLV_-(l)$ or $v \in \cLV_+(l)$,
  as well as defining the measure $c^v_4$ on $\cG \bs \{v\}$ by $c^v_4 (dg) := \frac{1}{1 - e^{-d(v,g)}} \, \overbar{\mu}^v(dg)$,
  yields 
   \begin{align*}
    0 = c^v_1 \, f(v) - \sum_{l \in \cL(v)} c^{v,l}_2 \, f_l'(v) + c^v_3 \, A f(v) - \int_{\cG \bs \{v\}} \big( f(g) - f(v) \big) \, c^v_4(dg).
   \end{align*}  
  This completes the proof, as insertion of the definitions offers the normalization
   \begin{align*}
     & c^v_1 + \sum_{l \in \cL(v)} c^{v,l}_2 + c^v_3 + \int_{\cG \bs \{v\}} \big( 1 - e^{-d(v,g)} \big) \, c^v_4 (dg) \\
     & = c^{v, \D}_1 + c^v_3 + \int_{\overbar{\cG \bs \{v\}}} \, \overbar{\mu}^v(dg)  \\
     & = \lim_{n \rightarrow \infty} \frac{1}{K^v_{\e_n}}  \Big( \frac{\PV_v ( X_{\t_{\e_n}} = \D )}{\EV_v(\t_{\e_n})} + 1 + \int_{\cG \bs \{v\}} \big( 1 - e^{-d(v,g)} \big) \, \nu^v_{\e_n} ( dg )  \Big) \\
     & = 1.  \qedhere
   \end{align*}
\end{proof}

By examining the proof, the reader may observe that the ``Brownian'' property of $X$ was not used anywhere.
Indeed, the above result holds true for any Feller process on~$\cG$ (we will not need this fact).

\begin{proof}[Proof of~Theorem~\ref{theo:BM characterization}, second part]
 It only remains to show condition \eqref{eq:G_BM:infinite jumpmeasure}:
 Assume that there is a $v \in \cV$ such that $p^{v,l}_2 = 0$ for all $l \in \cL(v)$, $p^v_3 = 0$ and $p^v_4$ is a finite measure on $\cG \bs \{v\}$.
 For any $f \in \cC_0(\cG)$, $\a > 0$, $u = \a U_\a f$ is an element of~$\sD(A)$, so by \eqref{eq:G_BM:generator boundary conditions} it fulfills  
  \begin{align*}
   \a U_\a f (v) \, \big( p^v_1 + p^v_4(\cG \bs \{v\}) \big)  - \int \a U_\a f(g) \, p^v_4(dg) = 0.
  \end{align*}
 Letting $\a \rightarrow +\infty$ yields with equation~\eqref{eq:Feller (resolvent convergence)} 
 and Lebesgue's dominated convergence theorem (as $\norm{\a U_\a f} \leq \norm{f}$) that for all $f \in \cC_0(\cG)$,
   \begin{align*}
    f (v) \, \big( p^v_1 + p^v_4(\cG \bs \{v\}) \big) = \int f(g) \, p^v_4(dg).
  \end{align*}
 But then $p^v_4$ must the Dirac measure in $v$, scaled by $p^v_1 + p^v_4(\cG \bs \{v\}) > 0$, which is impossible.
\end{proof}

\begin{remark} \label{rem:G_BM:trivial vertices}
 On every non-vertex point $g = (l,x) \in \cG^0$ of the graph $\cG$, the generator~$A$ of any Brownian motion $X$ on $\cG$
 reads
  \begin{align*}
   A f(g) = \frac{1}{2} \, \frac{\partial^2}{\partial x^2} f(l,x), \quad f \in \sD(A),
  \end{align*}
 being the usual differentiation of a function defined on some open subset of $\R$. 
 It is therefore necessary for the first derivate $f'$ of $f$ to exist and be continuous at $g$, that is,
  \begin{align*}
   \lim_{\xi \downarrow x} f_l'(\xi) = \lim_{\xi \uparrow x} f_l'(\xi).
  \end{align*}

 Therefore, if we introduce a new vertex $v'$ at $g = (l,x) \in \cG^0$, splitting the original edge~$l$ into two new edges $l_1'$, $l_2'$
 (as done in subsection~\ref{subsec:G_MG:tadpoles} in order to eliminate loops),
 the original Brownian motion $X$ will satisfy the boundary condition
  \begin{align*}
   \frac{1}{2} f_{l_1'}'(v') + \frac{1}{2} f_{l_2'}'(v') = 0, \quad f \in \sD(A),
  \end{align*}
 at the new vertex $v'$. Thus, we can always assume that we are able to introduce ``trivial'' vertices inside of existing edges which 
 do not change the generator or the \FW\ data of the underlying Brownian motion,
 in case the ``non-skew'' boundary condition above is chosen at the new vertices.
\end{remark}

\section{Further Results on the Generator of a Star Graph} \label{sec:G_IM:further results}

We are going to gain further insight into the star-graph case and derive results which will be necessary for our upcoming developments on the general case.

We first turn to the question on whether the generator of a Brownian motion on a star graph is uniquely characterized by the 
\FW\ data arising from Feller's theorem~\ref{theo:G_BM:Feller data}. In the context of star graphs, we know (see~\cite[Lemma~2.6]{WernerStar})
that the boundary conditions are also sufficient for a function to lie inside $\sD(A)$, that is, 
equality holds in equation~\eqref{eq:G_BM:generator boundary conditions}.
Of course, the generator domain $\sD(A)$ determines any Brownian motion.
Therefore, we need to ensure that no two different sets of boundary data give rise to the same set $\sD(A)$,
which does not seem obvious in the presence of non-local boundary conditions.

\begin{lemma} \label{lem:G_IM:Feller data is unique}
 For a star graph $\cG$ with star point $v$, let $c_1 \geq 0$, $c^e_2 \geq 0$ for each~$e \in \cE$, $c_3 \geq 0$, $c_4$ a measure on $\cG \bs \{v\}$ as well as
     $p_1 \geq 0$, $p^e_2 \geq 0$ for each $e \in \cE$, $p_3 \geq 0$, and~$p_4$~a measure on $\cG \bs \{v\}$ be given, which satisfy
 \begin{align*}
   c_1 + \sum_{e \in \cE} c^{e}_2 + c_3 + \int_{\cG \bs \{v\}} \big( 1 - e^{-d(v,g)} \big) \, c^v_4 \big( dg \big) & = 1, \\
   p_1 + \sum_{e \in \cE} p^{e}_2 + p_3 + \int_{\cG \bs \{v\}} \big( 1 - e^{-d(v,g)} \big) \, p^v_4 \big( dg \big) & = 1.
 \end{align*}
 If 
   \begin{align*}
      &  \Big\{ f \in \cC^2_0(\cG): 
                     c_1  f(v) - \sum_{e \in \cE} c^{e}_2  f_e'(v) + \frac{c_3}{2} f''(v) - \int_{\cG \bs \{v\}} \hspace*{-0.5em} \big( f(g) - f(v) \big) \, c_4(dg) = 0 \Big\} \\
    = \, & \Big\{ f \in \cC^2_0(\cG): 
                     p_1  f(v) - \sum_{e \in \cE} p^{e}_2  f_e'(v) + \frac{p_3}{2} f''(v) - \int_{\cG \bs \{v\}} \hspace*{-0.5em} \big( f(g) - f(v) \big) \, p_4(dg) = 0 \Big\},
  \end{align*}
 then 
   \begin{align*}
    c_1 = p_1, \quad \forall e \in \cE:~ c^e_2 = p^e_2, \quad c_3 = p_3, \quad c_4 = p_4.
  \end{align*}
\end{lemma}
\begin{proof}
 Let $X^p$ and $X^c$ be Brownian motions on the star graph $\cG$,
 constructed with the techniques of~\cite{WernerStar},
 which implement the boundary condition at $v$ given by the sets $\big(p_1, (p^e_2)_{e \in \cE}, p_3, p_4 \big)$ and $\big(c_1, (c^e_2)_{e \in \cE}, c_3, c_4 \big)$.
 With $A^p$, $U^p$ and $A^c$, $U^c$ being the generators and resolvents of $X^p$, $X^c$ respectively, \cite[Theorem~4.33]{WernerStar} asserts that
  \begin{align*}
  \sD(A^c)  = 
    \Big\{ & f \in \cC^2_0(\cG): \\
           &         c_1 f(v) - \sum_{e \in \cE} c^{e}_2 f_e'(v) + \frac{c_3}{2} f''(v) - \int_{\cG \bs \{v\}} \big( f(g) - f(v) \big) \, c_4(dg) = 0 \Big\} , \\
  \sD(A^p) = 
    \Big\{ & f \in \cC^2_0(\cG): \\
           &         p_1 f(v) - \sum_{e \in \cE} p^{e}_2 f_e'(v) + \frac{p_3}{2} f''(v) - \int_{\cG \bs \{v\}} \big( f(g) - f(v) \big) \, p_4(dg) = 0 \Big\}.
 \end{align*} 
 By assumption, the generators and thus the resolvents of $X^p$ and $X^c$ coincide, in particular 
 we have $U^p_\a f(v) = U^c_\a f(v)$ for all $\a > 0$, $f \in b\cC(\cG)$. \cite[Theorem~4.31]{WernerStar} then yields for all $\a > 0$, $f \in b\cC(\cG)$:
  \begin{equation}  \label{eq:G_IM:Feller data is unique, I}
  \begin{aligned}
   &  \quad ~ \frac { \sum_{e \in \cE} p^e_2 \, 2 \int_0^\infty e^{-\sqrt{2 \a} x} f(e,x) \, dx + p_3 f(0) + \int U^{W,D}_\a f(g) \, p_4(dg) }
                    { p_1 + \sqrt{2 \a} p_2 + \a p_3 + \int_0^\infty (1 - e^{-\sqrt{2 \a} l}) \, p^\Sigma_4 (dl) } & \\
   & =
     \frac { \sum_{e \in \cE} c^e_2 \, 2 \int_0^\infty e^{-\sqrt{2 \a} x} f(e,x) \, dx + c_3 f(0) + \int U^{W,D}_\a f(g) \, c_4(dg) }
           { c_1 + \sqrt{2 \a} c_2 + \a c_3 + \int_0^\infty (1 - e^{-\sqrt{2 \a} l}) \, c^\Sigma_4 (dl) },
  \end{aligned}  
  \end{equation}
 with $p^\Sigma_4(A) = \sum_{e \in \cE} p^e_4(A) = \sum_{e \in \cE} p_4 \big( \{e\} \times A \big)$, $A \in \sB \big( (0, +\infty) \big)$, and $c^\Sigma_4$ analogously.
 By inserting $f = 1$, we get
   \begin{align*}
   &  \quad ~ \frac{1}{\a} \, \frac { \sqrt{2 \a} p_2 + \a p_3 + \int \big( 1 - e^{-\sqrt{2 \a} x} \big) \, p^\Sigma_4(dx) }
                                    { p_1 + \sqrt{2 \a} p_2 + \a p_3 + \int_0^\infty (1 - e^{-\sqrt{2 \a} l}) \, p^\Sigma_4 (dl) } & \\
   & =
     \frac{1}{\a} \, \frac { \sqrt{2 \a} c_2 + \a c_3 + \int \big( 1 - e^{-\sqrt{2 \a} x} \big) \, c^\Sigma_4(dx) }
                           { c_1 + \sqrt{2 \a} c_2 + \a c_3 + \int_0^\infty (1 - e^{-\sqrt{2 \a} l}) \, c^\Sigma_4 (dl) },
  \end{align*}
 so when introducing 
  \begin{align*}
   \tilde{p}_\a & :=  \sqrt{2 \a} p_2 + \a p_3 + \int \big( 1 - e^{-\sqrt{2 \a} x} \big) \, p^\Sigma_4(dx), \\
   \tilde{c}_\a & :=  \sqrt{2 \a} c_2 + \a c_3 + \int \big( 1 - e^{-\sqrt{2 \a} x} \big) \, c^\Sigma_4(dx),
  \end{align*}
 it follows that
  \begin{align} \label{eq:G_IM:Feller data is unique, II}
   \frac{\tilde{p}_\a}{p_1 + \tilde{p}_\a} = \frac{\tilde{c}_\a}{c_1 + \tilde{c}_\a}.
  \end{align}
  
 If $p_1 \neq 0$, consider $D := \frac{c_1}{p_1}$.
 Then $c_1 = D \, p_1$ holds, and the above equation implies $\tilde{c}_\a = D \, \tilde{p}_\a$, that is
  \begin{align*}
   & \sqrt{2 \a} c_2 + \a c_3 + \int \big( 1 - e^{-\sqrt{2 \a} x} \big) \, c^\Sigma_4(dx) \\
   & = D \, \Big( \sqrt{2 \a} p_2 + \a p_3 + \int \big( 1 - e^{-\sqrt{2 \a} x} \big) \, p^\Sigma_4(dx) \Big).
  \end{align*}
 Dividing both sides by $\a$ and letting $\a \rightarrow \infty$ yields $c_3 = D \, p_3$, so 
  \begin{align} \label{eq:G_IM:Feller data is unique, III}
   \sqrt{2 \a} c_2 + \int \big( 1 - e^{-\sqrt{2 \a} x} \big) \, c^\Sigma_4(dx)
   = D \, \Big( \sqrt{2 \a} p_2 + \int \big( 1 - e^{-\sqrt{2 \a} x} \big) \, p^\Sigma_4(dx) \Big).
  \end{align} 
 Now dividing by $\sqrt{2 \a}$ and letting $\a \rightarrow \infty$ again yields $c_2 = D p_2$, thus
  \begin{align*}
   \int \big( 1 - e^{-\sqrt{2 \a} x} \big) \, c^\Sigma_4(dx)
   = D \, \int \big( 1 - e^{-\sqrt{2 \a} x} \big) \, p^\Sigma_4(dx).
  \end{align*} 
 But then
  \begin{align*}
    & c_1 + c_2 + c_3 + \int \big( 1 - e^{-\sqrt{2 \a} x} \big) \, c^\Sigma_4(dx) \\
    & = D \Big( p_1 + p_2 + p_3 + \int \big( 1 - e^{-\sqrt{2 \a} x} \big) \, p^\Sigma_4(dx) \Big),
  \end{align*}
 and by inserting $\a = \frac{1}{2}$, the normalizations of the $c$'s and $p$'s imply $D = 1$. 
 
 Thus, we have $c_1 = p_1$ and $c_3 = p_3$.
 Coming back to equation \eqref{eq:G_IM:Feller data is unique, I}, we obtain
  \begin{align*}
   & \sum_{e \in \cE} p^e_2 \, 2 \int_0^\infty e^{-\sqrt{2 \a} x} f(e,x) \, dx + \int U^{W,D}_\a f(g) \, p_4(dg) &  \\
    =  & \sum_{e \in \cE} c^e_2 \, 2 \int_0^\infty e^{-\sqrt{2 \a} x} f(e,x) \, dx + \int U^{W,D}_\a f(g) \, c_4(dg) &
  \end{align*}
 for all $\a > 0$, $f \in b\cC(\cG)$.
 Fix $e \in \cE$. By approximation with the help of Lebesgue's dominated convergence theorem,
 we can insert for $e^0 = \{e\} \times (0, \infty)$ the function $f = \1_{e^0}$ 
 in the above equation, yielding 
   \begin{align*}
   &  \frac{1}{\sqrt{2 \a}} \, p^e_2 + \frac{1}{\a} \, \int \big( 1 - e^{-\sqrt{2\a} x} \big) \, p^e_4(dx) 
    =   \frac{1}{\sqrt{2 \a}} \, c^e_2 + \frac{1}{\a} \, \int \big( 1 - e^{-\sqrt{2\a} x} \big) \, c^e_4(dx). &  
  \end{align*} 
 Multiplying by $\sqrt{2 \a}$ and letting $\a \rightarrow \infty$ gives $p^e_2 = c^e_2$. Therefore,
  \begin{align*}
   \forall \a > 0: \quad \int \big( 1 - e^{-\sqrt{2\a} x} \big) \, p^e_4(dx) = \int \big( 1 - e^{-\sqrt{2\a} x} \big) \, c^e_4(dx),
  \end{align*}
 which by the uniqueness of Laplace transforms 
 is only possible if $p^e_4 = c^e_4$. This completes the proof for $p_1 \neq 0$.
 
 If $p_1 = 0$, equation \eqref{eq:G_IM:Feller data is unique, II} implies that $c_1 = 0$ or $\tilde{p}_\a = 0$ for all $\a > 0$. The latter is impossible,
 so $c_1 = p_1 = 0$. Now using equation~\eqref{eq:G_IM:Feller data is unique, I} with $\a = \frac{1}{2}$ and $f = \1_{\{v\}}$ 
 (again approximating $f$ with $\cC_0(\cG)$-functions), 
 and utilizing the normalizations of the $c$'s and $p$'s, we get
  \begin{align*}
   \frac{p_3}{1 - \frac{p_3}{2}} = \frac{c_3}{1 - \frac{c_3}{2}},
  \end{align*}
 so $c_3 = p_3$. 
 
 First assume $p_3 \neq 0$. Inserting $c_3 = p_3$ in equation~\eqref{eq:G_IM:Feller data is unique, I} with $f = \1_{\{v\}}$ gives
  \begin{align*}
   \frac{p_3}{ \sqrt{2\a} p_2 + \a p_3 + \int \big( 1 - e^{-\sqrt{2 \a} x} \big) \, p^\Sigma_4(dx) }
   & = \frac{p_3}{ \sqrt{2\a} c_2 + \a p_3 + \int  \big( 1 - e^{-\sqrt{2 \a} x} \big) \, c^\Sigma_4(dx) },
  \end{align*}
 which is equivalent to
   \begin{align*}
    \sqrt{2\a} p_2 + \int \big( 1 - e^{-\sqrt{2 \a} x} \big) \, p^\Sigma_4(dx)
   & = \sqrt{2\a} c_2  + \int \big( 1 - e^{-\sqrt{2 \a} x} \big) \, c^\Sigma_4(dx),
  \end{align*}
 yielding equation \eqref{eq:G_IM:Feller data is unique, III} for $D = 1$. Thus, the rest of the proof then proceeds exactly as in the case $p_1 \neq 0$.
 
 If $p_1 = 0$ and $p_3 = 0$, we have already seen that $c_1 = 0$ and $c_3 = 0$ as well. Using equation~\eqref{eq:G_IM:Feller data is unique, I} with $\a = \frac{1}{2}$
 and $f(e,x) = e^{-\b x}$ for any $\b > 0$, it follows, as the the $c$'s and $p$'s are normalized, that 
  \begin{align*}
   &    \frac{2}{1 + \b} \, p_2 + \frac{2}{1 - \b^2} \, \int_0^\infty \big( e^{-\b x} - e^{-x} \big) \, p^\Sigma_4(dx) \\
    = \ & \frac{2}{1 + \b} \, c_2 + \frac{2}{1 - \b^2} \, \int_0^\infty \big( e^{-\b x} - e^{-x} \big) \, c^\Sigma_4(dx),
  \end{align*}
 with the integrals being finite, because $e^{-\b x} - e^{-x} = - e^{-x} \big( 1 - e^{-(\b-1)x} \big)$ for $\b > 1$
 and $0 \leq e^{-\b x} - e^{-x} \leq 1 - e^{-x}$ for $0 < \b \leq 1$. 
 Multiplying both sides by $\b$ and letting $\b \rightarrow +\infty$ yields $p_2 = c_2$, because 
 $\int_0^\infty \frac{ e^{-\b x} - e^{-x} }{\b - 1} \, p^\Sigma_4(dx) \rightarrow 0$ for $\b \rightarrow +\infty$.
 But then 
  \begin{align*}
   \int \big( e^{-\b x} - e^{-x} \big) \, p^\Sigma_4(dx) = \int \big( e^{-\b x} - e^{-x} \big) \, c^\Sigma_4(dx)
  \end{align*}
 holds for all $\b > 0$, and by adding $\int \big( 1 - e^{-x} \big) p^\Sigma_4(dx) = 1 - p_2 = 1 - c_2 = \int \big( 1 - e^{-x} \big) c^\Sigma_4(dx)$
 to both sides and setting $\b := \sqrt{2 \a}$, we get for all $\a > 0$
 \begin{align*}
  \int \big( 1 - e^{-\sqrt{2 \a} x} \big) \, p^\Sigma_4(dx) = \int \big( 1 - e^{-\sqrt{2 \a} x} \big) \, c^\Sigma_4(dx).
 \end{align*}
 The rest of the proof then proceeds as above. 
\end{proof}

We are going to employ the above result in order to show that the rather artificial part~$c^{\infty}_1$ of the killing weight $c_1 = c^\D_1 + c^\infty_1$ 
in Feller's theorem~\ref{theo:G_BM:Feller data} 
indeed vanishes in the star-graph case (here, the star vertex $v$ is left out in the notation of the \FW\ data).

We achieve this as follows: 
Starting with the Brownian motion $X$ which implements the killing parameter $c_1 = c^\D_1 + c^\infty_1$, 
we revive this process at its killing times via some revival distribution~$k$ with the identical copies method established in~\cite{WernerConcat}
(\cite[Section~3]{WernerStar} contains a short summary with applications in the Brownian context).
As killing can be interpreted as a jump to $\D$, which is now transformed to a jump to a revival point chosen by $k$,
we expect the killing weight $c_1$ to be transformed into a jump part $c_1 k$, which is then added to the original jump distribution~$c_4$.
However, an analysis of the boundary conditions for the revived process via two different methods shows a discrepancy:
The resolvent of the revived process can be decomposed with Dynkin's formula at the revival time, 
and shows that the ``full'' killing parameter $c_1 = c^\D_1 + c^\infty_1$ is shifted to the jump measure.
But when tracing back the explicit formulas of Feller's theorem for the \FW\ data of the revived process to the original process $X$,
it is seen that only the ``natural'' killing weight $c^\D_1$ is transformed, while leaving the ``artificial'' killing portion $c^\infty_1$ unaltered.
As the \FW\ data uniquely characterizes the process, this is only possible if $c^\infty_1$ already vanishes for the original process $X$. 

We are carrying out this program, starting with the analysis of the resolvent of the revived Brownian motion:

\begin{lemma} \label{lem:G_IM:Feller resolvent of revived X}
 Let $X$ be a Brownian motion on the star graph $\cG$ with generator
  \begin{align*}
   \sD(A)  = 
       \Big\{ & f \in \cC^2_0(\cG): \\ 
              &       c_1 f(v) - \sum_{e \in \cE} c^{e}_2 f_e'(v) + \frac{c_3}{2} f''(v) - \int_{\cG \bs \{v\}} \big( f(g) - f(v) \big) \, c_4(dg) = 0 \Big\}. 
  \end{align*}
 Let $q$ be a probability measure on $\cG$, and 
 $X^q$ be the identical copies process,
 resulting from successive revivals of $X^0 := X$ 
 with the revival kernel~$K^0$ which is defined by the transfer measure
  \begin{align*}
   k^0 \big( g, \, \cdot \, \big) := q, \quad g \in \cG.
  \end{align*}
 Then $X^q$ is a Brownian motion on $\cG$ with generator
  \begin{align*}
   \sD(A^q) = 
       \Big\{ &  f \in \cC^2_0(\cG) : \\ 
              & - \sum_{e \in \cE} p^{e}_2 f_e'(v) + \frac{p_3}{2} f''(v) - \int_{\cG \bs \{v\}} \big( f(g) - f(v) \big) \, (p_4 + p_1 \, q)(dg) = 0 \Big\}. 
  \end{align*}
\end{lemma}
\begin{proof}
 This result has already been proved in~\cite[Lemma~3.2]{WernerStar}, under the condition that the function
 $\varphi_\a := \EV_{\,\cdot\,}\big(e^{-\a \z}\big)$ satisfies
  \begin{align*}
   \varphi_\a \in \cC_0(\cG), \quad 1 - \varphi_\a \in \sD(A),
  \end{align*}
 with $\z$ being the lifetime of~$X$. Thus, it remains to check the above condition:
 
 \cite[Theorem~4.31]{WernerStar} together with \cite[Example~A.10]{WernerStar} show that the function 
    $\varphi_\a
     = 1 - \a \, U_\a \1_{\cG} $
 is in $\cC_0(\cG)$. 
 Furthermore, as $\D \notin \cG$ is isolated, we have $\1_{\cG} \in b\cC(\cG)$, so 
 $1 - \psi_\a = \a \, U_\a \1_{\cG}$ fulfills the boundary conditions for $X$ (cf.~the proof of~\cite[Theorem~4.33]{WernerStar}).
\end{proof}

Next, we deduce the \FW\ data of the revived process from the respective \FW\ data of the original process by explicitly computing the 
formulas given in Feller's theorem~\ref{theo:G_BM:Feller data}:

\begin{lemma} \label{lem:G_IM:Feller data of revived BB on star graph}
 Let $X$ be a Brownian motion on the star graph $\cG$ with generator
  \begin{align*}
   & \sD(A^X)  = 
       \Big\{  f \in \cC^2_0(\cG): \\ 
   &   \quad           (c^{\D}_1 + c^{\infty}_1) \, f(v) - \sum_{e \in \cE} c^{e}_2  f_e'(v) + \frac{c_3}{2} f''(v) - \int_{\cG \bs \{v\}} \big( f(g) - f(v) \big) \, c_4(dg) = 0 \Big\}
  \end{align*}
 with the \FW\ data $\big( c^{\D}_1, c^{\infty}_1, (c^{e}_2)_{e \in \cE}, c_3, c_4 \big)$ satisfying
  \begin{align*}
    c^{\D}_1 + c^{\infty}_1 + \sum_{e \in \cE} c^{e}_2 + c_3 + \int_{\cG \bs \{v\}} \big( 1 - e^{-d(v,g)} \big) \, c^v_4 (dg) = 1.
  \end{align*}
 If $c_1 = c^{\D}_1 + c^{\infty}_1 > 0$, construct $Y$ as the instant return process of $X$,
 that is, as the identical copies process of $X$ 
 with the revival kernel~$K^0$ being defined by the transfer measure 
  \begin{align*}
    k^0 \big( g, \, \cdot \, \big) = \e_v.
  \end{align*}
 Then $Y$ is a Brownian motion on $\cG$ with generator
  \begin{align*}
   \sD(A^Y)  = 
       \Big\{ & f \in \cC^2_0(\cG): \\ 
              &       c^{\infty}_1 f(v) - \sum_{e \in \cE} c^{e}_2 f_e'(v) + \frac{c_3}{2} f''(v) - \int_{\cG \bs \{v\}} \big( f(g) - f(v) \big) \, c_4(dg) = 0 \Big\}. 
  \end{align*} 
\end{lemma}
\begin{proof}
 By Lemma~\ref{lem:G_IM:Feller resolvent of revived X}, the revived process~$Y$ is a Brownian motion on $\cG$.
 As we will need to compare the formulas given in Feller's theorem~\ref{theo:G_BM:Feller data} for the \FW\ data of the processes $X$ and $Y$, 
 we indicate the defining entities for $X$, $Y$ by the corresponding superscript, that is, for instance
  \begin{align*}
   \t^X_\e = \inf \big\{ t \geq 0: X_t \in \comp \overline{\BB_e(v)} \big\}, \quad
   \t^Y_\e = \inf \big\{ t \geq 0: Y_t \in \comp \overline{\BB_e(v)} \big\}, \quad \e > 0.
  \end{align*}
 
 If $\PV_v (\t^X_\e < \z^X) = 0$ for all $\e > 0$, that is, if $\PV_v (\t^X_\e = \z^X) = 1$ holds for all~$\e > 0$,
 then (depending on whether $\EV_v(\t^X_\e)$ is infinite or finite) $v$ is an absorbing point or a holding point for~$X$,
 and in the latter case $X$ must jump directly from~$v$ to~$\D$ after an exponential holding time. The generators for these two cases read 
  \begin{align*}
    \sD(A^X) = \big\{ f \in \cC^2_0(\cG): \frac{1}{2} f''(v) = 0 \big\},
  \end{align*}
 and
  \begin{align*}
    \sD(A^X) = \big\{ f \in \cC^2_0(\cG): c_1 f(v) + \frac{c_3}{2} f''(v) = 0 \big\}
  \end{align*}
 with $c^{\infty}_1 = 0$, as $\nu^X_\e = 0$ holds for all $\e > 0$ in Feller's theorem~\ref{theo:G_BM:Feller data} by definition.
 But in both cases, the revived process $Y$ is just the Brownian motion absorbed in~$v$, so 
  \begin{align*}
    \sD(A^Y) = \big\{ f \in \cC^2_0(\cG): \frac{1}{2} f''(v) = 0 \big\},
  \end{align*}  
 conforming to the claim of the lemma.
 
 Otherwise, there is some $\e > 0$ with $\PV_v (\t^X_\e < \z^X) > 0$. As $\t^X_{\e'} \leq \t^X_{\e}$ holds for all $\e' < \e$,
 we then have for all $\e > 0$ sufficiently small
   \begin{align*}
     \PV_v (\t^X_\e < \z^X) > 0.
   \end{align*}
 We need to compare $\t^Y_\e$ with $\t^X_\e$: 
 While $\t^X_\e$ can be realized by $X$ jumping to $\D$ or entering $\cG \bs \overline{\BB_\e(v)}$,
 $\t^Y_\e$ is only realized if $X$ enters $\cG \bs \overline{\BB_\e(v)}$. If $\t^X_\e$ is realized by $X$ jumping to $\D$, then
 $Y$ restarts at $v$ and $\t^Y_\e = R^1 + \t^Y_\e \circ \T_{R^1}$ holds true, with the first revival time $R^1$ being equal to the death time $\z^X$ of $X$.
 Due to the strong Markov property, the number of revivals of $Y$ before leaving $\BB_\e(v)$ is geometrically distributed, so
  \begin{equation}
  \begin{aligned}  \label{eq:G_IM:Feller data of revived BB on star graph, proof I}
   \EV_v ( \t^Y_\e ) 
    = \sum_{n \in \N_0} & \big( n \, \EV_v( \t^X_\e \,|\, \t^X_\e = \z^X ) + \EV_v( \t^X_\e \,|\, \t^X_\e < \z^X ) \big) \\
        & \, \cdot \PV_v ( \t^X_\e = \z^X )^n \, \PV_v( \t^X_\e < \z^X ),
  \end{aligned}  
  \end{equation}
 which gives
  \begin{equation} \label{eq:G_IM:Feller data of revived BB on star graph, proof Ia}
  \begin{aligned}
   \EV_v ( \t^Y_\e ) 
   & = \frac{1-\PV_v( \t^X_\e < \z^X )}{\PV_v( \t^X_\e < \z^X )} \, \EV_v( \t^X_\e \,|\, \t^X_\e = \z^X ) + \EV_v( \t^X_\e \,|\, \t^X_\e < \z^X ) \\
   & = \frac{1}{\PV_v( \t^X_\e < \z^X )} \, \EV_v( \t^X_\e ).
  \end{aligned}   
  \end{equation}
  
 Before continuing, we prove equation~\eqref{eq:G_IM:Feller data of revived BB on star graph, proof I} rigorously: 
 By decomposing $\t^Y_\e$ with respect to the revival times $(R^n, n \in \N)$ of the concatenated process $Y$, we get
  \begin{align*}
   \EV_v ( \t^Y_\e ) = \sum_{n \in \N_0} \EV_v ( \t^Y_\e \,;\, R^{n} \leq \t^Y_\e < R^{n+1} ).
  \end{align*}
 Before the first revival time, $Y$ behaves just like $X$, so
  \begin{equation} \label{eq:G_IM:Feller data of revived BB on star graph, proof II}
  \begin{aligned}
    \EV_v ( \t^Y_\e \,;\, R^0 \leq \t^Y_\e < R^1 )
     & = \EV_v ( \t^X_\e \,;\, \t^X_\e < \z^X ) \\
     & = \EV_v ( \t^X_\e \,|\, \t^X_\e < \z^X ) \, \PV_v ( \t^X_\e < \z^X ).
  \end{aligned}   
  \end{equation}
 After the $n$-th revival, we are using the strong Markov property of $Y$ together with 
 \begin{align*}
  \EV_x \big( f(Y_{R^n}) \, \big| \, \sF_{R^n-} \big) = \int f \, d \e_v = f(v)  \quad \text{on $\{R^n < \infty \}$}
 \end{align*}
 (by the revival formula, cf.~\cite[Theorem~1.4]{WernerConcat}) and the definition of the revival kernel
 in order to compute
  \begin{align*}
   & \EV_v \big( \t^Y_\e \,;\, R^{n} \leq \t^Y_\e < R^{n+1} \big) \\
   & = \EV_v \big( \t^Y_\e \circ \T_{R^n} + R^n \,;\, R^{n} \leq \t^Y_\e , \t^Y_\e \circ \T_{R^n} < R^{n+1} \circ \T_{R^n} \big) \\
   & = \EV_v \big( \EV_{Y_{R^n}} ( \t^Y_\e \,;\, \t^Y_\e < R^{n+1} ) + R^n \, \PV_{Y_{R^n}} ( \t^Y_\e < R^{n+1} ) \,;\,  R^{n} \leq \t^Y_\e \big) \\
   & = \EV_v \big( \EV_v ( \t^Y_\e \,;\, \t^Y_\e < R^1 ) + R^n \, \PV_v( \t^Y_\e < R^1 ) \,;\,  R^{n} \leq \t^Y_\e \big) \\  
   & = \EV_v ( \t^X_\e \,|\, \t^X_\e < \z^X ) \, \PV_v ( \t^X_\e < \z^X ) \, \PV_v(R^{n} \leq \t^Y_\e ) \\
   & \quad  + \EV_v ( R^n  \,;\,  R^{n} \leq \t^Y_\e ) \,  \PV_v( \t^X_\e < \z^X ),
   \end{align*}
 where we used equation \eqref{eq:G_IM:Feller data of revived BB on star graph, proof II}
 as well as the relation $\PV_v( \t^Y_\e < R^1 ) = \PV_v( \t^X_\e < \z^X )$ for the last identity.
 It remains to show that
  \begin{align} \label{eq:G_IM:Feller data of revived BB on star graph, proof III}
    \PV_v(R^{n} \leq \t^Y_\e) = \PV_v( \t^X_\e = \z^X )^n
  \end{align}
 and 
  \begin{align} \label{eq:G_IM:Feller data of revived BB on star graph, proof IV}
   \EV_v ( R^n  \,;\,  R^{n} \leq \t^Y_\e ) = n \, \EV_v( \t^X_\e \,|\, \t^X_\e = \z^X ) \, \PV_v( \t^X_\e = \z^X )^n
  \end{align}
 for all $n \in \N_0$, which will be done inductively:
 For equation~\eqref{eq:G_IM:Feller data of revived BB on star graph, proof III}, the cases $n = 0$ and $n = 1$ are clear, and
 employing the same techniques as above, we conclude that
  \begin{align*}
   \PV_v \big( R^{n+1} \leq \t^Y_\e \big) 
   & = \PV_v \big( R^{n+1} \circ \T_{R^n} \leq \t^Y_\e \circ \T_{R^n} , R^{n} \leq \t^Y_\e \big) \\
   & = \EV_v \big( \PV_v (R^1 \leq \t^Y_\e) \,;\,  R^{n} \leq \t^Y_\e \big) \\
   & = \PV_v (R^1 \leq \t^Y_\e) \, \PV_v ( R^{n} \leq \t^Y_\e ) \\
   & = \PV_v( \t^X_\e = \z^X )^{n+1}.
  \end{align*}
 For equation~\eqref{eq:G_IM:Feller data of revived BB on star graph, proof IV}, the case $n = 0$ is again clear, and $n = 1$ is straightforward, as
  \begin{align*}
   \EV_v ( R^1  \,;\,  R^1 \leq \t^Y_\e )
   & = \EV_v ( \z^X  \,;\,  \z^X \leq \t^X_\e ) \\
   & = \EV_v ( \t^X_\e  \,;\,  \t^X_\e  = \z^X) \\
   & = \EV_v ( \t^X_\e  \,|\,  \t^X_\e =  \z^X) \, \PV_v ( \t^X_\e  = \z^X ).
  \end{align*}  
 The general case requires the same course of actions: 
 It is
  \begin{align*}
   & \EV_v \big( R^{n+1}  \,;\,  R^{n+1} \leq \t^Y_\e \big) \\
   & = \EV_v \big( R^{n+1}  \circ \T_{R^n} + R^n \,;\,  R^{n+1} \circ \T_{R^n} \leq \t^Y_\e \circ \T_{R^n} , R^{n} \leq \t^Y_\e \big) \\
   & = \EV_v \big( \EV_v(R^1 \,;\,  R^1 \leq \t^Y_\e) + R^n \, \PV_v(R^1 \leq \t^Y_\e) \,;\, R^{n} \leq \t^Y_\e \big),
  \end{align*}
 and using the inductive assumption for $\EV_v ( R^n  \,;\,  R^{n} \leq \t^Y_\e )$
 as well as the closed form~\eqref{eq:G_IM:Feller data of revived BB on star graph, proof III} for $\PV_v(R^{n} \leq \t^Y_\e)$ yields
  \begin{align*}
  \EV_v ( R^{n+1}  \,;\,  R^{n+1} \leq \t^Y_\e )
   & = \EV_v ( \t^X_\e  \,|\,  \t^X_\e =  \z^X) \, \PV_v ( \t^X_\e  = \z^X ) \,  \PV_v( \t^X_\e = \z^X )^n \\
   & \quad   + \PV_v( \t^X_\e = \z^X ) \, n \, \EV_v( \t^X_\e \,|\, \t^X_\e = \z^X ) \, \PV_v( \t^X_\e = \z^X )^n.
  \end{align*} 
 This finishes the proof of equations \eqref{eq:G_IM:Feller data of revived BB on star graph, proof I} and \eqref{eq:G_IM:Feller data of revived BB on star graph, proof Ia}.
 
 Next, we need to compare the exit distributions from $\overline{\BB_\e(v)}$ of $Y$ with the ones of $X$:
 If $X$ does not exit by jumping to $\D$, then $Y$ exits exactly like $X$:
  \begin{align} \label{eq:G_IM:Feller data of revived BB on star graph, proof V}
   \PV_v ( Y_{\t^Y_\e} \in B ) 
   = \PV_v ( X_{\t^X_\e} \in B \,|\, \t^X_\e < \z^X ),   \quad B \in \sB(\cG).
  \end{align}
 The rigorous proof of this claim is not very complicated: Decomposing the probability on the left-hand side via the revival times gives
  \begin{align*}
   \PV_v ( Y_{\t^Y_\e} \in B \,;\, \t^Y_\e < R^1) 
   & = \PV_v ( X_{\t^X_\e} \in B \,;\, \t^X_\e < \z^X).
  \end{align*}
 As $\t^Y_\e \circ R^n = \t^Y_\e - R^n$ on $\{ \t^Y_\e > R^n \}$, it follows that
  \begin{align*}
   & \PV_v \big( Y_{\t^Y_\e} \in B  \,;\, R^n < \t^Y_\e < R^{n+1} \big) \\
   & = \PV_v \big( Y_{\t^Y_\e} \circ \T_{R^n} \in B  \,;\, R^n < \t^Y_\e, \t^Y_\e \circ \T_{R^n} < R^{n+1} \circ \T_{R^n} \big) \\
   & = \EV_v \big( \PV_v ( Y_{\t^Y_\e} \in B \,;\, \t^Y_\e < R^1) \,;\, R^n < \t^Y_\e \big) \\
   & = \PV_v \big( X_{\t^X_\e} \in B \,;\, \t^X_\e < \z^X \big) \, \PV_v( \t^X_\e = \z^X )^n,
  \end{align*}
 where we used $Y_t = X_t$ for all $t < R^1$ as well as equation~\eqref{eq:G_IM:Feller data of revived BB on star graph, proof III} for the last identity.
 As $\t^Y_\e \neq R^n$ for all $n \in \N_0$, this proves equation~\eqref{eq:G_IM:Feller data of revived BB on star graph, proof V}, because
  \begin{align*}
   \PV_v ( Y_{\t^Y_\e} \in B ) 
   & = \PV_v ( X_{\t^X_\e} \in B \,;\, \t^X_\e < \z^X) \, \sum_{n \in \N_0} \PV_v( \t^X_\e = \z^X )^n \\
   & = \PV_v ( X_{\t^X_\e} \in B \,;\, \t^X_\e < \z^X) \, \frac{1}{\PV_v( \t^X_\e < \z^X )}.
  \end{align*}
  
 In order to calculate the domain of the generator $A^Y$ of $Y$, we need to reiterate the proof of Feller's theorem~\ref{theo:G_BM:Feller data}:
 Because $v$ is not a trap, Lemma~\ref{lem:A_FP:trap exit expectation} asserts that $\EV_v(\t^X_\e) < +\infty$ for all sufficiently small $\e > 0$,
 and therefore, $\EV_v(\t^Y_\e) < +\infty$ by equation~\eqref{eq:G_IM:Feller data of revived BB on star graph, proof Ia}.
 Furthermore, as seen in Lemma~\ref{lem:G_IM:Feller resolvent of revived X}, $Y$ is Feller, so Dynkin's formula is applicable for
 any $f \in \sD(A^Y)$. Then, as $Y$ cannot jump to $\D$ at all,
  \begin{align*}
   A^Y f(v) 
   & = \lim_{\e \downarrow 0} \frac{\EV_v \big( f(Y_{\t^Y_\e}) \big) - f(v)}{\EV_v(\t^Y_\e)} \\
   & = \lim_{\e \downarrow 0} \int_{\cG \bs \{v\}} \big( f(g) - f(v) \big) \frac{\PV_v ( Y_{\t^Y_\e} \in dg ) }{\EV_v(\t^Y_\e)},
  \end{align*}
 and inserting equations \eqref{eq:G_IM:Feller data of revived BB on star graph, proof V}, \eqref{eq:G_IM:Feller data of revived BB on star graph, proof Ia}
 and the measure $\nu^X_\e = \nu^v_\e$, as defined in Feller's theorem~\ref{theo:G_BM:Feller data} for the process $X$, yields
 \begin{align*}
   A^Y f(v) 
   & = \lim_{\e \downarrow 0} \int_{\cG \bs \{v\}} \big( f(g) - f(v) \big) \frac{\PV_v ( X_{\t^X_\e} \in dg ) / \PV_v( \t^X_\e < \z^X )}{\EV_v(\t^X_\e) / \PV_v( \t^X_\e < \z^X )} \\
   & = \lim_{\e \downarrow 0} \int_{\cG \bs \{v\}} \big( f(g) - f(v) \big) \, \nu^X_\e(dg).
  \end{align*}
 By now exactly following the proof of Feller's theorem for the process $Y$, but using 
  \begin{align*}
   K^X_\e = 1 + \frac{\PV_v(X_{\t^X_\e} = \D)}{\EV_v(\t^X_\e)} + \int_{\cG \bs \{v\}} \big( 1 - e^{-d(v,g)} \big) \, \nu^X_{\e} (dg)
  \end{align*}
 instead of the normalization $K^Y_\e$ (where the second summand would be missing), we get
  \begin{align*}
    c^{\infty}_1 f(v) - \sum_{e \in \cE} c^{e}_2 f_e'(v) + \frac{c_3}{2} f''(v) - \int_{\cG \bs \{v\}} \big( f(g) - f(v) \big) \, c_4(dg) = 0.
  \end{align*}
 In comparison to the boundary condition of $A^X$, the only term missing is $c^{\D}_1$, which is due to $\PV_v(Y_{\t^Y_\e} = \D) = 0$.
\end{proof}

We quickly remark that the boundary conditions in 
Lemmas~\ref{lem:G_IM:Feller resolvent of revived X} and \ref{lem:G_IM:Feller data of revived BB on star graph}
are not normalized anymore if $c^{\D}_1 \neq 0$, but can always be renormalized if needed.

We have thus shown that, when reviving a Brownian motion, the killing parameter $c_1$ or $c_1^\D$ transforms into a jump part:
The resolvent calculation (Lemma~\ref{lem:G_IM:Feller resolvent of revived X}) proves that $c_1$ is completely transformed,
while the approach via Feller's theorem (Lemma~\ref{lem:G_IM:Feller data of revived BB on star graph}) only transforms $c_1^\D$ and leaves $c_1^\infty$
as ``killing portion'' intact. We are utilizing this discrepancy order to prove Theorem~\ref{theo:G_IM:Feller data c_infty vanishes}:

\begin{proof}[Proof of Theorem~\ref{theo:G_IM:Feller data c_infty vanishes}]
 Let $X$ be a Brownian motion on a star graph $\cG$ with vertex~$v$, 
 and let the boundary condition of $X$ at~$v$ be described by $\big( c^{\D}_1, c^{\infty}_1, (c_2^e, e \in \cE), c_3, c_4 \Big)$ as given in Feller's theorem~\ref{theo:G_BM:Feller data}. 
 By~\cite[Lemma~2.6]{WernerStar}, the generator of $X$ reads
  \begin{align*}
   \sD(A^X)  = 
       \Big\{ & f \in \cC^2_0(\cG): \\ 
              &       c_1 f(v) - \sum_{e \in \cE} c^{e}_2 f_e'(v) + \frac{c_3}{2} f''(v) - \int_{\cG \bs \{v\}} \big( f(g) - f(v) \big) \, c_4(dg) = 0 \Big\}. 
  \end{align*}
  
 Define
  \begin{align*}
   \tilde{s} :=  \sum_{e \in \cE} c^{e}_2 + \frac{c_3}{2} + \int_{\cG \bs \{v\}} \big( 1 - e^{-d(v,g)} \big) \, c_4(dg).
  \end{align*}
 By recalling equation~\eqref{eq:G_BM:infinite jumpmeasure}, we see that $\tilde{s} > 0$.
   
 Assume $c^{\infty}_1 \neq 0$. 
 Consider the instant return process $Y$ of $X$, that is the identical copies process 
 resulting from successive revivals of $X$ at the killing point $v$.
 Lemma~\ref{lem:G_IM:Feller resolvent of revived X} applied with the revival distribution $q = \e_v$ gives
  \begin{align*}
   \sD(A^Y) = 
     \Big\{ & f \in \cC^2_0(\cG):  \\
            & - \sum_{e \in \cE} \tilde{c}^{e}_2 f_e'(v) + \frac{\tilde{c}_3}{2} f''(v) - \int_{\cG \bs \{v\}} \big( f(g) - f(v) \big) \, \tilde{c}_4(dg) = 0 \Big\},
  \end{align*}
 with renormalized boundary weights
  \begin{align*}
   \forall e \in \cE:~ \tilde{c}_2^e := \tilde{s}^{-1} \, c_2^e, \quad
   \tilde{c}_3 := \tilde{s}^{-1} \, c_3, \quad 
   \tilde{c}_4 := \tilde{s}^{-1} \, c_4.
  \end{align*}
 On the other hand, it is $c_1 \geq c^{\infty}_1  > 0$. Lemma~\ref{lem:G_IM:Feller data of revived BB on star graph} is applicable and shows that
  \begin{align*}
   \sD(A^Y) = 
     \Big\{ & f \in \cC^2_0(\cG): \\
            & \doubletilde{c}{}^\infty_1 f(v) - \sum_{e \in \cE} \doubletilde{c}{}^{e}_2 f_e'(v) + \frac{\doubletilde{c}_3}{2} f''(v) - \int_{\cG \bs \{v\}} \big( f(g) - f(v) \big) \, \doubletilde{c}_4(dg) = 0 \Big\},
  \end{align*} 
 with renormalized boundary weights
 \begin{align*}
  \doubletilde{c}{}^\infty_1 := \doubletilde{s}{}^{-1} \, c^{\infty}_1, \quad 
  \forall e \in \cE:~ \doubletilde{c}{}_2^e := \doubletilde{s}{}^{-1} \, c_2^e, \quad 
  \doubletilde{c}_3 := \doubletilde{s}{}^{-1} \, c_3, \quad 
  \doubletilde{c}_4 := \doubletilde{s}{}^{-1} \, c_4,
 \end{align*}
 for $\doubletilde{s} :=  c^{\infty}_1 + \tilde{s}$.
 
 As both of the above sets of $\sD(A^Y)$ are equal, Lemma~\ref{lem:G_IM:Feller data is unique} yields $\doubletilde{c}{}^\infty_1 = 0$, 
 which contradicts the assumption, as $\doubletilde{s} \, \doubletilde{c}{}^\infty_1 = c^{\infty}_1 > 0$.
\end{proof}

\begin{appendix}

\section{Metric Graphs} \label{app:metrix graphs}

In this appendix, we give a full, rigorous definition of metric graphs and functions thereon, followed
by the discussion of loops and 
a method of compactification, which will be needed for Theorem~\ref{theo:G_BM:Feller data} on the characterization of Brownian motions.

Following the common notion, a graph is a collection of two (disjoint) entities, called the set of vertices $\cV$ and the set of edges $\cL$,
whereby one vertex $\cLV(l)$ or two vertices $\big( \cLV_-(l), \cLV_+(l) \big)$ are assigned to each edge~$l \in \cL$ as its ``endpoint(s)'',
building up the graph's combinatorial structure.
When also assigning to each edge~$l \in \cL$ a positive length $\cR(l)$ (being $+\infty$ in case of $l$ having only one ``endpoint'') 
and thus identifying $l$ with some interval $[0, \cR(l)]$ ($[0, +\infty)$ in the case $\cR(l) = +\infty$),
it is possible to examine the resulting metric graph as a locally one-dimensional structure of
subintervals of $\R_+$, which are ``glued together'' at their respective endpoints.
This introduces the metric of ``shortest paths'' on this graph: Inside an edge, the metric will conform locally to the Euclidean distance on $\R$,
while the distance between points on different edges will be measured by the shortest path along the edges of the graph leading from
one point to the other.

By the identification of edges with intervals, the order of $\R_+$ introduces a ``orientation'' on the graph, which we will implement in the following way:
For an ``internal'' edge $l \in \cL$ with two endpoints $\big( \cLV_-(l), \cLV_+(l) \big)$, the ``initial point'' $0$ of the respective edge interval $[0, \cR(l)]$
will be identified with $\cLV_-(l)$, and the ``final point''~$\cR(l)$ with $\cLV_+(l)$.
For an ``external'' edge $l \in \cL$ with only one endpoint $\cLV(l)$, the ``initial point'' $0$ of its edge interval $[0, +\infty)$ will be equal to~$\cLV(l)$.
Despite of this ``orientation'' of the underlying intervals, we will only consider ``undirected graphs'' in the classical sense of this term, that is, paths along the edges are always
allowed in both directions.

 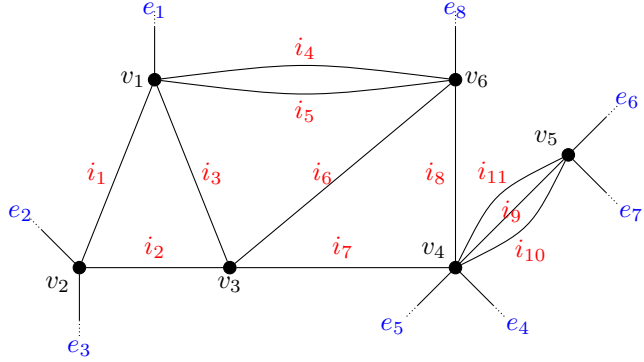
\begin{figure}[tb]
  \centering
  \begin{tikzpicture}[scale=0.5]
  \tikzstyle{every node}=[draw,shape=circle, fill, style={transform shape}];

  \node (A) at (2, 5) { };
  \node (B) at (0, 0) { };
  \node (C) at (4, 0) { };
  \node (D) at (10, 0) { };
  \node (E) at (13, 3) { };
  \node (F) at (10, 5) { }; 
  
  \tikzstyle{every node}=[];

  \node[left] (tA) at (A) {$v_1$};
  \node[below left] (tB) at (B) {$v_2$};
  \node[below] (tC) at (C) {$v_3$};
  \node[above left] (tD) at (D) {$v_4$};
  \node[above left] (tE) at (E) {$v_5$};
  \node[right] (tF) at (F) {$v_6$}; 
  
  \coordinate (A_1) at ($ (A) + (0, 1.4)$);
  \coordinate (B_1) at ($ (B) + (-1, 1)$);
  \coordinate (B_2) at ($ (B) + (0, -1.4)$);
  
  \coordinate (D_1) at ($ (D) + (1, -1)$);
  \coordinate (D_2) at ($ (D) + (-1, -1)$);
  \coordinate (E_1) at ($ (E) + (1, 1)$);
  \coordinate (E_2) at ($ (E) + (1, -1)$); 
  \coordinate (F_1) at ($ (F) + (0, 1.4)$); 

  \node[above, text=blue] (tA1) at (A_1) {$e_1$};
  \node[above left, text=blue] (tB1) at (B_1) {$e_2$};
  \node[below, text=blue] (tB2) at ($(B_2) - (0,0.3)$) {$e_3$};
  \node[below right, text=blue]  (tD1) at ($(D_1) + (0.1,-0.1)$) {$e_4$};
  \node[below left, text=blue] (tD2) at ($(D_2) - (0.1,0.1)$) {$e_5$};
  \node[above right, text=blue] (tE1) at (E_1) {$e_6$};
  \node[below right, text=blue] (tE2) at ($(E_2) + (0.1,-0.1)$) {$e_7$};
  \node[above, text=blue] (tF1) at (F_1) {$e_8$};

  \draw[-] (A) -- (B);
  \draw[-] (B) -- (C);
  \draw[-] (C) -- (A);
  
  \draw[-] (D) -- (E);
  \draw[-] (D) .. controls (12, 1) .. (E);  
  \draw[-] (D) .. controls (11, 2) .. (E);  
  \draw[-] (D) -- (F);

  \draw[-] (A) -- (A_1);
  \draw[densely dotted] (A_1) -- ($1.3*(A_1) - 0.3*(A)$);
  \draw[-] (B) -- (B_1);
  \draw[densely dotted] (B_1) -- ($1.3*(B_1) - 0.3*(B)$);
  \draw[-] (B) -- (B_2);
  \draw[densely dotted] (B_2) -- ($1.3*(B_2) - 0.3*(B)$);

  \draw[-] (D) -- (D_1);
  \draw[densely dotted] (D_1) -- ($1.3*(D_1) - 0.3*(D)$);
  \draw[-] (D) -- (D_2);
  \draw[densely dotted] (D_2) -- ($1.3*(D_2) - 0.3*(D)$);
  \draw[-] (E) -- (E_1);
  \draw[densely dotted] (E_1) -- ($1.3*(E_1) - 0.3*(E)$);
  \draw[-] (E) -- (E_2);  
  \draw[densely dotted] (E_2) -- ($1.3*(E_2) - 0.3*(E)$);
  \draw[-] (F) -- (F_1);
  \draw[densely dotted] (F_1) -- ($1.3*(F_1) - 0.3*(F)$);
  
  \draw[-] (C) -- (D);
  \draw[-] (C) -- (F);
  \draw[-] (A) .. controls ($0.5*(A)+0.5*(F) + (0,0.5)$) .. (F);  
  \draw[-] (A) .. controls ($0.5*(A)+0.5*(F) - (0,0.5)$) .. (F);

  \node[left, text=red] (iAB) at ($0.5*(A) + 0.5*(B)$) {$i_1$};
  \node[above, text=red] (iBC) at ($0.5*(B) + 0.5*(C)$) {$i_2$};
  \node[right, text=red] (iCA) at ($0.5*(A) + 0.5*(C)$) {$i_3$};
  \node[above, text=red] (iAF1) at ($0.5*(A) + 0.5*(F) + (0,0.3)$) {$i_4$};
  \node[below, text=red] (iAF2) at ($0.5*(A) + 0.5*(F) - (0,0.3)$) {$i_5$};
  \node[left, text=red] (iCF) at ($0.5*(C) + 0.5*(F)$) {$i_6$};
  \node[above, text=red] (iCD) at ($0.5*(C) + 0.5*(D)$) {$i_7$};
  \node[left, text=red] (iDF) at ($0.5*(D) + 0.5*(F)$) {$i_8$};
  
  \node[text=red] (iDE1) at ($0.5*(D) + 0.5*(E)$) {$i_9$};
  \node[below, text=red] (iDE2) at (12, 1) {$i_{10}$};
  \node[above, text=red] (iDE3) at (11, 2) {$i_{11}$};

\end{tikzpicture}
  \caption[A metric graph with $6$ vertices, $11$ internal edges, $8$ external edges]
          {A metric graph with $6$ vertices, $11$ internal edges, $8$ external edges.
           Here the curved lines are only used for illustrative reasons, they should still be considered as ``straight lines'' $[0, \cR(i)]$, $i \in \cI$.
           The ``orientation'' of the edges is not depicted here: e.g., if $\cLV(i_1) = \big( \cLV_-(i_1), \cLV_+(i_1) \big) = (v_1, v_2)$, then 
           $\big(i_1, 0 \big) \equiv v_1$ and $\big(i_1, \cR(i_1)\big) \equiv v_2$, or else $\big(i_1, 0 \big) \equiv v_2$ and $\big(i_1, \cR(i_1)\big) \equiv v_1$.}
           \label{fig:G_MG:metric graph}
 \end{figure}

\subsection{Basic Definitions} \label{subsec:G_MG:def}

An unified definition or notation for metric graphs does not appear to exist.
Classically, they originate in the context of ``quantum graphs'', see, e.g., \cite{Kuchment13}.
We follow the similar notational basis of~\cite{KostrykinSchrader06},
which Kostrykin, Potthoff and Schrader also use in their works \cite{KPS_Walsh12A}, \cite{KPS_Walsh12B}, \cite{KPS12}
on (continuous) Brownian motions on metric graphs. 
Observe that we will only consider \textit{finite} graphs, in the sense that the sets of vertices and edges will always be finite sets:

\begin{definition} \label{def:G_MG:graph}
 A tuple $\cG = (\cV, \cI, \cE, \cLV)$ is a \textdef{graph}, if $\cV \neq \emptyset$, $\cI$ and $\cE$ are finite, pairwise disjoint sets,
 and $\cLV$ is a map from the set $\cL := \cI \cup \cE$ into $(\cV \times \cV) \cup \cV$, such that $\cLV(e) \in \cV$ for all $e \in \cE$
 and $\cLV(i) = \big( \cLV_-(i), \cLV_+(i) \big) \in \cV \times \cV$ for all $i \in \cI$.
 $\cV$ is called the set of \textdef{vertices}, elements of $\cI$ and $\cE$ are called \textdef{internal edges} and \textdef{external edges}, $\cL$ is the set of all \textdef{edges}.
 For an internal edge $i$, $\cLV_-(i)$ and $\cLV_+(i)$ are called the \textdef{initial vertex} and \textdef{final vertex} of $i$,
 while for an external edge $e$, $\cLV(e)$ is the \textdef{initial vertex} of $e$.
 An internal edge~$i$ is called \textdef{loop}, if $\cLV_-(i) = \cLV_+(i)$.
 
 For a vertex $v \in \cV$, we define the sets
  \begin{align*}
   \cI_-(v) & := \big\{ i \in \cI: \cLV_-(i) = v \big\}, \quad 
   \cI_+(v) := \big\{ i \in \cI: \cLV_+(i) = v \big\}, \\
   \cI(v) & := \cI_-(v) \cup \cI_+(v), \\ 
   \cE(v) & := \big\{ e \in \cE: \cLV(e) = v \big\}, \\
   \cL(v) & := \cI(v) \cup \cE(v)
  \end{align*}
 of (initial, final) internal edges, external edges, all edges respectively, incident with~$v$.
\end{definition}

Whenever it is notationally convenient, we will also write $\cLV(l)$ for the set containing the vertex/vertices incident with the edge $l \in \cL$,
that is, $v \in \cLV(l)$ means $v \in \{ \cLV_-(l), \cLV_+(l) \}$ for an internal edge $l$, and $v \in \{ \cLV(l) \}$ for an external edge $l$.

\begin{definition} \label{def:G_MG:metric graph}
 Let $\cG = (\cV, \cI, \cE, \cLV)$ be a graph and $\cR \colon \cL \rightarrow (0, +\infty]$ be a map, such that $\cR(i) < +\infty$ for all $i \in \cI$,
 and $\cR(e) = +\infty$ for all $e \in \cE$. Then $(\cG, \cR)$ is called \textdef{metric graph}. For every edge $l \in \cL$, $\cR_l := \cR(l)$ is called \textdef{length} of $l$.
\end{definition}

The lengths of the edges and the graph's combinatorial structure 
induce the metric of the shortest paths on a metric graph $(\cV, \cI, \cE, \cLV, \cR)$,
which we will introduce rigorously next. To this end, consider 
 \begin{align*}
  \tcG = \cV \cup \bigcup_{i \in \cI} \big( \{i\} \times [0, \cR_i] \big) \cup \bigcup_{e \in \cE} \big( \{e\} \times [0, +\infty) \big). 
 \end{align*}
We extend the mapping $\cLV$ to $\tcG$ by setting $\cLV(v) := v$ for all $v \in \cV$ and $\cLV \big( (l,x) \big) := \cLV(l)$ 
for all~$(l,x) \in \tcG \bs \cV$.

The distance between two points inside the same edge can be measured by the Euclidean distance on $\R$, while
the distance of vertices can be measured by the length of the shortest possible path along the edges of the graph.
In order to distinguish both modes, we first define an auxiliary metric which only measures the direct distance inside the same edge:

\begin{definition} \label{def:G_MG:internal length}
 The \textdef{internal length} $\dint \colon \tcG \rightarrow [0, +\infty]$ is defined by
  \begin{align*}
   \forall e \in \cE, ~ x,y \in [0, +\infty): & \quad  \dint \big( (e,x), (e,y) \big) := \abs{x-y}, \\
   \forall i \in \cI, ~ x,y \in [0, \cR_i]:   & \quad  \dint \big( (i,x), (i,y) \big) := \abs{x-y}, \\
   \forall e \in \cE, ~ x \in [0, +\infty): & \quad  \dint \big( (e,x), \cLV(e) \big) := \dint \big( \cLV(e), (e,x) \big) := x, \\
   \forall i \in \cI, ~ x \in [0, \cR_i]:   & \quad  \dint \big( (i,x), \cLV_-(i) \big) := \dint \big( \cLV_-(i), (i,x) \big) := x, \\
                                            & \quad  \dint \big( (i,x), \cLV_+(i) \big) := \dint \big( \cLV_+(i), (i,x) \big) := \cR_i - x, \\
   \forall v \in \cV:                       & \quad  \dint (v,v) := 0,
  \end{align*}
 and $\dint(g_1, g_2) := +\infty$ for all other $g_1, g_2 \in \tcG$.
\end{definition}

The following metric properties of $\dint$ are immediate from its definition:

\begin{lemma} \label{lem:G_MG:dint properties}
 The following assertions hold true:
 \begin{enumerate}
  \item $\dint(g,g) = 0$ for all $g \in \tcG$.                                                     \label{itm:G_MG:dint properties i}
  \item $\dint(g_1, g_2) = \dint(g_2, g_1)$ for all $g_1, g_2 \in \tcG$.                           \label{itm:G_MG:dint properties ii}
  \item $\dint(g_1, g_3) \leq \dint(g_1, g_2) + \dint(g_2, g_3)$ for all $g_1, g_2, g_3 \in \tcG$. \label{itm:G_MG:dint properties iii}
 \end{enumerate}
\end{lemma}

In order to measure the distance between points on different edges, we need to consider the possible paths along the edges of the graph,
leading from the initial or final vertices of their respective edges:

\begin{definition} \label{def:G_MG:path}
 For $n \in \N_0$, $v_0, \ldots, v_n \in \cV$, $i_1, \ldots, i_n \in \cI$, 
 the ordered set $(v_0, i_1, v_1, \ldots, v_{n-1}, i_n, v_n)$ is called \textdef{path from~$v_0$ to~$v_n$} of \textdef{length}~$n$
 \textdef{across}~$(v_0, \ldots, v_n)$ \textdef{via}~$(i_1, \ldots, i_n)$, if
  \begin{align*}
   \forall k \in \{1, \ldots, n\}: \quad v_{k-1} \in \cLV(i_k), v_k \in \cLV(i_k).
  \end{align*}
 For $v, w \in \cV$, $\cP(v,w)$ is the set of all paths from $v$ to $w$, and
 $\cP = \bigcup_{v,w \in \cV} \cP(v,w)$ is the set of all possible paths.
\end{definition}

Notice that there is always a path from a vertex $v_0$ to itself, namely the path~$(v_0)$, and 
every path can be reversed: If $(v_0, i_1, v_1, \ldots, v_{n-1}, i_n, v_n)$ is a path from~$v_0$ to~$v_n$,
then $(v_n, i_n, v_{n-1}, \ldots, v_1, i_1, v_0)$ is a path from~$v_n$ to~$v_0$. 
In particular, $\cP(v,v)$ is not empty and $\cP(v,w) = \cP(w,v)$ holds for any vertices $v,w \in \cV$.
It also follows directly from the definition that paths can be concatenated: If 
\begin{align*}
  (v, i_1, v_1, \ldots, v_{n-1}, i_n, w) \quad \text{and} \quad (w, j_1, w_1, \ldots, w_{n-1}, j_n, u)
\end{align*}
are paths from $v$ to $w$, from $w$ to $u$ respectively,
then 
\begin{align*}
 (v, i_1, v_1, \ldots, v_{n-1}, i_n, w, j_1, w_1, \ldots, w_{n-1}, j_n, u)
\end{align*}
is a path from $v$ to $u$. Thus, the relation 
of being connected by a path is an equivalence relation on $\cV$.

\begin{definition} \label{def:G_MG:path length}
 The \textdef{length} of a path $\dPV \colon \cP \rightarrow [0, +\infty]$ is defined by 
  \begin{align*}
   \dPV \big( (v_0, i_1, v_1, \ldots, v_{n-1}, i_n, v_n) \big) := \cR_{i_1} + \cdots + \cR_{i_n}.
  \end{align*}
\end{definition}

We are now able to define a metric on the metric graph, induced by its combinatorial structure and its edge lengths:

\begin{definition} \label{def:G_MG:metric of the shortest paths}
 The \textdef{metric of the shortest paths} $d \colon \tcG \times \tcG \rightarrow [0, +\infty]$ 
 on a metric graph $(\cV, \cI, \cE, \cLV, \cR)$ is defined 
 for $v, w \in \cV$ by
  \begin{align*}
   d(v,w) 
   & := \inf_{ (v, \ldots, w) \in \cP(v,w)} \dPV \big( (v, \ldots, w) \big),
  \end{align*}
 as well as for $(g_1, g_2) \in (\tcG \times \tcG) \bs (\cV \times \cV)$ by
  \begin{align*}
   d \big( g_1, g_2 \big)
    := \inf \big\{ 
              & \dint(g_1, g_2), \\
              & \inf_{\substack{\, v_1 \in \cLV(g_1), \\ v_2 \in \cLV(g_2)}} \{ \dint(g_1, v_1) + d(v_1, v_2) + \dint(v_2, g_2) \}
            \big\}.
  \end{align*}
\end{definition}

Here, as usual, we set $\inf \emptyset := +\infty$. Therefore, $d(g_1,g_2) = +\infty$ holds if and only if there is
no path from $g_1$ to $g_2$ along the edges of $\cG$.

 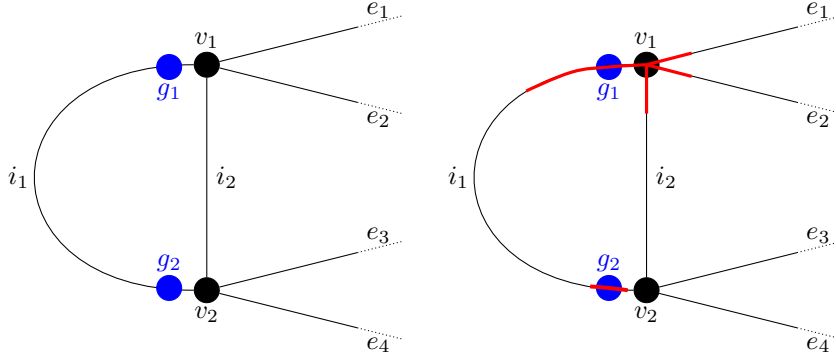
\begin{figure}[tb]
  \centering
  \begin{tikzpicture}
  \tikzstyle{every node}=[draw,shape=circle, fill, style={transform shape}];
  
  \node (D) at (0, 0) { };
  \node (E) at (0, 3) { };

  \tikzstyle{every node}=[];

  \node[below] (tD) at ($(D) + (0,-0.1)$) {$v_2$};
  \node[above] (tE) at ($(E) + (0,0.1)$) {$v_1$};
  
  
  \coordinate (D_1) at ($ (D) + (2, 0.5)$);
  \coordinate (D_2) at ($ (D) + (2, -0.5)$);
  \coordinate (E_1) at ($ (E) + (2, 0.5)$);
  \coordinate (E_2) at ($ (E) + (2, -0.5)$);

  \node[above right] (tD1) at (D_1) {$e_3$};
  \node[below right] (tD2) at (D_2) {$e_4$};
  \node[above right] (tE1) at (E_1) {$e_1$};
  \node[below right] (tE2) at (E_2) {$e_2$};

  \node[right] (tI2) at ($0.5*(E) + 0.5*(D)$) {$i_2$};  
  \node (tI1) at ($0.5*(E) + 0.5*(D) - (2.5,0)$) {$i_1$};

  \draw[-] (D) -- (E);
  \draw[-] (D) .. controls (-3,0) and (-3,3) .. (E);  

  \draw[-] (D) -- (D_1);
  \draw[densely dotted] (D_1) -- ($1.3*(D_1) - 0.3*(D)$);
  \draw[-] (D) -- (D_2);
  \draw[densely dotted] (D_2) -- ($1.3*(D_2) - 0.3*(D)$);
  \draw[-] (E) -- (E_1);
  \draw[densely dotted] (E_1) -- ($1.3*(E_1) - 0.3*(E)$);
  \draw[-] (E) -- (E_2);  
  \draw[densely dotted] (E_2) -- ($1.3*(E_2) - 0.3*(E)$);
  
  \tikzstyle{every node}=[draw,shape=circle, fill, style={transform shape}];
  \node[blue] (G1) at (-0.5,2.97) { };
  \node[blue] (G2) at (-0.5,0.03) { };
  
  \tikzstyle{every node}=[];
  \node[below, blue] at ($(G1) + (0,-0.1)$) {$g_1$};
  \node[above, blue] at ($(G2) + (0, 0.1)$) {$g_2$};
\end{tikzpicture}
  \begin{tikzpicture}
  \tikzstyle{every node}=[draw,shape=circle, fill, style={transform shape}];
  
  \node (D) at (0, 0) { };
  \node (E) at (0, 3) { };
  
  \node[blue] (G1) at (-0.5,2.97) { };
  \node[blue] (G2) at (-0.5,0.03) { };
  
  \tikzstyle{every node}=[];

  \node[below] (tD) at ($(D) + (0,-0.1)$) {$v_2$};
  \node[above] (tE) at ($(E) + (0,0.1)$) {$v_1$};
  
  
  \coordinate (D_1) at ($ (D) + (2, 0.5)$);
  \coordinate (D_2) at ($ (D) + (2, -0.5)$);
  \coordinate (E_1) at ($ (E) + (2, 0.5)$);
  \coordinate (E_2) at ($ (E) + (2, -0.5)$);

  \node[above right] (tD1) at (D_1) {$e_3$};
  \node[below right] (tD2) at (D_2) {$e_4$};
  \node[above right] (tE1) at (E_1) {$e_1$};
  \node[below right] (tE2) at (E_2) {$e_2$};

  \node[right] (tI2) at ($0.5*(E) + 0.5*(D)$) {$i_2$};  
  \node (tI1) at ($0.5*(E) + 0.5*(D) - (2.5,0)$) {$i_1$};

  \draw[-] (D) -- (E);
  \draw[-] (D) .. controls (-3,0) and (-3,3) .. (E);  

  \draw[-] (D) -- (D_1);
  \draw[densely dotted] (D_1) -- ($1.3*(D_1) - 0.3*(D)$);
  \draw[-] (D) -- (D_2);
  \draw[densely dotted] (D_2) -- ($1.3*(D_2) - 0.3*(D)$);
  \draw[-] (E) -- (E_1);
  \draw[densely dotted] (E_1) -- ($1.3*(E_1) - 0.3*(E)$);
  \draw[-] (E) -- (E_2);  
  \draw[densely dotted] (E_2) -- ($1.3*(E_2) - 0.3*(E)$);
  
  \node[below, blue] at ($(G1) + (0,-0.1)$) {$g_1$};
  \node[above, blue] at ($(G2) + (0, 0.1)$) {$g_2$};
  \draw[red, line width=1.2pt]  plot[smooth, tension=0.7] coordinates {(E) (G1) (-1,2.9) (-1.6, 2.65)};
  
  \draw[red, line width=1.2pt] (0, 3) -- (0.6,3.155);
  \draw[red, line width=1.2pt] (0, 3) -- (0.6,2.845);
  \draw[red, line width=1.2pt] (0, 3) -- (0,2.35);
  
  \draw[red, line width=1.5pt]  plot[smooth, tension=0.7] coordinates {(-0.25, 0) (G2) (-0.75,0.05)};
\end{tikzpicture}
  \caption[Shortest distance and neighborhoods in a metric graph]
          {Shortest distance and neighborhoods in a metric graph: Assume the edge lengths $\cR(i_1) = 10$, $\cR(i_2) = 5$
           and the points $g_1 = (i_1, 1)$, $g_2 = (i_2, 9)$. Then the internal distance inside the edge $i_1$ is given by $\dint(g_1, g_2) = 8$, 
           while the path across $(v_1, v_2)$ via $i_2$ 
           realizes the shortest distance $d(g_1, g_2) = 7$. On the right-hand graph, two neighborhoods of $g_1$ and $g_2$ are illustrated.}
           \label{fig:G_MG:shortest distance}
 \end{figure}

The reader should observe that the ``shortest path'' (and thus the distance) of two points inside the same edge must not equal 
the Euclidean distance of their local coordinates, cf.~figure~\ref{fig:G_MG:shortest distance} for an example. However, this will not cause any problems,
because the neighborhoods of points of the interior of an edge can always be chosen small enough in order to completely lie inside the corresponding edge.

We omit the easy but tedious proof of the metric properties for~$d$:

\begin{lemma} \label{lem:G_MG:d properties} 
 The following assertions hold true:
 \begin{enumerate}
  \item $d(g,g) = 0$ for all $g \in \tcG$.
  \item $d(g_1, g_2) = d(g_2, g_1)$ for all $g_1, g_2 \in \tcG$.
  \item $d(g_1, g_3) \leq d(g_1, g_2) + d(g_2, g_3)$ for all $g_1, g_2, g_3 \in \tcG$.
 \end{enumerate}
\end{lemma}

We introduce the \textdef{geometric representation} of the metric graph $(\cV, \cI, \cE, \cLV, \cR)$ by identifying the points which have zero distance:
 \begin{align} \label{eq:G_MG:geometric graph}
  \cG & := \bigslant{\tcG}{ \{ (g_1, g_2) \in \tcG: d(g_1, g_2) = 0 \}}.
 \end{align} 
The equivalence sets of $\cG$ are very simple here, as only the vertices are identified with the endpoints of their respective edges,
that is, we have the following classes of points:
 \begin{itemize}
  \item vertex points:
        $~{\{ v \} 
          \cup \big\{ (e, 0) : e \in \cE, v = \cLV(e) \big\}
          \cup \big\{ (i, 0): i \in \cI, v = \cLV_-(i) \big\}}$ \\
        ${\cup \, \big\{ (i, \cR_{i}): i \in \cI, v = \cLV_+(i) \big\}}$  for $v \in \cV$;
  \item inner points: $\big\{ (l, x) \big\}$ for $l \in \cL$, $x \in (0, \cR_l)$.
 \end{itemize}
Thus, $\cG$ can be seen as a collection of closed intervals and half lines of $\R$ of lengths given by $\cR$, with some of their endpoints being ``glued together'' by
the graph's combinatorial structure $\cLV$. 
We will call the ``position'' on these intervals $\{l\} \times [0, \cR_l]$ (with $[0, \cR_l] := [0, +\infty)$ if $\cR_l = +\infty$)
\textdef{local coordinate}, that is, a point $g = (l,x)$ has the local coordinate~$x$. Of course, this coordinate is only meaningful
in the context of its relative edge $l$, as the identification may ``glue together'' an ``initial'' coordinate $0$ of some edge 
with a ``final'' coordinate $\cR_i$ of some other edge $i$ at their shared vertex. 

We will identify any edge $l \in \cL$ with the set of its corresponding points $\{l\} \times [0, \cR_l]$.
For later use, we define the \textdef{open interior} of an edge $l \in \cL$
to be 
 \begin{align*}
  l^0 := \{ l \} \times ( 0, \cR_l ),
 \end{align*}
as well as the set $\cG^0$ of all inner points of $\cG$ by
 \begin{align*}
  \cG^0 := \bigcup_{l \in \cL} \big( \{l\} \times (0, \cR_l) \big) = \cG \bs \cV.
 \end{align*}

Owing to the triangle inequality of $d$ on $\tcG$, $d$ assumes the same value on all representants of an equivalence class.
Thus, it can be extended to a mapping $d \colon \cG \times \cG \rightarrow [0, +\infty]$ on the equivalence classes. 
It follows from Lemma~\ref{lem:G_MG:d properties} that $d$ is a metric on $\cG$.
Here we allow a metric to take~values in $[0, +\infty]$. 
This is a slight extension of the standard definition of a ``metric'', which
does not impact any topological results that will be needed later (see \cite[Chapter~1]{Burago01}).

The topology on $\cG$ induced by $d$ is structured as follows: Inside $\cG^0$, it locally ``looks'' like the topology of some interval of $\R_+$,
as for all $(l,x) \in \cG^0$, $\e \in \big( 0, \min \{ x, \cR_l - x \} \big)$,
 \begin{align*}
  \BB_\e \big( (l,x) \big)
  = \big\{ g \in \cG: d \big( (l,x), g \big) < \e \big\}
  = \{l\} \times (x-\e, x + \e),
 \end{align*}
which is ``glued together'' at the vertices by $\cLV$,
as for all $v \in \cV$, the ball around~$v$ with radius~$\e \in \big( 0, \min \{ \cR_l, l \in \cL(v) \} \big)$ is
 \begin{align*}
  \BB_\e (v)
  = \big\{ g \in \cG: d \big( v, g \big) < \e \big\}
  & =   \smashoperator[r]{\bigcup\limits_{\substack{l \in \cL(v) \\ \cLV_{-}(l) = v}}} \big( \{l\} \times [0, \e) \big)
   \cup \smashoperator[r]{\bigcup\limits_{\substack{l \in \cI(v) \\ \cLV_{+}(l) = v}}} \big( \{l\} \times (\cR_l - \e,\cR_l] \big).
 \end{align*}

\begin{theorem}
 $d$ defines a complete, separable metric on $\cG$.
\end{theorem}
\begin{proof}
 As every sequence in $\cG$ can be identified with a sequence in 
  \begin{align*}
   \bigcup_{i \in \cI} \big( \{i\} \times [0, \cR_i] \big) \cup \bigcup_{e \in \cE} \big( \{e\} \times [0, +\infty) \big),
  \end{align*}
 and each of the intervals $[0, \cR_i]$, $[0, +\infty)$ is complete,
 every Cauchy sequence in $\cG$ converges. Furthermore, every edge is homeomorphic to an interval,
 which contains a countable, dense subset (take, e.g., the rational points),
 and the topology of $\cG$ inside $\cG^0$ locally coincides with the internal topology induced on the edges,
 so
 using the (finite) union of these countable separability sets for all edges $l \in \cL$ together with the (finite) set of vertices gives a separability set for $\cG$.
\end{proof}

\subsection{Discussion of Loops} \label{subsec:G_MG:tadpoles}

Loops, that is internal edges $i \in \cI$ with the same initial and final vertex $\cLV_-(i) = \cLV_+(i)$, will provide a nuisance in our constructions.
The following technique, as described in \cite[Section~VI]{KPS12}, will allow us to eliminate the loops while maintaining
the graph's topological structure
(and thus, when applied in the context of Brownian motions, will not alter the description of the processes on the graph, see Remark~\ref{rem:G_BM:trivial vertices}).

Assume we are given a metric graph $\cG = (\cV, \cI, \cE, \cLV, \cR)$ with a non-empty set of loops $\cI_t = \{ i \in \cI: \cLV_-(i) = \cLV_+(i) \}$.
We split every loop into two ``regular'' internal edges by introducing,
for each $i \in \cI_t$, a new vertex $v^i_t$ and two new internal edges~$i^{+}$ and~$i^{-}$, each with edge length $\cR(i)/2$,
thus defining a new metric graph $\tcG = (\tcV, \tcI, \tcE, \tcLV, \tcR)$ with 
 $\tcV := \cV \cup \{ v^i_t: i \in \cI_t \}$,
 $\tcI := (\cI \bs \cI_t) \cup \{ i^{+}, i^{-} : i \in \cI_t \}$, and
 $\tcE := \cE$.
The edge lengths $\tcR$ and
the new graph's combinatorial structure $\tcLV$ are chosen to be
equal to the old ones $\cR$, $\cLV$ respectively, on the remaining original set~$(\cI \bs \cI_t) \cup \cE$, 
and are extended to the new edges
by
 $\tcR(i^{-}) := \tcR(i^{+}) := \cR(i)/2$ 
 and~$\tcLV(i^{-}) := \big( \cLV_-(i), v^i_t \big)$, $\tcLV(i^{+}) := \big( v^i_t, \cLV_+(i) \big)$, 
 for $i \in \cI_t$,
see figure~\ref{fig:G_MG:tadpoles}.

 \begin{figure}[tb]
  \centering
  \begin{tikzpicture}[scale=0.5]

  \draw[blue] (2,5.2) arc (0:360:1);
  \draw[blue] (12,5) arc (0:360:1);

  \tikzstyle{every node}=[draw,shape=circle, fill, style={transform shape}];

  \node (A) at (2, 5) { };  
  \node (B) at (1, 2) { };
  \node (C) at (4, 2) { };
  \node (D) at (10, 2) { };
  \node (E) at (13, 3) { };
  \node (F) at (10, 5) { };

  \tikzstyle{every node}=[];
  
  \node at (-2, 6) {\scalebox{1.2}{$\cG$}};

  \node[left] (tA) at (A) {$v_1$};
  \node[below left] (tB) at (B) {$v_2$};
  \node[below] (tC) at (C) {$v_3$};
  \node[above left] (tD) at (D) {$v_4$};
  \node[above left] (tE) at (E) {$v_5$};
  \node[right] (tF) at (F) {$v_6$};

  \node (A_tad) at (-0.3, 5.2) { };
  \node (F_tad) at (12.4, 5) { }; 
  
  \node[blue] (tA_tad) at (A_tad) {$i_1$};
  \node[blue] (tF_tad) at (F_tad) {$i_2$};
  
  \coordinate (A_1) at ($ (A) + (0, 1.4)$);
  \coordinate (B_1) at ($ (B) + (-1, 1)$);
  \coordinate (B_2) at ($ (B) + (0, -1.4)$);
  
  \coordinate (D_1) at ($ (D) + (1, -1)$);
  \coordinate (D_2) at ($ (D) + (-1, -1)$);
  \coordinate (E_1) at ($ (E) + (1, 1)$);
  \coordinate (E_2) at ($ (E) + (1, -1)$); 
  \coordinate (F_1) at ($ (F) + (0, 1.4)$); 


  \draw[-] (A) -- (B);
  \draw[-] (B) -- (C);
  \draw[-] (C) -- (A);
  
  \draw[-] (D) -- (E);
  \draw[-] (D) .. controls (12, 2) .. (E);  
  \draw[-] (D) .. controls (11, 3) .. (E);  
  \draw[-] (D) -- (F);

  \draw[-] (A) -- (A_1);
  \draw[densely dotted] (A_1) -- ($1.3*(A_1) - 0.3*(A)$);
  \draw[-] (B) -- (B_1);
  \draw[densely dotted] (B_1) -- ($1.3*(B_1) - 0.3*(B)$);
  \draw[-] (B) -- (B_2);
  \draw[densely dotted] (B_2) -- ($1.3*(B_2) - 0.3*(B)$);

  \draw[-] (D) -- (D_1);
  \draw[densely dotted] (D_1) -- ($1.3*(D_1) - 0.3*(D)$);
  \draw[-] (D) -- (D_2);
  \draw[densely dotted] (D_2) -- ($1.3*(D_2) - 0.3*(D)$);
  \draw[-] (E) -- (E_1);
  \draw[densely dotted] (E_1) -- ($1.3*(E_1) - 0.3*(E)$);
  \draw[-] (E) -- (E_2);  
  \draw[densely dotted] (E_2) -- ($1.3*(E_2) - 0.3*(E)$);
  \draw[-] (F) -- (F_1);
  \draw[densely dotted] (F_1) -- ($1.3*(F_1) - 0.3*(F)$);
  
  \draw[-] (C) -- (D);
  \draw[-] (C) -- (F);
  \draw[-] (A) .. controls ($0.5*(A)+0.5*(F) + (0,0.5)$) .. (F);  
  \draw[-] (A) .. controls ($0.5*(A)+0.5*(F) - (0,0.5)$) .. (F);

%

\end{tikzpicture}
  
  \begin{tikzpicture}[scale=0.5]

  \draw[red] (2,5.2) arc (0:180:1);
  \draw[blue] (2,5.2) arc (0:-180:1);
  
  \draw[blue] (12,5) arc (0:180:1);
  \draw[red] (12,5) arc (0:-180:1);

  \tikzstyle{every node}=[draw,shape=circle, fill, style={transform shape}];

  \node (A) at (2, 5) { };
  \node[red] (A_tad) at (0, 5.2) { };
  
  \node (B) at (1, 2) { };
  \node (C) at (4, 2) { };
  \node (D) at (10, 2) { };
  \node (E) at (13, 3) { };
  \node (F) at (10, 5) { }; 
  \node[red] (F_tad) at (12, 5) { }; 
  
  \tikzstyle{every node}=[];
  
  \node at (-2, 6) {\scalebox{1.2}{$\tcG$}};

  \node[left] (tA) at (A) {$v_1$};
  \node[below left] (tB) at (B) {$v_2$};
  \node[below] (tC) at (C) {$v_3$};
  \node[above left] (tD) at (D) {$v_4$};
  \node[above left] (tE) at (E) {$v_5$};
  \node[right] (tF) at (F) {$v_6$}; 
  
  \node[left, red] (tA_tad) at (A_tad) {$v^{i_1}_t$};
  \node[left,blue] (tA_tad1) at ($0.5*(A) + 0.5*(A_tad) - (0,1.2)$) {$i^+_1$}; 
  \node[above,red] (tA_tad2) at ($0.5*(A) + 0.5*(A_tad) + (0,0.9)$) {$i^-_1$}; 
  
  \node[right, red] (tF_tad) at (F_tad) {$v^{i_2}_t$};
  \node[below,red] (tF_tad1) at ($0.5*(F) + 0.5*(F_tad) - (-0.1,0.8)$) {$i^-_2$}; 
  \node[above,blue] (tF_tad2) at ($0.5*(F) + 0.5*(F_tad) + (0.1,0.8)$) {$i^+_2$}; 
  
  \coordinate (A_1) at ($ (A) + (0, 1.4)$);
  \coordinate (B_1) at ($ (B) + (-1, 1)$);
  \coordinate (B_2) at ($ (B) + (0, -1.4)$);
  
  \coordinate (D_1) at ($ (D) + (1, -1)$);
  \coordinate (D_2) at ($ (D) + (-1, -1)$);
  \coordinate (E_1) at ($ (E) + (1, 1)$);
  \coordinate (E_2) at ($ (E) + (1, -1)$); 
  \coordinate (F_1) at ($ (F) + (0, 1.4)$); 


  \draw[-] (A) -- (B);
  \draw[-] (B) -- (C);
  \draw[-] (C) -- (A);
  
  \draw[-] (D) -- (E);
  \draw[-] (D) .. controls (12, 2) .. (E);  
  \draw[-] (D) .. controls (11, 3) .. (E);  
  \draw[-] (D) -- (F);

  \draw[-] (A) -- (A_1);
  \draw[densely dotted] (A_1) -- ($1.3*(A_1) - 0.3*(A)$);
  \draw[-] (B) -- (B_1);
  \draw[densely dotted] (B_1) -- ($1.3*(B_1) - 0.3*(B)$);
  \draw[-] (B) -- (B_2);
  \draw[densely dotted] (B_2) -- ($1.3*(B_2) - 0.3*(B)$);

  \draw[-] (D) -- (D_1);
  \draw[densely dotted] (D_1) -- ($1.3*(D_1) - 0.3*(D)$);
  \draw[-] (D) -- (D_2);
  \draw[densely dotted] (D_2) -- ($1.3*(D_2) - 0.3*(D)$);
  \draw[-] (E) -- (E_1);
  \draw[densely dotted] (E_1) -- ($1.3*(E_1) - 0.3*(E)$);
  \draw[-] (E) -- (E_2);  
  \draw[densely dotted] (E_2) -- ($1.3*(E_2) - 0.3*(E)$);
  \draw[-] (F) -- (F_1);
  \draw[densely dotted] (F_1) -- ($1.3*(F_1) - 0.3*(F)$);
  
  \draw[-] (C) -- (D);
  \draw[-] (C) -- (F);
  \draw[-] (A) .. controls ($0.5*(A)+0.5*(F) + (0,0.5)$) .. (F);  
  \draw[-] (A) .. controls ($0.5*(A)+0.5*(F) - (0,0.5)$) .. (F);

%

\end{tikzpicture}
  \caption[Extension of a metric graph for elimination of loops]
          {Extension of a metric graph for elimination of loops: Pictured is a metric graph~$\cG$ with two loops $i_1$, $i_2$ at $v_1$, $v_6$.
           By splitting each loop~$i$ up into two new internal edges $i^-$, $i^+$, connected via the original vertex and a newly adjoined vertex~$v^i_t$,
           we obtain the resulting graph~$\tcG$ which does not possess any~loops.} \label{fig:G_MG:tadpoles}
 \end{figure}
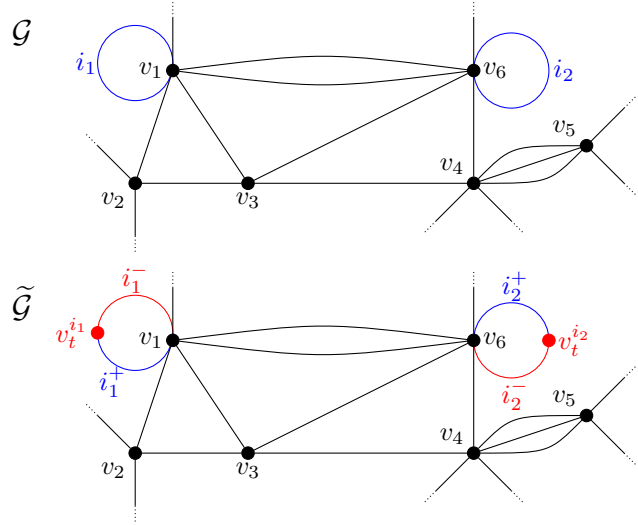 

Due to the identification of the new edges' endpoints with the adjoined vertices, 
and to the graph's metric being only dependent on the length of paths,
the induced topology on the new metric graph $\tcG$ equals the topology on $\cG$. 
$\tcG$ does not possess any loops.
Therefore, we will always be able to restrict our attention to metric graphs without loops in the sequel, as
all our examinations will solely be based on the topological structure of the underlying graph, but not on its representation.

\subsection{Functions on a Metric Graph} \label{subsec:G_MG:functions}

Any real valued function $f$ on a metric graph $\cG$ can be represented by collections of real~values $(f_v, v \in \cV)$ at the vertices $\cV$ and
of functions $(f_l, l \in \cL)$ on the edges $\cL$, with $f_l \colon [0, \cR_l] \rightarrow \R$
satisfying ${f_l(x) = f \big( (l,x) \big)}$, $x \in [0, \cR_l]$  (where in the following we set for notationally convenience $[0, \cR_l] := [0, +\infty)$ for $l \in \cE$),
and $f_v = f(v)$, $v \in \cV$. As the endpoints of the edges are identified by the graph's combinatorial structure, the values
 \begin{align*}
  f_e(0) = f \big( (e,0) \big), ~ f_v = f(v), ~ f_{i_-}(0) = f \big( (i_-,0) \big), ~ f_{i_+}(\cR_{i_+}) = f \big( (i_+, \cR_{i_+}) \big)
 \end{align*}
must coincide in case $e \in \cE$, $v = \cLV(e)$, and $i_- \in \cI$, $v = \cLV_-(i_-)$, and $i_+ \in \cI$, $v = \cLV_+(i_+)$.

In every small neighborhood of a non-vertex point $g \in \cG^0$, a real valued function~$f$ on~$\cG$ can locally be interpreted as a function on some one-dimensional interval.
Thus, the differentiability of $f_l$ at $x$ induces the notion of differentiability of~$f$ at~$g = (l,x) \in \cG^0$. In order to define differentiability at the vertices,
we must take care of the edges' ``orientation'':

\begin{definition}
 Let $f \colon \cG \rightarrow \R$ be a function on $\cG$, $v \in \cV$ and $l \in \cL(v)$.
 Then the \textdef{directional derivative of $f$ at $v$ along $l$} is defined by
  \begin{align*}
   f_l'(v) := 
    \begin{cases}
     \lim_{\xi \rightarrow v, \xi \in l^0} f'(\xi),   & v = \cLV_-(l), \\
     - \lim_{\xi \rightarrow v, \xi \in l^0} f'(\xi), & v = \cLV_+(l),
    \end{cases}
  \end{align*}
 whenever the right-hand side exists.
\end{definition}

\begin{definition} \label{def:G_MG:C2 functions}
 Let $\cC^{0,2}_0(\cG)$ be the subspace of all functions $f$ in $\cC_0(\cG)$, which are twice continuously differentiable on $\cG^0$, such that 
 for every $v \in \cV$, $l \in \cL(v)$, 
  \begin{align*}
   f_l''(v) := \lim_{\xi \rightarrow v, \xi \in l^0} f''(\xi)
  \end{align*}
 exists, and for every $e \in \cE$, $f_e''$ vanishes at infinity.
 Let $\cC^2_0(\cG)$ be the subset of those functions~$f$ in~$\cC^{0,2}_0(\cG)$ for which $f''$ extends from~$\cG^0$ to a function in~$\cC_0(\cG)$.
\end{definition}

By definition, a function $f \in \cC^{0,2}_0(\cG)$ lies in $\cC^2_0(\cG)$, if and only if for every $v \in \cV$,
the second derivatives at $v$ coincide, that is, if $f_k''(v) = f_l''(v)$ holds for all $k,l \in \cL(v)$,
and in this case, we will just write $f''(v)$ for this value. 
If $f \in \cC^2_0(\cG)$, then, for any edge $l \in \cL$, the limits of the first derivatives at its endpoint(s)
$\lim_{x \downdownarrows 0} f_l'(x)$ (and $\lim_{x \upuparrows \cR_l} f_l'(x)$, if~$l \in \cI$) must exist,
which can easily be seen by the fundamental theorem of calculus.
However, these limits on various edges do not need to coincide at their common vertex:
In general, the first derivate
$f'$ of $f \in \cC^2_0(\cG)$ does not extend from $\cC_0(\cG^0)$ to a function in $\cC_0(\cG)$.

We will mainly be concerned with the following operator on $\cC^2_0(\cG)$:

\begin{definition}
 The \textdef{Laplacian} $\D$ on $\cG$ is defined by 
  \begin{align*}
   \D \colon \cC^2_0(\cG) \rightarrow \cC_0(\cG), ~ f \mapsto \D(f) := f''.
  \end{align*}
\end{definition}

\subsection{Compactification of a Metric Graph} \label{subsec:MG:compactification}

We introduce the following method of ``cutting out'' vertex points from an existing graph and compactifying the resulting set.
This technique is used in the proof of~Theorem~\ref{theo:G_BM:Feller data}.

Let $(\cV, \cI, \cE, \cLV, \cR)$ be a metric graph with geometric representation
 \begin{align*}
  \tcG = \cV \cup \bigcup_{i \in \cI} \big( \{i\} \times [0, \cR_i] \big) \cup \bigcup_{e \in \cE} \big( \{e\} \times [0, +\infty) \big),
 \end{align*}
and $\cG$ be the set $\tcG$ with vertex points and endpoints of edges identified by its canonical metric~$d$, as introduced in subsection~\ref{subsec:G_MG:def}.
Let $\cV_0 \subsetneq \cV$, and $\tcG_1$ be the subset of~$\cG$ which results from removing the vertices $\cV_0$ together with their identified edge points from $\cG$,
that is, consider
 \begin{align*}
  \tcG_1  
  := & \ \tcG \bs \Big( \cV_0 \cup \bigcup_{i_- \in \cI_-(\cV_0)} \{(i_-,0)\}
                         \cup \bigcup_{i_+ \in \cI_+(\cV_0)} \{(i_+,\cR_i)\}
                         \cup \bigcup_{e \in \cE(\cV_0)}   \{(e,0)\}       \Big) \\
   = & \ \big( \cV \bs \cV_0 \big) 
      \cup \bigcup_{i \in \cI} \big( \{i\} \times I_i \big)
      \cup \bigcup_{e \in \cE} \big( \{e\} \times E_e \big)
 \end{align*}
with
 \begin{align*}
  I_i
  := \begin{cases}
     [0, \cR_i], & i \in \cI \bs \cI(\cV_0), \\
     (0, \cR_i], & i \in \cI_-(\cV_0) \bs \cI_+(\cV_0), \\
     [0, \cR_i), & i \in \cI_+(\cV_0) \bs \cI_-(\cV_0), \\
     (0, \cR_i), & i \in \cI_-(\cV_0) \cap \cI_+(\cV_0),
    \end{cases}
  \quad
  E_e
  := \begin{cases}
     [0, +\infty), & e \in \cE \bs \cE(\cV_0), \\
     (0, +\infty), & e \in \cE(\cV_0).
    \end{cases}
 \end{align*}
We compactify $\tcG_1$ by adjoining the missing interval endpoints $0$, $\cR_i$, $+\infty$, where needed.
For convenience (and for staying in the context of a metric graph as much as possible),
we also add new vertices for newly adjoined finite endpoints. Altogether, we set
 \begin{align*}
  \overbar{\tcG_1}
  := \cV_1
     \cup \bigcup_{i \in \cI} \big( \{i\} \times [0, \cR_i] \big)
     \cup \bigcup_{e \in \cE} \big( \{e\} \times [0, +\infty] \big),
 \end{align*}
with 
 \begin{align*}
  \cV_1 := \big( \cV \bs \cV_0 \big) \cup \big\{ v^i_-, i \in \cI_-(\cV_0) \big\} \cup \big\{ v^i_+, i \in \cI_+(\cV_0) \big\} \cup \big\{ v^e, e \in \cE(\cV_0) \big\},
 \end{align*}
where all new vertices $v^i_-, v^i_+, v^e$ are distinct points which are not in $\cG$. 
We adapt the combinatorial structure of the original graph to $\overbar{\tcG_1}$ by adjusting 
the mapping~$\cLV$ to~$\cLV_1 \colon \cL \rightarrow (\cV_1 \times \cV_1) \cup \cV_1$, defined~by
 \begin{align*}
  \cLV_1(i) 
  = \begin{cases}
     \big( \cLV_-(i), \cLV_+(i) \big), & i \in \cI \bs \cI(\cV_0), \\
     \big( v^i_-, \cLV_+(i) \big), & i \in \cI_-(\cV_0) \bs \cI_+(\cV_0), \\
     \big( \cLV_-(i),v^i_+ \big), & i \in \cI_+(\cV_0) \bs \cI_-(\cV_0), \\
     \big( v^i_-,v^i_+ \big), & i \in \cI_-(\cV_0) \cap \cI_+(\cV_0),
    \end{cases}  
  \quad
  \cLV_1(e)
  = \begin{cases}
     \cLV(e), & e \in \cE \bs \cE(\cV_0), \\
     v^e, & e \in \cE(\cV_0).
    \end{cases}
 \end{align*}

Thus, by removing vertices from the original graph $\cG$, we disconnected some edges which needed new initial or final vertices.
We added these, and additionally compactified the non-compact external edges $\{e\} \times [0, +\infty)$ to $\{e\} \times [0, +\infty]$.
Observe that the latter results in the ``compactified graph'' $\overbar{\tcG_1}$ not being a metric graph in the sense of our definition anymore.
 
Let $d_1$ be the metric of shortest paths, as defined in subsection~\ref{subsec:G_MG:def},
for the just constructed metric graph $\big( (\cV_1, \cI, \cE, \cLV_1), \cR \big)$.
We extend the metric $d_1$ to $\overbar{\tcG_1}$ by defining the distance of a point ``at infinity'' $(e, +\infty)$, $e \in \cE$, to any other point to be $+\infty$.
Then, as usual, we identify the points $g_1, g_2 \in \overbar{\tcG_1}$ for which $d_1(g_1, g_2) = 0$ holds true, naming the 
resulting set of equivalence sets $\overbar{\cG_1}$.

In order to be able to distinguish between the original vertex points of $\cG$ and the newly introduced ones of $\overbar{\cG_1}$ in the local representation,
we set
 \begin{itemize}
  \item if $i \in \cI_-(\cV_0)$: $(i, 0+)$ for $(i,0) = v^i_-$,
  \item if $i \in \cI_+(\cV_0)$: $(i, \cR_i-)$ for $(i,\cR_i) = v^i_+$,
  \item if $e \in \cE(\cV_0)$:  $(e, 0+)$ for $(e,0) = v^e$.
 \end{itemize}

Let the topology inside $\overbar{\cG_1} \bs \{(e, +\infty), e \in \cE\}$ be induced by $d_1$,
while for each $e \in \cE$, the point~$(e, +\infty)$ is a distinct point in the topology, 
topological inserted as the point at infinity of $\{e\} \times [0, +\infty)$
by the same technique the ``point at infinity''~$+\infty$ 
is embedded in $[0, +\infty)$ by the Alexandroff one-point compactification, that is,
as a point outside every compact set. 
 
Observe that by removing a vertex point $v$ and compactifying the resulting graph, the ``connection'' of all edges incident with $v$ is removed and 
a new endpoint is adjoint for each disconnected edge. Furthermore, every external edge $\{e\} \times [0, +\infty)$ is compactified to $\{e\} \times [0, +\infty]$,
thus adding points $(e, +\infty)$ for all external edges $e \in \cE$, see figure~\ref{fig:G_MG:compactification}.

 \begin{figure}[tb]
  \centering
  \begin{tikzpicture}[scale=0.5]
  \tikzstyle{every node}=[draw,shape=circle, fill, style={transform shape}];

  \node (A) at (2, 5) { };
  \node (B) at (0, 0) { };
  \node[red] (C) at (4, 0) { };
  \node (D) at (10, 0) { };
  \node (E) at (13, 3) { };
  \node (F) at (10, 5) { }; 
  
  \tikzstyle{every node}=[];

  \node at (-2, 6) {\scalebox{1.2}{$\cG$}};
  
  \node[left] (tA) at (A) {$v_1$};
  \node[below left] (tB) at (B) {$v_2$};
  \node[below, text=blue] (tC) at (C) {$v_3$};
  \node[above left] (tD) at (D) {$v_4$};
  \node[above left] (tE) at (E) {$v_5$};
  \node[right] (tF) at (F) {$v_6$}; 
  
  \coordinate (A_1) at ($ (A) + (0, 1.4)$);
  \coordinate (B_1) at ($ (B) + (-1, 1)$);
  \coordinate (B_2) at ($ (B) + (0, -1.4)$);
  
  \coordinate (D_1) at ($ (D) + (1, -1)$);
  \coordinate (D_2) at ($ (D) + (-1, -1)$);
  \coordinate (E_1) at ($ (E) + (1, 1)$);
  \coordinate (E_2) at ($ (E) + (1, -1)$); 
  \coordinate (F_1) at ($ (F) + (0, 1.4)$); 

  \node[above] (tA1) at (A_1) {$e_1$};
  \node[above left] (tB1) at (B_1) {$e_2$};
  \node[below] (tB2) at ($(B_2) - (0,0.3)$) {$e_3$};
  \node[below right]  (tD1) at ($(D_1) + (0.1,-0.1)$) {$e_4$};
  \node[below left] (tD2) at ($(D_2) - (0.1,0.1)$) {$e_5$};
  \node[above right] (tE1) at (E_1) {$e_6$};
  \node[below right] (tE2) at ($(E_2) + (0.1,-0.1)$) {$e_7$};
  \node[above] (tF1) at (F_1) {$e_8$};

  \draw[-] (A) -- (B);
  \draw[-] (B) -- (C);
  \draw[-] (C) -- (A);
  
  \draw[-] (D) -- (E);
  \draw[-] (D) .. controls (12, 1) .. (E);  
  \draw[-] (D) .. controls (11, 2) .. (E);  
  \draw[-] (D) -- (F);

  \draw[-] (A) -- (A_1);
  \draw[densely dotted] (A_1) -- ($1.3*(A_1) - 0.3*(A)$);
  \draw[-] (B) -- (B_1);
  \draw[densely dotted] (B_1) -- ($1.3*(B_1) - 0.3*(B)$);
  \draw[-] (B) -- (B_2);
  \draw[densely dotted] (B_2) -- ($1.3*(B_2) - 0.3*(B)$);

  \draw[-] (D) -- (D_1);
  \draw[densely dotted] (D_1) -- ($1.3*(D_1) - 0.3*(D)$);
  \draw[-] (D) -- (D_2);
  \draw[densely dotted] (D_2) -- ($1.3*(D_2) - 0.3*(D)$);
  \draw[-] (E) -- (E_1);
  \draw[densely dotted] (E_1) -- ($1.3*(E_1) - 0.3*(E)$);
  \draw[-] (E) -- (E_2);  
  \draw[densely dotted] (E_2) -- ($1.3*(E_2) - 0.3*(E)$);
  \draw[-] (F) -- (F_1);
  \draw[densely dotted] (F_1) -- ($1.3*(F_1) - 0.3*(F)$);
  
  \draw[-] (C) -- (D);
  \draw[-] (C) -- (F);
  \draw[-] (A) .. controls ($0.5*(A)+0.5*(F) + (0,0.5)$) .. (F);  
  \draw[-] (A) .. controls ($0.5*(A)+0.5*(F) - (0,0.5)$) .. (F);

  \node[left] (iAB) at ($0.5*(A) + 0.5*(B)$) {$i_1$};
  \node[above] (iBC) at ($0.5*(B) + 0.5*(C)$) {$i_2$};
  \node[right] (iCA) at ($0.5*(A) + 0.5*(C)$) {$i_3$};
  \node[above] (iAF1) at ($0.5*(A) + 0.5*(F) + (0,0.3)$) {$i_4$};
  \node[below] (iAF2) at ($0.5*(A) + 0.5*(F) - (0,0.3)$) {$i_5$};
  \node[left] (iCF) at ($0.5*(C) + 0.5*(F)$) {$i_6$};
  \node[above] (iCD) at ($0.5*(C) + 0.5*(D)$) {$i_7$};
  \node[left] (iDF) at ($0.5*(D) + 0.5*(F)$) {$i_8$};
  
  \node (iDE1) at ($0.5*(D) + 0.5*(E)$) {$i_9$};
  \node[below] (iDE2) at (12, 1) {$i_{10}$};
  \node[above] (iDE3) at (11, 2) {$i_{11}$};
\end{tikzpicture}
  
  \begin{tikzpicture}[scale=0.5]
  \tikzstyle{every node}=[draw,shape=circle, fill, style={transform shape}];

  \node (A) at (2, 5) { };
  \node (B) at (0, 0) { };
  \coordinate (C) at (4, 0) { };
  \node (D) at (10, 0) { };
  \node (E) at (13, 3) { };
  \node (F) at (10, 5) { }; 
  
  \node[red] (C1) at ($0.07*(A)+0.93*(C)$) {};
  \node[red] (C2) at ($0.07*(B)+0.93*(C)$) {};
  \node[red] (C3) at ($0.07*(D)+0.93*(C)$) {};
  \node[red] (C4) at ($0.07*(F)+0.93*(C)$) {};
  
  \tikzstyle{every node}=[];

  \node at (-3, 6) {\scalebox{1.2}{$\overbar{\cG_1}$}};
  
  \node[left] (tA) at (A) {$v_1$};
  \node[below left] (tB) at (B) {$v_2$};
  \node[above left] (tD) at (D) {$v_4$};
  \node[above left] (tE) at (E) {$v_5$};
  \node[right] (tF) at (F) {$v_6$};

  
  \tikzstyle{every node}=[text=red];
  \node[above left] (tC1) at (C1) {$v^{i_3}_-$};
  \node[below left] (tC2) at (C2) {$v^{i_2}_-$};
  \node[below right] (tC3) at (C3) {$v^{i_7}_-$};
  \node[above right] (tC4) at (C4) {$v^{i_6}_+$};
  
  \tikzstyle{every node}=[];
  
  \coordinate (A_1) at ($ (A) + (0, 1.4)$);
  \coordinate (B_1) at ($ (B) + (-1, 1)$);
  \coordinate (B_2) at ($ (B) + (0, -1.4)$);
  
  \coordinate (D_1) at ($ (D) + (1, -1)$) ;
  \coordinate (D_2) at ($ (D) + (-1, -1)$);
  \coordinate (E_1) at ($ (E) + (1, 1)$);
  \coordinate (E_2) at ($ (E) + (1, -1)$); 
  \coordinate (F_1) at ($ (F) + (0, 1.4)$);

  \draw[-] (A) -- (B);
  \draw[-] (B) -- (C2);
  \draw[-] (C1) -- (A);

  \draw[-] (A) -- (A_1);
  \draw[densely dotted] (A_1) -- ($1.3*(A_1) - 0.3*(A)$);
  \draw[-] (B) -- (B_1);
  \draw[densely dotted] (B_1) -- ($1.3*(B_1) - 0.3*(B)$);
  \draw[-] (B) -- (B_2);
  \draw[densely dotted] (B_2) -- ($1.3*(B_2) - 0.3*(B)$);

  \draw[-] (D) -- (D_1);
  \draw[densely dotted] (D_1) -- ($1.3*(D_1) - 0.3*(D)$);
  \draw[-] (D) -- (D_2);
  \draw[densely dotted] (D_2) -- ($1.3*(D_2) - 0.3*(D)$);
  \draw[-] (E) -- (E_1);
  \draw[densely dotted] (E_1) -- ($1.3*(E_1) - 0.3*(E)$);
  \draw[-] (E) -- (E_2);  
  \draw[densely dotted] (E_2) -- ($1.3*(E_2) - 0.3*(E)$);
  \draw[-] (F) -- (F_1);
  \draw[densely dotted] (F_1) -- ($1.3*(F_1) - 0.3*(F)$);
  
  \draw[-] (D) -- (E);
  \draw[-] (D) .. controls (12, 1) .. (E);  
  \draw[-] (D) .. controls (11, 2) .. (E);  
  \draw[-] (D) -- (F);
  
  \draw[-] (C3) -- (D);
  \draw[-] (C4) -- (F);
  \draw[-] (A) .. controls ($0.5*(A)+0.5*(F) + (0,0.5)$) .. (F);  
  \draw[-] (A) .. controls ($0.5*(A)+0.5*(F) - (0,0.5)$) .. (F);

  \node[left] (iAB) at ($0.5*(A) + 0.5*(B)$) {$i_1$};
  \node[above] (iBC) at ($0.5*(B) + 0.5*(C)$) {$i_2$};
  \node[right] (iCA) at ($0.5*(A) + 0.5*(C)$) {$i_3$};
  \node[above] (iAF1) at ($0.5*(A) + 0.5*(F) + (0,0.3)$) {$i_4$};
  \node[below] (iAF2) at ($0.5*(A) + 0.5*(F) - (0,0.3)$) {$i_5$};
  \node[left] (iCF) at ($0.5*(C) + 0.5*(F)$) {$i_6$};
  \node[above] (iCD) at ($0.5*(C) + 0.5*(D)$) {$i_7$};
  \node[left] (iDF) at ($0.5*(D) + 0.5*(F)$) {$i_8$};
  
  \node (iDE1) at ($0.5*(D) + 0.5*(E)$) {$i_9$};
  \node[below] (iDE2) at (12, 1) {$i_{10}$};
  \node[above] (iDE3) at (11, 2) {$i_{11}$};

  \tikzstyle{every node}=[draw,shape=circle, fill, style={transform shape}, red];
  \node (eA1) at ($1.3*(A_1) - 0.3*(A)$) {};
  \node (eB1) at ($1.3*(B_1) - 0.3*(B)$) {};
  \node (eB2) at ($1.3*(B_2) - 0.3*(B)$) {};
  \node (eD1) at ($1.3*(D_1) - 0.3*(D)$) {};
  \node (eD2) at ($1.3*(D_2) - 0.3*(D)$) {};
  \node (eE1) at ($1.3*(E_1) - 0.3*(E)$) {};
  \node (eE2) at ($1.3*(E_2) - 0.3*(E)$) {};
  \node (eF1) at ($1.3*(F_1) - 0.3*(F)$) {};

  \tikzstyle{every node}=[red];
  \node[left] (teA1) at (eA1) {$\scriptstyle{(e_1, +\infty)}$};
  \node[above left] (teB1) at (eB1) {$\scriptstyle{(e_2, +\infty)}$};
  \node[left] (teB2) at (eB2) {$\scriptstyle{(e_3, +\infty)}$};
  \node[below right] (teD1) at (eD1) {$\scriptstyle{(e_4, +\infty)}$};
  \node[below left] (teD2) at (eD2) {$\scriptstyle{(e_5, +\infty)}$};
  \node[above right] (teE1) at (eE1) {$\scriptstyle{(e_6, +\infty)}$};
  \node[below right] (teE2) at (eE2) {$\scriptstyle{(e_7, +\infty)}$};
  \node[right] (teF1) at (eF1) {$\scriptstyle{(e_8, +\infty)}$};
  
  \tikzstyle{every node}=[];
  \node (tA1) at ($0.5*(teA1) + 0.5*(A)$) {$e_1$};
  \node[above] (tB1) at ($0.5*(B_1) + 0.5*(B)$) {$e_2$};
  \node (tB2) at ($0.5*(teB2) + 0.5*(B) + (1.1, 0)$) {$e_3$};
  \node[below] (tD1) at ($0.5*(D_1) + 0.5*(D)$) {$e_4$};
  \node[below] (tD2) at ($0.5*(D_2) + 0.5*(D)$) {$e_5$};
  \node[above] (tE1) at ($0.5*(E_1) + 0.5*(E)$) {$e_6$};
  \node[below] (tE2) at ($0.5*(E_2) + 0.5*(E)$) {$e_7$};
  \node (tF1) at ($0.5*(teF1) + 0.5*(F)$) {$e_8$}; 
\end{tikzpicture}
  \caption[A metric graph and its resulting compactification]
          {Illustration of a metric graph $\cG$ and its resulting compactification~$\overbar{\cG_1}$ when the vertex set~$\cV_0 := \{v_3\}$ is removed from~$\cG$.
           Here, $\cI_-(v_3) = \{ i_2, i_3, i_7 \}$, $\cI_+(v_3) = \{ i_6 \}$, $\cE(v_3) = \emptyset$. 
           The new points introduced by the compactification are depicted in red.} \label{fig:G_MG:compactification}
 \end{figure}
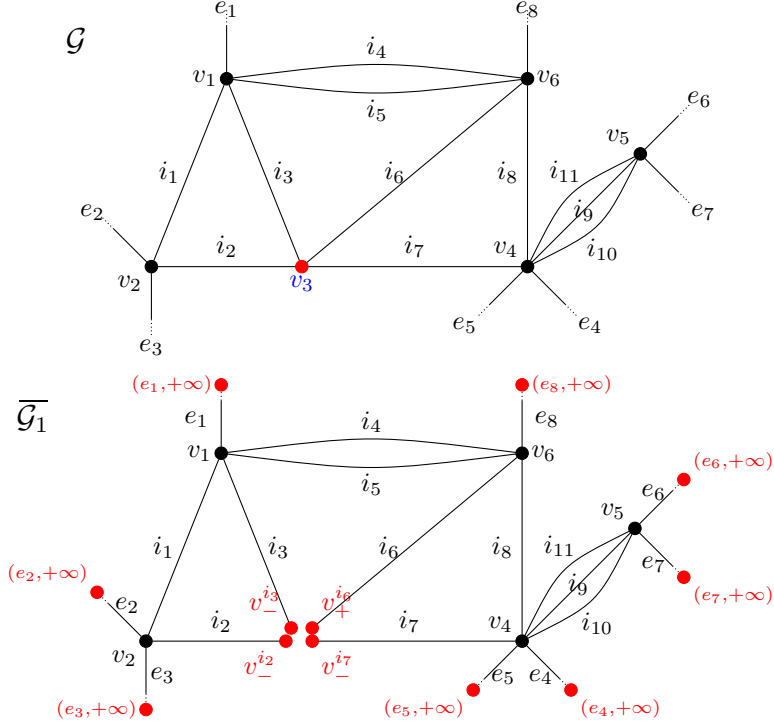

\begin{definition}
 $\cC(\overbar{\cG_1})$ is the set of all continuous functions on~$\overbar{\cG_1}$, that is, the set of all functions $f \colon \overbar{\cG_1} \rightarrow \R$
 which are continuous on $\overbar{\cG_1} \bs \{(e, +\infty), e \in \cE\}$ with respect to $d_1$ and for which
  \begin{align*}
   f \big( (e, +\infty) \big) = \lim_{x \rightarrow +\infty} f \big( (e,x) \big)
  \end{align*}
 exists for all $e \in \cE$. $\cC(\overbar{\cG_1})$ is endowed with its natural norm 
  \begin{align*}
   \norm{f}_\infty := \sup_{x \in \overbar{\cG_1}} \bigabs{f(x)}, \quad f \in \cC(\overbar{\cG_1}).
  \end{align*}
\end{definition}

%
%

The separability of $\cC(\overbar{\cG_1})$ will be essential in Theorem~\ref{theo:G_BM:Feller data}:

\begin{theorem} \label{theo:G_MG:cont functions on compactification separable}
 $\big( \cC(\overbar{\cG_1}), \norm{\,\cdot\,}_\infty \big)$ is separable.
\end{theorem}
\begin{proof}
  As $\big( \cC ([a,b]), \norm{\,\cdot\,}_\infty \big)$, for $a, b \in \R$, $a < b$, and $\big( \cC ([0, \infty]), \norm{\,\cdot\,}_\infty \big)$ are separable,
  we are able to approximate every continuous function on each separate edge~$l \in \cL$ of~$\overbar{\cG_1}$ 
  by functions in a respective separability set of~$\cC([0, \cR_l])$.
  By also approximating the values on the vertices~$v \in \cV$ by values in $\Q$, and 
  continuously connecting the approximations, e.g.\ by linear interpolation, we obtain 
  a countable, dense subset of~$\cC(\overbar{\cG_1})$.  
\end{proof}

\section{Strong Markov Processes \& Feller Processes} \label{app:Feller processes}

In the present article, a (strong) Markov process 
 \begin{align*}
  X = \big( \Omega, \sG, (\sG_t, t \geq 0), (X_t, t \geq 0), (\Theta_t, t \geq 0), (\PV_x, x \in E) \big)
 \end{align*}
on a measurable space $(E, \sE)$ is understood in the canonical sense of Dynkin~\cite{Dynkin65}, Blumenthal--Getoor~\cite{BlumenthalGetoor69} and Sharpe~\cite{Sharpe88}.
The reader may consult~\cite[Section~2.1]{WernerStar} for a short summary. In particular, we always assume right-continuity and normality.
As usual, the associated semigroup $(T_t, t \geq 0)$ and resolvent $(U_\a, \a > 0)$ are defined for all non-negative or bounded, $\sE$-measurable functions 
$f \in p\sE \cup b\sE$ by 
 \begin{align*}
  T_t f(x) := \EV_x \big( f(X_t) \big), \quad U_\a f(x) := \EV_x \Big( \int_0^\infty  e^{-\alpha t} \, f(X_t) \, dt \Big), \quad x \in E.
 \end{align*}
We will mainly be using following well-known localization techniques for both the resolvent and the generator:

Given a strong Markov process $X$, its resolvent can be decomposed at any stopping time~$\tau$ with the help of Dynkin's formula~\cite[Section~5.1]{Dynkin65}
 \begin{align}\label{eq:Dynkins formula (resolvent)}
   U_\alpha f(x) = \EV_x \Big( \int_0^\tau  e^{-\alpha t} \, f(X_t) \, dt \Big) + \EV_x \big( e^{-\alpha \tau} \, U_\alpha f(X_\tau) \, \1_{\{\tau < \infty\}} \big).
 \end{align}

We call a Markov process Feller process, if its semigroup is $\cC_0$-Feller,\footnote{For a locally compact space~$E$ with countable base,
   $\cC_0(E)$ is the set of all continuous functions which vanish at infinity. The space of all continuous and bounded functions on~$E$ is denoted by~$b\cC(E)$.} that is, if 
 \begin{enumerate}
  \item $T_t \cC_0(E) \subseteq \cC_0(E)$ for all $t \geq 0$, and   \label{itm:Feller (semigroup)}
  \item $\lim_{t \downarrow 0} T_t f(x) = f(x)$ for all $f \in \cC_0(E)$, $x \in E$. \label{itm:Feller (normality)}
 \end{enumerate}
Here, \ref{itm:Feller (normality)} is already implied by the assumed right continuity and normality of any Markov process, and yields
 \begin{align}\label{eq:Feller (resolvent convergence)}
   \lim_{\a \rightarrow \infty} \a U_\a f(x) = f(x) \quad \text{for all $f \in \cC_0(E)$, $x \in E$.}
 \end{align}   
Furthermore, condition~\ref{itm:Feller (semigroup)} can be equivalently replaced (cf.~\cite[Appendix~B]{KPS12}) 
by the corresponding condition of the resolvent, that is,
 \begin{align}\label{eq:Feller (resolvent)}
   U_\a \cC_0(E) \subseteq \cC_0(E) \quad \text{for all $\a > 0$}.
 \end{align}   

Every Feller process is a strong Markov process (cf.~\cite[Sections~III.8--III.9]{RogersWilliams1}), and, under the usual hypotheses,
a right process (cf.~\cite[Corollary~4.1.4]{MarcusRosen06}). 
It is uniquely determined by its semigroup restricted to $\cC_0(E)$ (see~\cite[Proposition~III.2.2]{RevuzYor94}),
and thus (cf.~\cite[Theorems~1.1, 1.2, 1.7]{Dynkin65}) by its resolvent on $\cC_0$ or equivalently by its weak $\cC_0$-generator
 \begin{align}
   A \colon \sD(A) \rightarrow \cC_0, \quad A f(x) := \lim_{t \downarrow 0} \frac{T_t f(x) - f(x)}{t},
 \end{align}
with its domain $\sD(A)$ being the set of all $f \in \cC_0$ for which the right-hand limit exists and constitutes a function in $\cC_0$. 

Dynkin's formula \cite[Theorem~5.2]{Dynkin65} gives the means to localize the generator:
Given a sequence $(\e_n, n \in \N)$ of positive numbers converging to $0$, let $\t_{\e_n}$ be the first exit time of $X$ from the closed ball $\overline{\BB_x(\e_n)}$.
If $0 < \EV_x(\t_{\e_n}) < +\infty$ for all $n \in \N$, then the generator of a Feller process can be computed by
  \begin{align}\label{eq:Dynkins formula (generator)}
   A f(x) 
   & = \lim_{n \rightarrow \infty} \frac{\EV_v \big( f\big(X(\t_{\e_n})\big) \big) - f(x)}{\EV_x(\t_{\e_n})}, \quad f \in \sD(A), \ x \in E,
  \end{align}

Dynkin's formula for the generator is not applicable for traps, that is, for points $x \in E$ satisfying
  \begin{align*}
   \PV_x \big( \forall t \geq 0 : X_t = x \big) = 1.
  \end{align*}
For all other points, there exists a sequence $(\e_n, n \in \N)$ which satisfies the requirements of the above formula
 (we follow~\cite[p.~53]{Knight81}):
 
\begin{lemma} \label{lem:A_FP:trap exit expectation}
 Let $X$ be a Feller process on a metric space $(E,d)$, $x \in E$, and consider the first exit times
   \begin{align*}
    \t_\e := \inf \big\{ t \geq 0: d(X_t, X_0) > \e \big\}, \quad  \e > 0.
   \end{align*}
 If $x$ is not a trap for $X$, then there exists $\d > 0$ such that
  \begin{align*}
   \forall \e \in (0,\d): \quad \EV_x(\t_\e) < +\infty.
  \end{align*}
\end{lemma}
\begin{proof}
 As $x$ is not a trap, there exists $\tilde{f} \in \sD(A)$ with $A \tilde{f}(x) \neq 0$ (see~\cite[pp.~135ff]{Dynkin65}).
 The domain~$\sD(A)$ is a linear subset of~$\cC_0(E)$, thus we can rescale $\tilde{f}$ to $f \in \sD(A)$ such that 
   \begin{align*}
    \exists \d > 0: \forall y \in \BB_\d(x): \quad Af(y) \geq 1.
   \end{align*}
 Let $\e \in (0, \d)$. For any $t \geq 0$ consider the stopping time $\t_\e \wedge t$. Then $\EV_x(\t_\e \wedge t) < +\infty$,
 and another one of Dynkin's formulas \cite[Corollary to~Theorem~5.1]{Dynkin65} yields
   \begin{align*}
    \EV_x \big( f(X_{\t_\e \wedge t}) - f(X_0) \big)
     = \EV_x \Big( \int_0^{\t_\e \wedge t} A f(X_s) \, dt \Big)
     \geq \EV_x(\t_\e \wedge t),
   \end{align*}
 as $\PV_x$-a.s., $X_s \in \overbar{\BB_\e(x)} \subseteq \BB_\d(x)$ holds for all $s < \t_\e$.
 Then, by Lebesgue's dominated convergence theorem,
  \begin{align*}
    \EV_x(\t_\e) & \leq \limsup_{t \rightarrow \infty} \EV_x(\t_\e \wedge t) \leq 2 \norm{f}_\infty < \infty. \qedhere
  \end{align*}
\end{proof}

\section{Dirichlet Brownian Motions on Intervals} 

In contrast to star graphs, general metric graphs typically possess internal edges. 
Therefore, we need to extend our findings of \cite[Appendix~A.1]{WernerStar} on Brownian motions stopped or killed when leaving half-lines
to the case of finite intervals. As usual, a standard Brownian motion $B$ on $\R$ is understood to be a 
continuous, strong Markov process on $\R$ with transition semigroup
  \begin{align*}
   T^B_t f(x) = \int_{\R} f(y) \, \frac{1}{\sqrt{2\pi t}} \, e^{-\frac{(y-x)^2}{2t}} \, dt, \quad x \in \R, ~ t \geq 0, ~ f \in b\sB(\R).
  \end{align*}
Its resolvent 
  \begin{align*}
    U^B_\a f(x)
    & = \int_{\R} \frac{1}{\sqrt{2 \a}} \, e^{-\sqrt{2 \a} \, \abs{y-x}} \, f(y) \, dy, \quad \a > 0,
  \end{align*}
maps $b\sB(\R)$ to $b\cC(\R)$ and $\cC_0(\R)$ to $\cC^2_0(\R)$.

In the following, let $\h_x := \inf \{ t \geq 0: B_t = x \}$ be the first time $B$ hits $x$.
 
\begin{example} \label{ex:B_HL:dirichlet BB on half line}
 Let ${(\abs{B_t}, t \geq 0)}$ be the reflecting Brownian motion on $\R_+$ with its first hitting time $\h_0$ of the origin,
 and consider the \textdef{killed Brownian motion}~$B^{[0,\infty)}$ on~$\R_+$ resulting from killing $\abs{B}$ at $\h_0$:
  \begin{align*}
    B^{[0,\infty)}_t :=
      \begin{cases}
        \abs{B_t}, & t < \h_0, \\
        \D, & t \geq \h_0.
      \end{cases}
  \end{align*}
 The resolvent $U^{[0,\infty)}$ of the killed process $B^{[0,\infty)}$ can directly be 
 computed with the help of Dynkin's formula~\eqref{eq:Dynkins formula (resolvent)} for the resolvent:
 \begin{align} \label{eq:B_HL:dirichlet resolvent}
  U^{[0,\infty)}_\a f(x)
  & = U^B_\a f(x) -  e^{-\sqrt{2 \a} x} \, U^B_\a f(0),
 \end{align}
  where the function $f$ in $U^B_\a f$ is an arbitrary continuation of $f \in b\sB( [0,\infty) )$ to~$b\sB(\R)$. 
 Thus, the stability properties of the Brownian resolvent $U^B$ give
 \begin{align*}
  U^{[0,\infty)} b\sB(\R_+) \subseteq b\cC(\R_+) \quad \text{and} \quad U^{[0,\infty)} \cC_0(\R_+) \subseteq \cC^2_0(\R_+).
 \end{align*}
\end{example}

\begin{example} \label{ex:B_IN:dirichlet BB on interval}
 Like in Example~\ref{ex:B_HL:dirichlet BB on half line} for the half line, we consider the Brownian motion on $[a,b]$ 
 killed when it reaches the boundary, that is, the process $B^{[a,b]}$ defined by
   \begin{align*}
    B^{[a,b]}_t :=
      \begin{cases}
        B_t, & t < \h_a \wedge \h_b, \\
        \D,  & t \geq \h_a \wedge \h_b.
      \end{cases}
  \end{align*}
 We compute its resolvent by using the decomposition of the standard Brownian motion~$B$ at~$\h_a \wedge \h_b$
 with the help of Dynkin's formula~\eqref{eq:Dynkins formula (resolvent)}. For all $f \in b\sB([a,b])$, $x \in [a,b]$, this gives
 (cf.~\cite[Section~1.7]{ItoMcKean74} for the passage time formulas)
 \begin{align*}
  & U^{[a,b]}_\a f(x) \\
  & = \EV_x \Big( \int_0^{\h_a \wedge \h_b} e^{-\a t} \, f(B_t) \, dt \Big) \\
  & = U^B_\a f(x) - \EV_x \big( e^{-\a \h_a} \,;\, \h_a < \h_b \big) \, U^B_\a f(a) - \EV_x \big( e^{-\a \h_b} \,;\, \h_b < \h_a \big) \, U^B_\a f(b) \\
  & = U^B_\a f(x) - \frac{\sinh \big( \sqrt{2\a} \, (b - x) \big)}{\sinh \big( \sqrt{2\a} \, (b - a) \big)} \, U^B_\a f(a) 
                  - \frac{\sinh \big( \sqrt{2\a} \, (x - a) \big)}{\sinh \big( \sqrt{2\a} \, (b - a) \big)} \, U^B_\a f(b),
 \end{align*}
 with the boundary values 
 \begin{align*}
  U^{[a,b]}_\a f(a) = 0, \quad U^{[a,b]}_\a f(b) = 0.
 \end{align*}
 As $U^B$ maps $\cC_0(\R)$ to $\cC^2_0(\R)$, $U^{[a,b]}$ maps $\cC([a,b])$ to $\cC^2([a,b])$. 
Differentiation of the above formula then yields, for all $x \in [a,b]$,
 \begin{align*}
   U^{[a,b]}_\a f''(x) 
  & = 2 \big( \a \, U^{[a,b]}_\a f(x) - f(x) \big). 
 \end{align*}
\end{example}

\end{appendix}

\section*{Acknowledgements}
The main parts of this paper were developed during the author's Ph.D.\ thesis~\cite{Werner16} supervised by Prof.~J\"urgen~Potthoff, whose constant support the 
author gratefully acknowledges.

\bibliographystyle{amsplain}
\bibliography{dissMG}

\providecommand{\bysame}{\leavevmode\hbox to3em{\hrulefill}\thinspace}
\providecommand{\MR}{\relax\ifhmode\unskip\space\fi MR }
\providecommand{\MRhref}[2]{%
  \href{http://www.ams.org/mathscinet-getitem?mr=#1}{#2}
}
\providecommand{\href}[2]{#2}
\begin{thebibliography}{10}

\bibitem{Bauer96}
Heinz Bauer and Robert~B. Burckel, \emph{Probability theory}, De Gruyter Stud.
  Math. 23, De Gruyter, Berlin, 1996.

\bibitem{Kuchment13}
Gregory Berkolaiko and Peter Kuchment, \emph{Introduction to quantum graphs},
  Math. Surveys Monogr., no. 186, American Mathematical Society, 2013.

\bibitem{Billingsley79}
Patrick Billingsley, \emph{Probability and measure}, Wiley Ser. Probab. Math.
  Stat., Wiley, New York, 1979.

\bibitem{BlumenthalGetoor69}
Robert~M. Blumenthal and Ronald~K. Getoor, \emph{{M}arkov processes and
  potential theory}, Pure Appl. Math. 29, Academic Press, New York, 1969.

\bibitem{Burago01}
Jurij~D. Burago, Dmitrij~J. Burago, and Sergej Ivanov, \emph{A course in metric
  geometry}, Grad. Stud. Math., no.~33, American Mathematical Society,
  Providence, 2001.

\bibitem{DellacherieMeyerA}
Claude Dellacherie and Paul-Andr{\'e} Meyer, \emph{Probabilities and potential
  [{A}]}, North Holland Math. Stud. 29, North-Holland, Amsterdam, 1978.

\bibitem{Dynkin65}
Evgenij~B. Dynkin, \emph{{M}arkov processes, vol. {I}}, Grundlagen Math. Wiss.
  121, Springer-Verlag, Berlin, 1965.

\bibitem{ItoMcKean63}
Kiyoshi It{\^{o}} and Henry~P. McKean, \emph{{B}rownian motions on a half
  line}, Illinois J. Math. \textbf{7} (1963), 181--231.

\bibitem{ItoMcKean74}
\bysame, \emph{Diffusion processes and their sample paths}, 2 ed., Grundlehren
  Math. Wiss. 125, Springer, 1974.

\bibitem{Knight81}
Frank~B. Knight, \emph{Essentials of {B}rownian motion and diffusion}, Math.
  Surveys Monogr. 18, American Mathematical Society, 1981.

\bibitem{KPS12}
Vadim {Kostrykin}, J\"urgen {Potthoff}, and Robert {Schrader}, \emph{{B}rownian
  motions on metric graphs}, J. Math. Phys. \textbf{53} (2012), 095206.

\bibitem{KPS_Walsh12A}
\bysame, \emph{Construction of the paths of {B}rownian motions on star graphs
  {I}}, Commun. Stoch. Anal. \textbf{6} (2012), no.~2, 223--245.

\bibitem{KPS_Walsh12B}
\bysame, \emph{Construction of the paths of {B}rownian motions on star graphs
  {II}}, Commun. Stoch. Anal. \textbf{6} (2012), no.~2, 247--261.

\bibitem{KostrykinSchrader06}
Vadim {Kostrykin} and Robert {Schrader}, \emph{Laplacians on metric graphs:
  Eigenvalues, resolvents and semigroups}, Quantum Graphs and Their
  Applications (G.~Berkolaiko, R.~Carlson, S.~A. Fulling, and P.~Kuchment,
  eds.), Contemp. Math. 415, Amer. Math. Soc., 2006, pp.~201--225.

\bibitem{MarcusRosen06}
Michael~B. Marcus and Jay Rosen, \emph{{M}arkov processes, {G}aussian
  processes, and local times}, Cambridge Stud. Adv. Math., no. 100, Cambridge
  Univ. Press, Cambridge, 2006.

\bibitem{RevuzYor94}
Daniel Revuz and Marc Yor, \emph{Continuous martingales and {B}rownian motion},
  2 ed., Grundlagen Math. Wiss. 293, Springer, Berlin, 1994.

\bibitem{RogersWilliams1}
L.~Chris~G. Rogers and David Williams, \emph{Diffusions, {M}arkov processes and
  martingales. 1. foundations}, 2 ed., Cambridge University Press, 2000.

\bibitem{Sharpe88}
Michael Sharpe, \emph{General theory of {M}arkov processes}, Pure Appl. Math.
  133, Academic Press, 1988.

\bibitem{Werner16}
Florian Werner, \emph{{B}rownian motions on metric graphs}, Ph.D. thesis,
  University of Mannheim, 2016.

\bibitem{WernerStar}
\bysame, \emph{{B}rownian motions on star graphs with non-local boundary
  conditions}, ArXiv e-prints (2018), 1803.07027.

\bibitem{WernerConcat}
\bysame, \emph{Concatenation and pasting of right processes}, ArXiv e-prints
  (2018), 1801.02595.

\end{thebibliography}

\end{document}